\documentclass[oneside,english]{amsart}
\usepackage[T1]{fontenc}
\usepackage[latin9]{inputenc}
\usepackage{babel}
\usepackage{mathrsfs}
\usepackage{amsthm}
\usepackage{amssymb}
\usepackage[all]{xy}
\usepackage[unicode=true,pdfusetitle,
 bookmarks=true,bookmarksnumbered=false,bookmarksopen=false,
 breaklinks=false,pdfborder={0 0 1},backref=false,colorlinks=false]
 {hyperref}

\makeatletter
\numberwithin{equation}{section}
\numberwithin{figure}{section}
\theoremstyle{plain}
\newtheorem{thm}{\protect\theoremname}
  \theoremstyle{remark}
  \newtheorem{rem}[thm]{\protect\remarkname}
  \theoremstyle{plain}
  \newtheorem{lem}[thm]{\protect\lemmaname}
  \theoremstyle{definition}
  \newtheorem{defn}[thm]{\protect\definitionname}
  \theoremstyle{plain}
  \newtheorem{prop}[thm]{\protect\propositionname}
  \theoremstyle{definition}
  \newtheorem{example}[thm]{\protect\examplename}
  \theoremstyle{plain}
  \newtheorem{cor}[thm]{\protect\corollaryname}

\newcommand{\xyR}[1]{\xydef@\xymatrixrowsep@{#1}}
\newcommand{\xyC}[1]{\xydef@\xymatrixcolsep@{#1}}

\newcommand{\Rep}{\mathsf{Rep}}
\newcommand{\Coh}{\mathsf{Coh}}
\newcommand{\Mod}{\mathsf{Mod}}
\newcommand{\Vect}{\mathsf{Vect}}

\newcommand{\Modif}{\mathsf{Modif}}
\newcommand{\HT}{\mathsf{HT}}
\newcommand{\Norm}{\mathsf{Norm}}

\newcommand{\Bun}{\mathsf{Bun}}
\newcommand{\Gra}{\mathsf{Gr}}
\newcommand{\Fil}{\mathsf{Fil}}

\newcommand{\inv}{\mathrm{inv}}
\newcommand{\Spec}{\mathrm{Spec}}
\newcommand{\Spa}{\mathrm{Spa}}
\newcommand{\Hom}{\mathrm{Hom}}
\newcommand{\Ext}{\mathrm{Ext}}
\newcommand{\coker}{\mathrm{coker}}
\newcommand{\length}{\mathrm{length}}
\newcommand{\Tor}{\mathrm{Tor}}
\newcommand{\rank}{\mathrm{rank}}
\newcommand{\sk}{\mathrm{sk}}
\newcommand{\Gr}{\mathrm{Gr}}
\newcommand{\Gal}{\mathrm{Gal}}

\newcommand{\Sub}{\mathrm{Sub}}
\newcommand{\Sym}{\mathrm{Sym}}
\newcommand{\Jump}{\mathrm{Jump}}
\newcommand{\HTT}{\mathrm{HT}}
\newcommand{\loc}{\mathrm{loc}}
\newcommand{\Frac}{\mathrm{Frac}}
\newcommand{\Aut}{\mathrm{Aut}}
\newcommand{\Proj}{\mathrm{Proj}}
\newcommand{\Exp}{\mathrm{Exp}}
\newcommand{\Log}{\mathrm{Log}}

\newcommand{\Pic}{\mathrm{Pic}}
\newcommand{\eqd}{\stackrel{\mathrm{def}}{=}}

\renewcommand{\C}{\mathsf{C}}

\makeatother

  \providecommand{\corollaryname}{Corollary}
  \providecommand{\definitionname}{Definition}
  \providecommand{\examplename}{Example}
  \providecommand{\lemmaname}{Lemma}
  \providecommand{\propositionname}{Proposition}
  \providecommand{\remarkname}{Remark}
\providecommand{\theoremname}{Theorem}

\begin{document}

\title{Harder-Narasimhan Filtrations for Breuil-Kisin-Fargues modules}

\author{Christophe Cornut and Macarena Peche Irissarry. }
\begin{abstract}
We define and study Harder-Narasimhan filtrations on Breuil-Kisin-Fargues
modules and related objects relevant to $p$-adic Hodge theory. 
\end{abstract}


\date{January, 2018}

\maketitle

\section{Introduction}

\subsection{Context}

Cohomology theories provide classifying functors from categories of
algebraic varieties to various realisation categories. Grothendieck
conjectured that there is a universal such functor, and thus also
a universal realisation category, which he called the category of
motives. He also worked out an elementary bottom-up construction of
this universal functor and its target category, assuming a short list
of hard conjectures -- the so-called standard conjectures, on which
little progress has been made. A top-down approach to Grothendieck's
conjecture aims to cut down the elusive category of motives from the
various realisation categories of existing cohomology theories, and
this first requires assembling them in some ways. 

Over an algebraically closed complete extension $C$ of $\mathbb{Q}_{p}$,
Bhatt, Morrow and Scholze \cite{BaMoSc16} have recently defined a
new (integral) $p$-adic cohomology theory, which specializes to all
other known such theories and nicely explains their relations and
pathologies. It takes values in the category of Breuil-Kisin-Fargues
modules (hereafter named BKF-modules), a variant of Breuil-Kisin modules
due to Fargues~\cite{Fa15}. This new realisation category has various,
surprisingly different but nevertheless equivalent incarnations, see
\cite[14.1.1]{ScWe15},~\cite[7.5]{Sc17} or section~$3$; beyond
its obvious relevance for $p$-adic motives, it is also expected to
play a role in the reformulation of the $p$-adic Langlands program
proposed by Fargues \cite{Fa16}. 

In this paper, we mostly investigate an hidden but implicit structure
of these BKF-modules: they are equiped with some sort of Harder-Narasimhan
formalism, adapted from either \cite{Imp16} or \cite{LeWE16}, which
both expanded the original constructions of Fargues \cite{Fa12} from
$p$-divisible groups over $\mathcal{O}_{C}$ to Breuil-Kisin modules.

\subsection{Overview}

In section~\ref{sec:BKF-Modules}, we define our categories of BKF-modules,
review what Barghav, Morrow and Scholze had to say about them, exhibit
the HN-filtrations (which we call Fargues filtrations) and work out
their basic properties. In section~\ref{sec:FunctorsOfFargues},
we turn our attention to the curvy avatar of BKF-modules up to isogenies,
namely admissible modifications of vector bundles on the curve, and
to their Hodge-Tate realizations. The link between all three incarnations
of sthukas with one paw was established by Fargues, according to Scholze
who sketched a proof in his lectures at Berkeley%
\footnote{Between \cite{ScWe15} and \cite{BaMoSc16}, the paw was twisted from
$A\xi$ to $A\xi'$. We follow the latter convention. No sthukas were
harmed in the making of \emph{our} paper, but our valuations have
lame normalizations. %
}. We redo Scholze's proof in slow motion and investigate the Fargues
filtration on the curvy and Hodge-Tate side, where it tends to be
more tractable. We also clarify various issues pertaining to exactness,
and introduce some full subcategories where the Fargues filtration
is particularly well-behaved. In a subsequent work, we will show that
ordinary BKF-modules with $G$-structures factor through these subcategories
and compute the corresponding reduction maps, from lattices in the
étale realization to lattices in the crystalline realization.

\subsection{Results}

We refer to the main body of the paper for all notations. 

We define Fargues filtrations $\mathcal{F}_{F}$ and their types $t_{F}$
on $\Mod_{A,t}^{\varphi}$ (\ref{sub:FarguesOnModAt}), $\Mod_{A,f}^{\varphi,\ast}$
(\ref{sub:TypeHN}), $\Mod_{\mathcal{O}_{K}^{\flat},f}^{\varphi}$
(\ref{sub:RisOkflat}), $\Modif_{X}^{ad}$ and $\HT_{E}^{B_{dR}}$
(\ref{sub:DefFarguesFiltrOnModif} and \ref{sub:HodgeTateModDef}).
We show that they are compatible with $\otimes$-product constructions
on $\Mod_{A,f}^{\varphi,\ast}$ (prop.~\ref{prop:catOfHNTypeTensor}),
$\Mod_{\mathcal{O}_{K}^{\flat}}^{\varphi}$ (prop.~\ref{prop:FarguesonO_Kflattensor}),
$\Modif_{X}^{ad}$ and $\HT_{E}^{B_{dR}}$ (prop.~\ref{prop:OnHT+AdModif_Farguestensor}).
On the isogeny category $\Mod_{A,\ast}^{\varphi}\otimes E$, we only
define a type $t_{F,\infty}$ (\ref{sub:FarguesTypeBKF}), analogous
to Fargues's renormalized Harder-Narasimhan function in~\cite{Fa12}.
This type matches the Fargues type $t_{F}$ on $\Mod_{A,f}^{\varphi,\ast}\otimes E$
(prop.~\ref{prop:TypeHN}), and proposition~\ref{prop:SpecialSubCats}
compares it with the Fargues type $t_{F}$ on $\Modif_{X}^{ad}$ and
$\HT_{E}^{B_{dR}}$. We define Hodge filtrations $\mathcal{F}_{H}$
and their types $t_{H}$ on $\Mod_{\mathcal{O}_{K}^{\flat},f}^{\varphi}$
(\ref{sub:RisOkflat}), $\Mod_{\mathcal{O}_{L},f}^{\varphi}$ (\ref{sub:RisOL}),
$\Mod_{A[\frac{1}{\pi}],f}^{\varphi}$ and $\Mod_{A,\ast}^{\varphi}\otimes E$
(\ref{sub:RisA=00005B1/pi=00005D}), $\Modif_{X}^{ad}$ (\ref{sub:AdmModifDef})
and $\HT_{E}^{B_{dR}}$ (\ref{sub:HodgeTateModDef}). We define opposed
Newton filtrations $\mathcal{F}_{N}$ and $\mathcal{F}_{N}^{\iota}$
and their types $t_{N}$ and $t_{N}^{\iota}$ on $\Mod_{L,f}^{\varphi}$
(\ref{sub:RisL}), and a Newton (or slope) filtration $\mathcal{F}_{N}$
with type $t_{N}$ on $\Bun_{X}$ (\ref{sub:Newton4BundlesOnCurve})
and $\Modif_{X}^{ad}$ (\ref{sub:AdmModifDef}). The Hodge and Newton
filtrations are compatible with $\otimes$-product constructions and
satisfy some exactness properties. If $K=C$ is algebraically closed,
then for a finite free BKF-module $M\in\Mod_{A,f}^{\varphi}$ mapping
to the admissible modification $\underline{\mathcal{E}}\in\Modif_{X}^{ad}$,
we establish the following inequalities:
\[
\xyR{1pc}\xyC{1.5pc}\xymatrix{t_{H}(M\otimes E)\ar@{}[r]|{\geq}\ar@{=}[d] & t_{H}(M\otimes\mathcal{O}_{K}^{\flat})\ar@{}[rr]|{\geq} &  & t_{F}(M\otimes\mathcal{O}_{K}^{\flat})\ar@{}[r]|{\geq} & t_{F,\infty}(M)\ar@{:}[dd]_{?}^{?}\\
t_{H}(M\otimes E)\ar@{}[r]|{\geq}\ar@{=}[d] & t_{H}(M\otimes\mathcal{O}_{L})\ar@{}[r]|{\geq} & t_{N}^{\iota}(M\otimes L)\ar@{=}[d]\\
t_{H}(\underline{\mathcal{E}})\ar@{}[rr]|{\geq} &  & t_{N}(\underline{\mathcal{E}})\ar@{}[rr]|{\geq} &  & t_{F}(\underline{\mathcal{E}})
}
\]
We failed to establish our hope that $t_{F,\infty}(M)=t_{F}(\underline{\mathcal{E}})$
(as did Fargues for $p$-divisible groups in \cite[Théorème 20]{Fa12}
and the second named author for Breuil-Kisin modules in \cite[Proposition 3.11]{Imp16}),
but we nevertheless show in proposition~\ref{prop:SpecialSubCats}
that
\[
\begin{array}{rcll}
t_{F,\infty}(M) & \leq & t_{F}(\underline{\mathcal{E}}) & \mbox{if }M\in\Mod_{A,f}^{\varphi,\ast},\\
t_{F,\infty}(M) & \geq & t_{F}(\underline{\mathcal{E}}) & \mbox{if }\underline{\mathcal{E}}\in\Modif_{X}^{ad,\ast}.
\end{array}
\]
We also investigate sufficient conditions for the equality $\mathcal{F}_{F}(\underline{\mathcal{E}})=\mathcal{F}_{N}(\underline{\mathcal{E}})$. 
\begin{rem}
The definition of the full subcategory $\Mod_{A,f}^{\varphi,\ast}$
of $\Mod_{A,f}^{\varphi}$ is inspired by the notion of $p$-divisible
groups of HN-type, due to Fargues, and expanded to Breuil-Kisin modules
in~\cite{Imp16}. The definition of the full subcategory $\Modif_{X}^{ad,\ast}$
of $\Modif_{X}^{ad}$ is new to this paper. We do not know if these
subcategories are related under Fargues's equivalence $\underline{\mathcal{E}}:\Mod_{A,f}^{\varphi}\otimes E\stackrel{\simeq}{\longrightarrow}\Modif_{X}^{ad}$
(see theorem~\ref{thm:FarguesScholze}).
\end{rem}

\subsection{Thanks}

First and foremost, Laurent Fargues, obviously. Then also Matthew
Morrow. And Jared Weinstein for his notes, Peter Scholze for his talks.

\subsection{Notations}

\subsubsection{Types\label{sub:Types}}

Let $(\Gamma,+,\leq)$ be a totally ordered commutative group. For
$r\in\mathbb{N}$, we consider the following submonoid of $\Gamma^{r}$:
\[
\Gamma_{\geq}^{r}\eqd\left\{ (\gamma_{1},\cdots,\gamma_{r})\in\Gamma^{r}:\gamma_{1}\geq\cdots\geq\gamma_{r}\right\} .
\]
It is equipped with a partial order defined by 
\begin{multline*}
(\gamma_{1},\cdots,\gamma_{r})\leq(\gamma_{1}^{\prime},\cdots,\gamma_{r}^{\prime})\iff\begin{cases}
\forall s\in\{1,\cdots,r\} & \gamma_{1}+\cdots+\gamma_{s}\leq\gamma_{1}^{\prime}+\cdots+\gamma_{s}^{\prime},\\
\mbox{and} & \gamma_{1}+\cdots+\gamma_{r}=\gamma_{1}^{\prime}+\cdots+\gamma_{r}^{\prime},
\end{cases}
\end{multline*}
with an involution $\iota:\Gamma_{\geq}^{r}\rightarrow\Gamma_{\geq}^{r}$
and functions $\deg,\max,\min:\Gamma_{\geq}^{r}\rightarrow\Gamma$
defined by
\[
(\gamma_{1},\cdots,\gamma_{r})^{\iota}\eqd(-\gamma_{r},\cdots,-\gamma_{1}),\qquad\begin{array}{rcl}
\deg(\gamma_{1},\cdots,\gamma_{r}) & \eqd & \gamma_{1}+\cdots+\gamma_{r},\\
(\gamma_{1},\cdots,\gamma_{r})^{\max} & \eqd & \gamma_{1},\\
(\gamma_{1},\cdots,\gamma_{r})^{\min} & \eqd & \gamma_{r}.
\end{array}
\]
For $r_{1},r_{2}\in\mathbb{N}$, there is also a ``convex sum''
map
\[
\ast:\Gamma_{\geq}^{r_{1}}\times\Gamma_{\geq}^{r_{2}}\rightarrow\Gamma_{\geq}^{r_{1}+r_{2}}
\]
which concatenates and reorders the elements. We set $\Gamma_{+}:=\{\gamma\in\Gamma:\gamma\geq0\}$
and 
\[
\Gamma_{+,\geq}^{\infty}\eqd\underrightarrow{\lim}\,\Gamma_{+,\geq}^{r}=\left\{ (\gamma_{i})_{i=1}^{\infty}\left|\begin{array}{ll}
\forall i\geq1 & \gamma_{i}\geq\gamma_{i+1}\geq0\\
\forall i\gg1 & \gamma_{i}=0
\end{array}\right.\right\} 
\]
where $\Gamma_{+,\geq}^{r}:=\Gamma_{\geq}^{r}\cap\Gamma_{+}^{r}$,
with the transition morphisms $\Gamma_{+,\geq}^{r}\hookrightarrow\Gamma_{+,\geq}^{r+1}$
given by $(\gamma_{1},\cdots,\gamma_{r})\mapsto(\gamma_{1},\cdots,\gamma_{r},0)$.
Thus $\Gamma_{+,\geq}^{\infty}$ is yet another partially ordered
monoid equipped with a degree function $\deg:\Gamma_{+,\geq}^{\infty}\rightarrow\Gamma_{+}$
and a ``convex sum'' operator 
\[
\ast:\Gamma_{+,\geq}^{\infty}\times\Gamma_{+,\geq}^{\infty}\rightarrow\Gamma_{+,\geq}^{\infty}.
\]
If $\Gamma\subset\mathbb{R}$, we will often identify $\Gamma_{\geq}^{r}$
with the monoid of all continuous concave functions $f:[0,r]\rightarrow\mathbb{R}$
such that $f(0)=0$ and $f$ is affine of slope $\gamma_{i}\in\Gamma$
on $[i-1,i]$ for all $i\in\{1,\cdots,r\}$. Under this identification,
$f\leq f'$ if and only if $f(s)\leq f'(s)$ for all $s\in[0,r]$
with equality for $s=r$, $f^{\iota}(s)=f(r-s)-f(r)$ for all $s\in[0,r]$,
$\deg(f)=f(r)$ and finally for $f_{1}\in\Gamma_{\geq}^{r_{1}}$,
$f_{2}\in\Gamma_{\geq}^{r_{2}}$ and $s\in[0,r_{1}+r_{2}]$, 
\[
f_{1}\ast f_{2}(s)=\max\left\{ f_{1}(s_{1})+f_{2}(s_{2})\left|\begin{array}{l}
s_{1}\in[0,r_{1}],\, s_{2}\in[0,r_{2}]\\
\mbox{and }s=s_{1}+s_{2}
\end{array}\right.\right\} .
\]
Similarly, we will identify $\Gamma_{+,\geq}^{\infty}$ with the monoid
of all continuous concave functions $f:\mathbb{R}_{+}\rightarrow\mathbb{R}_{+}$
such that $f(0)=0$, $f$ is affine of slope $\gamma_{i}\in\Gamma_{+}$
on $[i-1,i]$ for all positive integer $i$, with $\gamma_{i}=0$
for $i\gg0$. Then $f\leq f'$ if and only if $f(s)\leq f'(s)$ for
all $s\in\mathbb{R}_{+}$ with equality for $s\gg0$, $\deg(f)=f(s)$
for $s\gg0$ and 
\[
f_{1}\ast f_{2}(s)=\max\left\{ f_{1}(t)+f_{2}(s-t)\vert t\in[0,s]\right\} .
\]

\subsubsection{Filtrations\label{sub:Filtrations}}

In~\cite{Co14}, we defined a notion of $\Gamma$-filtrations for
finite free quasi-coherent sheaves (aka vector bundles) over schemes,
and in \cite{Co16} we investigated a notion of $\mathbb{R}$-filtrations
on bounded modular lattices of finite length. Here is a common simple
framework for $\Gamma$-filtrations and their types. If $(X,\leq)$
is a bounded partially ordered set with smallest element $0_{X}$
and largest element $1_{X}$, then a $\Gamma$-filtration on $X$
is a function $\mathcal{F}:\Gamma\rightarrow X$ which is non-increasing,
exhaustive, separated and left-continuous: $\mathcal{F}(\gamma_{1})\geq\mathcal{F}(\gamma_{2})$
for $\gamma_{1}\leq\gamma_{2}$, $\mathcal{F}(\gamma)=1_{X}$ for
$\gamma\ll0$, $\mathcal{F}(\gamma)=0_{X}$ for $\gamma\gg0$ and
for every $\gamma\in\Gamma$, there is a $\gamma'<\gamma$ such that
$\mathcal{F}$ is constant on $]\gamma',\gamma]:=\{\eta\in\Gamma\vert\gamma'<\eta\leq\gamma\}$.
If all chains of $X$ are finite, the formula
\[
\mathcal{F}(\gamma)=\begin{cases}
0_{X} & \mbox{for }\gamma>\gamma_{1}\\
c_{i} & \mbox{for }\gamma_{i+1}<\gamma\leq\gamma_{i}\\
1_{X} & \mbox{for }\gamma\leq\gamma_{s}
\end{cases}
\]
yields a bijection between the set $\mathbf{F}^{\Gamma}(X)$ of all
$\Gamma$-filtrations on $X$ and the set of all pairs $(c_{\bullet},\gamma_{\bullet})$
where $c_{\bullet}=\mathcal{F}(\Gamma)=(c_{0}<\cdots<c_{s})$ is a
(finite) chain of length $s$ in $X$ with $c_{0}=0_{X}$ and $c_{s}=1_{X}$,
while $\gamma_{\bullet}=\Jump(\mathcal{F})=(\gamma_{1}>\cdots>\gamma_{s})$
is a decreasing sequence in $\Gamma$. We then set $\mathcal{F}_{+}(\gamma):=\max\left\{ \mathcal{F}(\eta):\eta>\gamma\right\} $.
If $\rank:X\rightarrow\mathbb{N}$ is an increasing function and $r=\rank(1_{X})$,
then all chains of $X$ are finite of length $s\leq r$ and any $\Gamma$-filtration
$\mathcal{F}\in\mathbf{F}^{\gamma}(X)$ as a well-defined type $\mathbf{t}(\mathcal{F})\in\Gamma_{\geq}^{r}$:
for any $\gamma\in\Gamma$, the multiplicity of $\gamma$ in $\mathbf{t}(\mathcal{F})$
is equal to $\rank(\mathcal{F}(\gamma))-\rank(\mathcal{F}_{+}(\gamma))$.

If $\C$ is an essentially small quasi-abelian category equipped with
a rank function $\rank:\sk\,\C\rightarrow\mathbb{N}$, as defined
in \cite[3.1]{Co16}, then for every object $X$ of $\C$, the partially
ordered set $\Sub(X)$ of all strict subobjects of $X$ is a bounded
modular lattice of finite length. A $\Gamma$-filtration on $X$ is
then a $\Gamma$-filtration on $\Sub(X)$, and we denote by $\mathbf{F}^{\Gamma}(X)$
the set of all $\Gamma$-filtrations on $X$. For $\mathcal{F}\in\mathbf{F}^{\Gamma}(X)$,
we typically write 
\[
\mathcal{F}^{\gamma}=\mathcal{F}^{\geq\gamma}=\mathcal{F}(\gamma),\quad\mathcal{F}_{+}^{\gamma}=\mathcal{F}^{>\gamma}=\mathcal{F}_{+}(\gamma)\quad\mbox{and}\quad\Gr_{\mathcal{F}}^{\gamma}=\mathcal{F}^{\gamma}/\mathcal{F}_{+}^{\gamma}.
\]
If $r=\rank(X)$, the type map $\mathbf{t}:\mathbf{F}^{\Gamma}(X)\rightarrow\Gamma_{\geq}^{r}$
is given by 
\[
\mathbf{t}(\mathcal{F})=(\gamma_{1}\geq\cdots\geq\gamma_{r})\iff\forall\gamma\in\Gamma:\quad\rank\,\Gr_{\mathcal{F}}^{\gamma}=\#\{i:\gamma_{i}=\gamma\}
\]
and the degree map $\deg:\mathbf{F}^{\Gamma}(X)\rightarrow\Gamma$
is given by 
\[
\deg(\mathcal{F})=\deg(\mathbf{t}(\mathcal{F}))={\textstyle \sum_{\gamma\in\Gamma}}\rank\,\Gr_{\mathcal{F}}^{\gamma}\cdot\gamma.
\]
If $0\rightarrow x\rightarrow X\rightarrow y\rightarrow0$ is an exact
sequence in $\C$, any $\Gamma$-filtration $\mathcal{F}\in\mathbf{F}^{\Gamma}(X)$
induces $\Gamma$-filtrations $\mathcal{F}_{x}\in\mathbf{F}^{\Gamma}(x)$
and $\mathcal{F}_{y}\in\mathbf{F}^{\Gamma}(y)$, and we have 
\[
\mathbf{t}(\mathcal{F})=\mathbf{t}(\mathcal{F}_{x})\ast\mathbf{t}(\mathcal{F}_{y})\quad\mbox{in}\quad\Gamma_{\geq}^{r}.
\]
We denote by $\Gra^{\Gamma}\C$ and $\Fil^{\Gamma}\C$ the quasi-abelian
categories of $\Gamma$-graded and $\Gamma$-filtered objects in $\C$.
For finite dimensional vector spaces over a field $k$, we set
\[
\Gr_{k}^{\Gamma}\eqd\Gr^{\Gamma}\Vect_{k}\quad\mbox{and}\quad\Fil_{k}^{\Gamma}\eqd\Fil^{\Gamma}\Vect_{k}.
\]
 When $\Gamma=\mathbb{R}$, we simplify our notations to $\mathbf{F}(X):=\mathbf{F}^{\mathbb{R}}(X)$.

\subsubsection{Invariants\label{sub:Invariants}}

Let $\mathcal{O}$ be a valuation ring with fraction field $K$, maximal
ideal $\mathfrak{m}$ and residue field $k$. We denote by $(\Gamma,+,\leq)$
the totally ordered commutative group $(K^{\times}/\mathcal{O}^{\times},\cdot,\leq)$,
when we want to view it as an additive group. We extend the total
orders to $K/\mathcal{O}^{\times}=K^{\times}/\mathcal{O}^{\times}\cup\{0\}$
and $\Gamma\cup\{-\infty\}$, by declaring that the added elements
are smaller than everyone else. We denote by $\left|-\right|:K\rightarrow K/\mathcal{O}^{\times}$
the projection. Thus for every $\lambda_{1},\lambda_{2}\in K$, $\left|\lambda_{1}\right|\leq\left|\lambda_{2}\right|\iff\mathcal{O}\lambda_{1}\subset\mathcal{O}\lambda_{2}$.
We write 
\[
\Exp:\Gamma\cup\{-\infty\}\longleftrightarrow K^{\times}/\mathcal{O}^{\times}\cup\{0\}:\Log
\]
for the corresponding isomorphisms. When the valuation has height
$1$, i.e.~when it is given by a genuine absolute value $\left|-\right|:K\rightarrow\mathbb{R}_{+}$,
we will identify $K^{\times}/\mathcal{O}^{\times}$ with the corresponding
subgroup $\left|K^{\times}\right|\subset\mathbb{R}_{+}^{\times}$,
and $\Gamma$ with a subgroup of $\mathbb{R}$, using genuine logarithms
and exponential maps in a suitable base. For every element $\gamma\in\Gamma$,
\[
I(\gamma)\eqd\left\{ x\in K:\left|x\right|\leq\Exp(-\gamma)\right\} 
\]
is a free, rank one $\mathcal{O}$-submodule of $K$. If $\gamma\in\Gamma_{+}$,
it is a principal ideal of $\mathcal{O}$ and 
\[
\mathcal{O}(\gamma)\eqd\mathcal{O}/I(\gamma)
\]
is a finitely presented torsion $\mathcal{O}$-module. These modules
are the building blocks of the category of finitely presented torsion
$\mathcal{O}$-modules, which we denote by $\C$. 
\begin{lem}
\label{lem:DefInvTorsionOmod}For any $M\in\C$, there is a unique
element $(\gamma_{i})_{i=1}^{\infty}$ in $\Gamma_{+,\geq}^{\infty}$
such that $M\simeq\oplus_{i=1}^{\infty}\mathcal{O}(\gamma_{i})$.
Then $I(\sum_{i=1}^{\infty}\gamma_{i})$ is the Fitting ideal of $M$,
$I(\gamma_{i})$ is the annihilator of $\Lambda_{\mathcal{O}}^{i}(M)$
and $\max\{i:\gamma_{i}\neq0\}$ is the minimal number of generators
of $M$. \end{lem}
\begin{proof}
By~\cite[6.1.14]{GaRa03}, $M\simeq\oplus_{i=1}^{r}\mathcal{O}(\gamma_{i})$
for some $r\in\mathbb{N}$ and $\gamma_{1}\geq\cdots\geq\gamma_{r}>0$.
Plainly, $I(\gamma_{1})$ is the annihilator of $M$, $I(\sum_{i=1}^{r}\gamma_{i})$
is the Fitting ideal of $M$ and $r=\dim_{k}M\otimes_{\mathcal{O}}k$
is the minimal number of generators of $M$. For every $i\geq1$,
$\Lambda_{\mathcal{O}}^{i}M\simeq\oplus_{I}\mathcal{O}(\gamma_{I})$
where $I$ ranges through the subsets of $\{1,\cdots,r\}$ with $i$
elements and $\gamma_{I}:=\min\{\gamma_{i}:i\in I\}$, thus indeed
$I(\gamma_{i})$ is the annihilator of $\Lambda_{\mathcal{O}}^{i}M$. \end{proof}
\begin{defn}
We denote the above invariant by $\inv(M)=(\inv_{i}(M))_{i=1}^{\infty}$
and set 
\[
r(M)\eqd\max\{i:\inv_{i}(M)\neq0\},\quad\length(M)\eqd\sum_{i=1}^{\infty}\inv_{i}(M).
\]
\end{defn}
\begin{lem}
\label{lem:InvTorsOModdecUnderSubquo}Fix $M,N\in\C$ and suppose
that $N$ is a subquotient of $M$. Then 
\[
r(N)\leq r(M)\quad\mbox{and}\quad\forall i:\quad\inv_{i}(N)\leq\inv_{i}(M).
\]
\end{lem}
\begin{proof}
We just need to establish the second claim when $N$ is either a submodule
or a quotient of $M$. For $X\in\C$, set $X^{\vee}:=\Hom_{\mathcal{O}}(X,K/\mathcal{O})$.
One checks using the previous lemma that this defines an exact duality
on $\C$, with $\inv(X)=\inv(X^{\vee})$. We may thus even assume
that $N$ is a quotient of $M$. Our claim now follows from the previous
lemma and the surjectivity of $\Lambda_{\mathcal{O}}^{i}M\twoheadrightarrow\Lambda_{\mathcal{O}}^{i}N$.\end{proof}
\begin{lem}
\label{lem:InvTorsOMod_Formula}For $M\in\C$ and any positive integer
$r$, 
\[
\sum_{i=1}^{r}\inv_{i}(M)=\max\left\{ \length\left\langle x_{1},\cdots,x_{r}\right\rangle :x_{i}\in M\right\} .
\]
\end{lem}
\begin{proof}
It is plainly sufficient to establish that for every submodule $N$
of $M$ generated by $r$ elements, $\length(N)\leq\sum_{i=1}^{r}\inv_{i}(M)$.
Now $r(N)\leq r$ by lemma~\ref{lem:DefInvTorsionOmod}, thus indeed
$\length(N)=\sum_{i=1}^{r}\inv_{i}(N)\leq\sum_{i=1}^{r}\inv_{i}(M)$
by lemma~\ref{lem:InvTorsOModdecUnderSubquo}. \end{proof}
\begin{lem}
\label{lem:InvTorsOModAndExactSeq}Let $0\rightarrow M_{1}\rightarrow M_{2}\rightarrow M_{3}\rightarrow0$
be an exact sequence of $\mathcal{O}$-modules. Suppose that two out
of $\{M_{1},M_{2},M_{3}\}$ belong to $\C.$ Then so does the third
one and 
\[
\length(M_{1})+\length(M_{3})=\length(M_{3})\quad\mbox{in}\quad\Gamma_{+},
\]
\[
\inv(M_{1})\ast\inv(M_{3})\leq\inv(M_{2})\le\inv(M_{1})+\inv(M_{3})\quad\mbox{in}\quad\Gamma_{+,\geq}^{\infty}.
\]
Moreover, $\inv(M_{1})\ast\inv(M_{3})=\inv(M_{3})$ if and only if
the exact sequence splits.\end{lem}
\begin{proof}
The first assertion holds for any coherent ring. The additivity of
the length comes from \cite[6.3.1]{GaRa03} and lemma~\ref{lem:DefInvTorsionOmod}.
For $x_{1},\cdots,x_{r}\in M_{1}$ and $z_{1},\cdots,z_{s}\in M_{3}$,
set $y_{i}=x_{i}\in M_{2}$ for $1\leq i\leq r$ and lift $z_{i}\in M_{3}$
to some $y_{r+i}\in M_{2}$ for $1\leq i\leq s$. Then 
\begin{eqnarray*}
\length\left(\left\langle y_{1},\cdots,y_{r+s}\right\rangle \right) & = & \length\left(\left\langle y_{1},\cdots,y_{r+s}\right\rangle \cap M_{1}\right)+\length\left(\left\langle z_{1},\cdots,z_{s}\right\rangle \right)\\
 & \geq & \length\left(\left\langle x_{1},\cdots,x_{r}\right\rangle \right)+\length\left(\left\langle z_{1},\cdots,z_{s}\right\rangle \right).
\end{eqnarray*}
Lemma~\ref{lem:InvTorsOMod_Formula} now implies that indeed $\inv(M_{1})\ast\inv(M_{3})\leq\inv(M_{2})$.
For the second inequality, let $r$ be a positive integer, fix a surjective
homomorphism $M_{2}\twoheadrightarrow M_{2}^{\prime}$ where $M_{2}^{\prime}=\oplus_{i=1}^{r}\mathcal{O}(\inv_{i}(M_{2}))$,
let $M_{1}^{\prime}\subset M_{2}^{\prime}$ be the image of $M_{1}$
and $M_{3}^{\prime}=M_{2}^{\prime}/M_{1}^{\prime}$, so that $M_{i}^{\prime}$
is a finitely presented quotient of $M_{i}$ for $i\in\{1,2,3\}$.
Then
\begin{align*}
{\textstyle \sum_{i=1}^{r}}\inv_{i}(M_{2}) & =\length(M_{2}^{\prime})=\length(M_{1}^{\prime})+\length(M_{3}^{\prime})\\
 & ={\textstyle \sum_{i=1}^{r}}\inv_{i}(M_{1}^{\prime})+{\textstyle \sum_{i=1}^{r}}\inv_{i}(M_{3}^{\prime})\\
 & \leq{\textstyle \sum_{i=1}^{r}}\inv_{i}(M_{1})+{\textstyle \sum_{i=1}^{r}}\inv_{i}(M_{3})
\end{align*}
with equality for $r\gg0$, using lemma~\ref{lem:InvTorsOModdecUnderSubquo}
and the aforementioned additivity of $\length$. Therefore indeed
$\inv(M_{2})\leq\inv(M_{1})+\inv(M_{3})$ in $\Gamma_{+,\geq}^{\infty}$. 

If our exact sequence splits, then plainly $\inv(M_{2})=\inv(M_{1})\ast\inv(M_{3})$.
We prove the converse implication by induction on $r(M_{2})$. If
$r(M_{2})=0$, there is nothing to prove: $M_{1}=M_{2}=M_{3}=0$.
Suppose therefore that $r(M_{2})\geq1$ and $\inv(M_{2})=\inv(M_{1})\ast\inv(M_{3})$,
and let $\gamma=\inv_{1}(M_{2})$. Then also $\gamma=\inv_{1}(M_{1})$
or $\gamma=\inv_{1}(M_{3})$. Using the duality $X\mapsto X^{\vee}$
from the proof of lemma~\ref{lem:InvTorsOModdecUnderSubquo}, we
may assume that $\gamma=\inv_{1}(M_{3})$. Write $M_{3}=M_{3}^{\circ}\oplus M_{3}^{\prime}$
with $M_{3}^{\circ}\simeq\mathcal{O}(\gamma)$. This lifts to a splitting
$M_{2}=M_{2}^{\circ}\oplus M_{2}^{\prime}$ of the $\mathcal{O}(\gamma)$-module
$M_{2}$, with $M_{1}\subset M_{2}^{\prime}$ and $M_{2}^{\circ}\simeq\mathcal{O}(\gamma)$.
Since $\inv(M_{i})=\inv(M_{i}^{\circ})\ast\inv(M_{i}^{\prime})$ for
$i\in\{2,3\}$, we still have $\inv(M_{2}^{\prime})=\inv(M_{1})\ast\inv(M_{3}^{\prime})$
for the exact sequence $0\rightarrow M_{1}\rightarrow M_{2}^{\prime}\rightarrow M_{3}^{\prime}\rightarrow0$.
But $r(M_{2})=r(M_{2}^{\prime})+1$, so this last sequence splits
and so does the initial one.
\end{proof}
\noindent For every $M\in\C$, there is a canonical $\Gamma$-filtration
$\mathcal{F}(M)$ on $M\otimes k$ defined by
\[
\mathcal{F}^{\gamma}(M)\eqd\begin{cases}
\frac{M[I(-\gamma)]+\mathfrak{m}M}{\mathfrak{m}M}\subset\frac{M}{\mathfrak{m}M}=M\otimes k & \mbox{if }\gamma\leq0,\\
0 & \mbox{if }\gamma\geq0.
\end{cases}
\]
It depends functorially upon $M$ and one checks easily that we have
\[
\inv(M)=\mathbf{t}^{\iota}(\mathcal{F}(M))\quad\mbox{in}\quad\Gamma_{+,\geq}^{r}\subset\Gamma_{+,\geq}^{\infty}
\]
where $r=r(M)$. In particular, $\length(M)=-\deg(\mathcal{F}(M))$.
\begin{lem}
\label{lem:InvTorsOModExSeqCaseEqual}For any exact sequence $0\rightarrow M_{1}\rightarrow M_{2}\rightarrow M_{3}\rightarrow0$
of finitely presented torsion $\mathcal{O}$-modules, the following
conditions are equivalent: 
\begin{enumerate}
\item The exact sequence splits.
\item For every $\gamma\in\Gamma$, the induced complex of $k$-vector spaces
\[
0\rightarrow\mathcal{F}^{\gamma}(M_{1})\rightarrow\mathcal{F}^{\gamma}(M_{2})\rightarrow\mathcal{F}^{\gamma}(M_{3})\rightarrow0
\]
is a short exact sequence.
\item We have $\inv(M_{2})=\inv(M_{1})\ast\inv(M_{3})$ in $\Gamma_{+,\geq}^{\infty}$.
\end{enumerate}
\end{lem}
\begin{proof}
Plainly $(1)\Rightarrow(2)$ and $(3)\Rightarrow(1)$ by lemma~\ref{lem:InvTorsOModAndExactSeq}.
If $(2)$ holds, then
\[
0\rightarrow\left(M_{1}\otimes k,\mathcal{F}(M_{1})\right)\rightarrow\left(M_{2}\otimes k,\mathcal{F}(M_{2})\right)\rightarrow\left(M_{3}\otimes k,\mathcal{F}(M_{3})\right)\rightarrow0
\]
is an exact sequence of $\Gamma$-filtered $k$-vector spaces, thus
\[
\mathbf{t}(\mathcal{F}(M_{2}))=\mathbf{t}(\mathcal{F}(M_{1}))\ast\mathbf{t}(\mathcal{F}(M_{3}))\quad\mbox{in}\quad\Gamma_{\geq}^{r}
\]
where $r=r(M_{2})=r(M_{1})+r(M_{3})$, therefore 
\[
\mathbf{t}^{\iota}(\mathcal{F}(M_{2}))=\mathbf{t}^{\iota}(\mathcal{F}(M_{1}))\ast\mathbf{t}^{\iota}(\mathcal{F}(M_{3}))\quad\mbox{in}\quad\Gamma_{\geq}^{r}
\]
from which $(3)$ immediately follows. 
\end{proof}

\subsubsection{Lattices\label{sub:Lattices}}

An $\mathcal{O}$-lattice in a finite dimensional $K$-vector space
$V$ is a finitely generated $\mathcal{O}$-submodule $L$ of $V$
spanning $V$ over $K$. Any such $L$ is finite free over $\mathcal{O}$
by \cite[VI, \S 3, 6, Lemme 1]{BoAC56}. We denote by $\mathcal{L}(V)$
the set of all $\mathcal{O}$-lattices in $V$. Since $\mathcal{O}$
is an elementary divisor ring \cite[\S 10]{Ka66}, for every $L_{1},L_{2}\in\mathcal{L}(V)$,
there is an $\mathcal{O}$-basis $(e_{1},\cdots,e_{r})$ of $L_{1}$
and elements $(x_{1},\cdots,x_{r})$ of $K^{\times}$ such that $(x_{1}e_{1},\cdots,x_{r}e_{r})$
is an $\mathcal{O}$-basis of $L_{2}$ and $\left|x_{1}\right|\geq\cdots\geq\left|x_{r}\right|$
-- we say that the basis is adapted to $L_{1}$ and $L_{2}$. If $\gamma_{i}=\log\left|x_{i}\right|$,
then $(\gamma_{1},\cdots,\gamma_{r})$ belongs to $\Gamma_{\geq}^{r}$
and does not depend upon the chosen basis. Indeed, one checks using
the given adapted basis that the formula
\[
\mathcal{F}^{\gamma}(L_{1},L_{2})\eqd\frac{L_{1}\cap I(\gamma)L_{2}+\mathfrak{m}L_{1}}{\mathfrak{m}L_{1}}\subset\frac{L_{1}}{\mathfrak{m}L_{1}}=L_{1}\otimes k
\]
defines a $\Gamma$-filtration $\mathcal{F}(L_{1},L_{2})$ on $L_{1}\otimes k$,
whose type $\mathbf{d}(L_{1},L_{2})\in\Gamma_{\geq}^{r}$ equals $(\gamma_{1},\cdots,\gamma_{r})$.
In particular, $L_{1}=L_{2}$ if and only if $\mathbf{d}(L_{1},L_{2})=0$.
This computation also shows that $\mathbf{d}(L_{2},L_{1})=\mathbf{d}^{\iota}(L_{1},L_{2})$
in $\Gamma_{\geq}^{r}$. If $L_{1}\subset L_{2}$, then $Q=L_{2}/L_{1}$
is a finitely presented torsion $\mathcal{O}$-module, $\mathbf{d}(L_{1},L_{2})\in\Gamma_{+,\geq}^{r}$
and $\mathbf{d}(L_{1},L_{2})=\inv(Q)$ in $\Gamma_{+,\geq}^{\infty}$.
If moreover $L_{1}\subset\mathfrak{m}L_{2}$ (i.e.~$\inv_{r}(Q)\neq0$),
the projection $L_{2}\twoheadrightarrow Q$ induces an isomorphism
$L_{2}\otimes k\simeq Q\otimes k$ mapping $\mathcal{F}(L_{2},L_{1})$
to $\mathcal{F}(Q)$. 
\begin{lem}
\label{lem:triangIneq}For $L_{1},L_{2},L_{3}\in\mathcal{L}(V)$,
we have the following triangular inequality: 
\[
\mathbf{d}(L_{1},L_{3})\leq\mathbf{d}(L_{1},L_{2})+\mathbf{d}(L_{2},L_{3})\quad\mbox{in}\quad\Gamma_{\geq}^{r}.
\]
\end{lem}
\begin{proof}
For any $x\in K^{\times}$ and $L,L'\in\mathcal{L}(V)$, if $\gamma=\log\left|x\right|$,
then 
\[
\mathbf{d}(x^{-1}L,L')=\mathbf{d}(L,xL')=\mathbf{d}(L,L')+(\gamma,\cdots,\gamma)\quad\mbox{in}\quad\Gamma_{\geq}^{r}.
\]
Changing $(L_{1},L_{2},L_{3})$ to $(xL_{1},L_{2},x^{-1}L_{3})$ for
a suitable $x$, we may thus assume that $L_{1}\subset L_{2}\subset L_{3}$.
Applying lemma~\ref{lem:InvTorsOModAndExactSeq} to the exact sequence
\[
0\rightarrow L_{2}/L_{1}\rightarrow L_{3}/L_{1}\rightarrow L_{3}/L_{2}\rightarrow0
\]
we obtain the desired inequality. \end{proof}
\begin{rem}
When $\Gamma\hookrightarrow\mathbb{R}$, the previous lemma also follows
from~\cite[5.2.8 \& 6.1]{Co14}.\end{rem}
\begin{lem}
\label{lem:RelPos=000026ExcSeq}Let $0\rightarrow V_{1}\rightarrow V_{2}\rightarrow V_{3}\rightarrow0$
be an exact sequence of $K$-vector spaces. For any pair of $\mathcal{O}$-lattices
$L_{2},L_{2}^{\prime}\in\mathcal{L}(V_{2})$, their inverse and direct
images in $V_{1}$ and $V_{3}$ are $\mathcal{O}$-lattices $L_{1},L'_{1}\in\mathcal{L}(V_{1})$
and $L_{3},L'_{3}\in\mathcal{L}(V_{3})$, and we have 
\[
\mathbf{d}(L_{2},L'_{2})\geq\mathbf{d}(L_{1},L'_{1})\ast\mathbf{d}(L_{3},L'_{3})\quad\mbox{in}\quad\Gamma_{\geq}^{r_{2}}
\]
where $r_{i}=\dim_{K}V_{i}$, with equality if and only if for every
$\gamma\in\Gamma$, 
\[
0\rightarrow\mathcal{F}^{\gamma}(L_{1},L_{1}^{\prime})\rightarrow\mathcal{F}^{\gamma}(L_{2},L_{2}^{\prime})\rightarrow\mathcal{F}^{\gamma}(L_{3},L_{3}^{\prime})\rightarrow0
\]
is an exact sequence. \end{lem}
\begin{proof}
Plainly $0\rightarrow L_{1}\rightarrow L_{2}\rightarrow L_{3}\rightarrow0$
and $0\rightarrow L_{1}^{\prime}\rightarrow L_{2}^{\prime}\rightarrow L_{3}^{\prime}\rightarrow0$
are exact; thus $L_{3}$ and $L_{3}^{\prime}$ are finitely generated
over $\mathcal{O}$, in particular they are both $\mathcal{O}$-lattices
in $V_{3}$ and free over $\mathcal{O}$; it follows that both exact
sequences split, which implies that $L_{1}$ and $L_{1}^{\prime}$
are also (finite free) $\mathcal{O}$-lattices in $V_{1}$. For the
remaining claims, we may as above replace $L_{2}^{\prime}$ by $xL_{2}^{\prime}$
for some $x\in K^{\times}$ (which replaces $L_{i}^{\prime}$ by $xL_{i}^{\prime}$
for $i\in\{1,3\}$) to reduce to the case where $L_{i}^{\prime}\subset\mathfrak{m}L_{i}\subset L_{i}$
for all $i\in\{1,2,3\}$. Applying lemma~\ref{lem:InvTorsOModAndExactSeq}
to the resulting exact sequence of finitely presented torsion $\mathcal{O}$-modules
\[
0\rightarrow L_{1}/L_{1}^{\prime}\rightarrow L_{2}/L_{2}^{\prime}\rightarrow L_{3}/L_{3}^{\prime}\rightarrow0
\]
we obtain the inequality $\mathbf{d}^{\iota}(L_{2},L'_{2})\geq\mathbf{d}^{\iota}(L_{1},L'_{1})\ast\mathbf{d}^{\iota}(L_{3},L'_{3})$
in $\Gamma_{\geq}^{r_{2}}$, which is equivalent to the desired inequality
$\mathbf{d}(L_{2},L'_{2})\geq\mathbf{d}(L_{1},L'_{1})\ast\mathbf{d}(L_{3},L'_{3})$.
Moreover by lemma~\ref{lem:InvTorsOModExSeqCaseEqual}, equality
holds in either one of them if and only if for every $\gamma\in\Gamma$,
\[
0\rightarrow\mathcal{F}^{\gamma}(L_{1},L_{1}^{\prime})\rightarrow\mathcal{F}^{\gamma}(L_{2},L_{2}^{\prime})\rightarrow\mathcal{F}^{\gamma}(L_{3},L_{3}^{\prime})\rightarrow0
\]
is an exact sequence of $k$-vector spaces. This proves the lemma.\end{proof}
\begin{rem}
When $\Gamma\hookrightarrow\mathbb{R}$, the inequality also follows
from~\cite[5.2.10 \& 6.1]{Co14}.
\end{rem}
\noindent For $L_{1},L_{2}\in\mathcal{L}(V)$, we denote by $\nu(L_{1},L_{2})\in\Gamma$
the degree of $\mathbf{d}(L_{1},L_{2})$.

\subsubsection{Tensor products\label{sub:TensorProd}}

There are also compatible notions of tensor products, symmetric and
exterior powers for types, objects and $\Gamma$-filtered objects
in arbitrary quasi-tannakian categories, and $\mathcal{O}$-lattices
in $K$-vector spaces. All of these notions are fairly classical,
and their various compatibilities easily checked. For instance if
$L$ and $L^{\prime}$ are $\mathcal{O}$-lattices in $V$ and $i\in\mathbb{N}$,
then $\Lambda_{\mathcal{O}}^{i}L$ and $\Lambda_{\mathcal{O}}^{i}L'$
are $\mathcal{O}$-lattices in $\Lambda_{K}^{i}V$, $\mathcal{F}(\Lambda_{\mathcal{O}}^{i}L,\Lambda_{\mathcal{O}}^{i}L')$
is the $\Gamma$-filtration $\Lambda_{k}^{i}\mathcal{F}(L,L')$ on
$\Lambda_{k}^{i}(L\otimes_{\mathcal{O}}k)$ which is the image of
the $\Gamma$-filtration $\mathcal{F}(L,L')^{\otimes i}$ on $(L\otimes_{\mathcal{O}}k)^{\otimes i}$
under the projection $(L\otimes_{\mathcal{O}}k)^{\otimes i}\twoheadrightarrow\Lambda_{k}^{i}(L\otimes_{\mathcal{O}}k)$,
where $\mathcal{F}(L,L')^{\otimes i}(\gamma)=\sum_{\gamma_{1}+\cdots+\gamma_{i}=\gamma}\mathcal{F}(L,L')(\gamma_{1})\otimes\cdots\otimes\mathcal{F}(L,L')(\gamma_{i})$
in $(L\otimes_{\mathcal{O}}k)^{\otimes i}$ for every $\gamma\in\Gamma$.
The type $\mathbf{d}(\Lambda_{\mathcal{O}}^{i}L,\Lambda_{\mathcal{O}}^{i}L')=\mathbf{t}(\Lambda_{k}^{i}\mathcal{F}(L,L'))=\Lambda^{i}\mathbf{d}(L,L')$
is obtained from $\mathbf{d}(L,L')=(\gamma_{1},\cdots,\gamma_{r})$
(with $r=\dim_{K}V$) by reordering the elements $\gamma_{I}=\sum_{j\in I}\gamma_{j}$
where $I$ ranges through all subsets of $\{1,\cdots,r\}$ of cardinality
$i$.

\section{Breuil-Kisin-Fargues Modules\label{sec:BKF-Modules}}

\subsection{The rings}

Let $p$ be a prime number, $E$ be a finite extension of $\mathbb{Q}_{p}$,
$K$ a perfectoid field extension of $E$, $K^{\flat}$ the tilt of
$K$. We denote by $\mathcal{O}_{E}$, $\mathcal{O}_{K}$ and $\mathcal{O}_{K}^{\flat}$
the ring of integers in $E$, $K$ and $K^{\flat}$, with maximal
ideals $\mathfrak{m}_{E}$, $\mathfrak{m}_{K}$ and $\mathfrak{m}_{K}^{\flat}$,
and perfect residue fields $\mathbb{F}_{q}:=\mathcal{O}_{E}/\mathfrak{m}_{E}$
(finite with $q$ elements) and $\mathbb{F}:=\mathcal{O}_{K}/\mathfrak{m}_{K}=\mathcal{O}_{K}^{\flat}/\mathfrak{m}_{K}^{\flat}$.
We fix once and for all a uniformizer $\pi$ of $E$. We denote by
$W_{\mathcal{O}_{E}}(-)$ the Witt vector functor with values in $\mathcal{O}_{E}$-algebras,
as defined in~\cite[1.2]{FaFo15}. We set 
\[
A(\mathcal{O}_{K})\eqd W_{\mathcal{O}_{E}}(\mathcal{O}_{K}^{\flat}),\quad A(K)\eqd W_{\mathcal{O}_{E}}(K^{\flat}),\quad\mathcal{O}_{L}\eqd W_{\mathcal{O}_{E}}(\mathbb{F}),\quad L\eqd\Frac(\mathcal{O}_{L}).
\]
Thus $A(K)$ and $\mathcal{O}_{L}$ are complete discrete valuation
rings with uniformizer $\pi$ and residue fields respectively equal
to $K^{\flat}$ and $\mathbb{F}$, while our main player $A:=A(\mathcal{O}_{K})$
is a non-noetherian complete local ring with maximal ideal $\mathfrak{m}$
and residue field $\mathbb{F}$. We denote by $\varphi$ the Frobenius
$x\mapsto x^{q}$ in characteristic $p$ or its extension to $\mathcal{O}_{E}$-Witt
vectors. The ring homomorphisms $\mathcal{O}_{K}^{\flat}\hookrightarrow K^{\flat}$
and $\mathcal{O}_{K}^{\flat}\twoheadrightarrow\mathbb{F}$ induce
$\varphi$-equivariant homomorphisms of $\mathcal{O}_{E}$-algebras
$A\hookrightarrow A(K)$ and $A\twoheadrightarrow\mathcal{O}_{L}$.
The formula $\theta(\sum_{n\geq0}[(x_{i,n})_{i}]\pi^{n})=\sum x_{0,n}\pi^{n}$
defines a surjective ring homomorphism $\theta:A\twoheadrightarrow\mathcal{O}_{K}$
whose kernel is a principal ideal. Here $[-]$ is the Teichmüller
lift and $(x_{i,n})_{i\geq0}\in\mathcal{O}_{K}^{\flat}$ for all $n\geq0$,
i.e.~$x_{i,n}\in\mathcal{O}_{K}$ with $x_{i,n}=x_{i+1,n}^{p}$ for
all $i\geq0$. We fix a generator $\xi$ of $\ker(\theta)$ and set
$\xi':=\varphi(\xi)$. We write $\varpi$ for the image of $\xi$
in $\mathcal{O}_{K}^{\flat}=A_{1}$, where more generally $A_{n}:=A/\pi^{n}A$
for $n\in\mathbb{N}$. Thus $\varpi$ is a pseudo-uniformizer of $K^{\flat}$,
i.e.~a non-zero element of $\mathfrak{m}_{K}^{\flat}$. For an $A$-module
$M$ and $n\in\mathbb{N}$, we define\emph{ }
\[
M(K)\eqd M\otimes_{A}A(K),\quad M(\mathcal{O}_{L})\eqd M\otimes_{A}\mathcal{O}_{L},\quad M_{n}\eqd M\otimes_{A}A_{n}=M/\pi^{n}M.
\]
In particular, $M_{1}=M\otimes_{A}\mathcal{O}_{K}^{\flat}$. We normalize
the absolute value of $K^{\flat}$ by requiring that $q\left|\varpi^{q}\right|=1$.

\subsection{Categories of $A$-modules}

\subsubsection{~\label{sub:ProjDimModA}}

For an $A$-module $M$, we denote by $\tilde{M}$ the corresponding
quasi-coherent sheaf on $X:=\Spec\, A$. Since $U:=X\setminus\{\mathfrak{m}\}$
is a quasi-compact open subscheme of the affine scheme $X$, there
is an exact sequence \cite[II Corollaire 4]{SGA2r} of $A$-modules
\[
0\rightarrow H_{\{\mathfrak{m}\}}^{0}(X,\tilde{M})\rightarrow M=H^{0}(X,\tilde{M})\rightarrow H^{0}(U,\tilde{M})\rightarrow H_{\{\mathfrak{m}\}}^{1}(X,\tilde{M})\rightarrow0
\]
and for every $i\geq1$, an isomorphism of $A$-modules 
\[
H^{i}(U,\tilde{M})\simeq H_{\{\mathfrak{m}\}}^{i+1}(X,\tilde{M}).
\]
Moreover for any sequence of parameters $(a,b)$ spanning an ideal
$I$ with $\sqrt{I}=\mathfrak{m}$, 
\[
H_{\{\mathfrak{m}\}}^{i}(X,\tilde{M})\simeq H^{i}\left(\left[M\rightarrow M\left[{\textstyle \frac{1}{a}}\right]\oplus M\left[{\textstyle \frac{1}{b}}\right]\rightarrow M\left[{\textstyle \frac{1}{ab}}\right]\right]\right)
\]
by \cite[II Proposition 5]{SGA2r}, thus $H_{\{\mathfrak{m}\}}^{i}(X,\tilde{M})=H^{i-1}(U,\tilde{M})=0$
for $i\geq3$. Also, 
\[
H_{\{\mathfrak{m}\}}^{i}(X,\tilde{M})=\underrightarrow{\lim}\,\Ext_{A}^{i}\left(A/I^{n},M\right)
\]
for any $i\geq0$ if moreover $(a,b)$ is regular by \cite[II Lemme 9]{SGA2r}.
For $M=A$, we find
\[
H_{\{\mathfrak{m}\}}^{i}\left(X,\mathcal{O}_{X}\right)=\begin{cases}
0 & \mbox{if }i\neq2\\
\mathcal{E} & \mbox{if }i=2
\end{cases}\quad\mbox{with}\quad\mathcal{E}=\frac{A\left[{\textstyle \frac{1}{\pi[\varpi]}}\right]}{A\left[{\textstyle \frac{1}{\pi}}\right]+A\left[{\textstyle \frac{1}{[\varpi]}}\right]}\neq0
\]
using \cite[4.6]{BaMoSc16} for $i=1$. By~\cite[2.6 \& 2.7]{HaMa07}
and with the definition given there, 
\[
\mbox{p-depth}_{A}(M)=\sup\left\{ k\geq0:H_{\{\mathfrak{m}\}}^{i}\left(X,\tilde{M}\right)=0\,\mbox{for all }i<k\right\} .
\]
In particular $\mbox{p-depth}_{A}(A)=2$. We say that the $A$-module
$M$ is perfect if it has a finite resolution by finite free $A$-modules.
The Auslander-Buchsbaum theorem of \cite[Chapter 6, Theorem 2]{No76}
then assert that for any such $M$, 
\[
\mbox{proj}.\mbox{dim}_{A}(M)+\mbox{p-depth}_{A}(M)=2.
\]
In particular, $\mbox{proj}.\mbox{dim}_{A}(M)\leq1$ if and only if
the $A$-submodule 
\[
M[\mathfrak{m}^{\infty}]\eqd H_{\{\mathfrak{m}\}}^{0}(X,\tilde{M})
\]
of $M=H^{0}(X,\tilde{M})$ is trivial, and $M$ is finite free if
and only if moreover 
\[
H_{\{\mathfrak{m}\}}^{1}(X,\tilde{M})=\coker\left(M\rightarrow H^{1}(U,\tilde{M})\right)
\]
is trivial.

\subsubsection{~\label{sub:basicsonModA*}}

We denote by $\Mod_{A}$ the abelian category of all $A$-modules.
Let $\Mod_{A,\ast}$ be the strictly full subcategory of finitely
presented $A$-modules $M$ such that $M[\frac{1}{\pi}]$ is a projective
$A[\frac{1}{\pi}]$-module. Any such $M$ is a perfect $A$-module
and $M[\frac{1}{\pi}]$ is actually finite and free over $A[\frac{1}{\pi}]$
by \cite[4.9 \& 4.12]{BaMoSc16}. By \cite[4.13]{BaMoSc16}, the $A$-dual
$M^{\vee}:=\Hom_{A}(M,A)$ is finite free over $A$, so is the bidual
$M_{f}:=M^{\vee\vee}$, the kernel of the canonical morphism $M\rightarrow M_{f}$
is the torsion submodule $M[\pi^{\infty}]$ of $M$, it is a finitely
presented $A$-module killed by $\pi^{n}$ for $n\gg0$, and the cokernel
of $M\rightarrow M_{f}$ is a finitely presented torsion $A$-module
$\overline{M}$ supported at $\mathfrak{m}$. We claim that $M[\mathfrak{m}^{\infty}]$
is then also a finitely presented $A$-module (supported at $\mathfrak{m}$).
To see this, note that 
\[
M[\mathfrak{m}^{\infty}]=\ker\left(M[\pi^{\infty}]\rightarrow M[\pi^{\infty}][{\textstyle \frac{1}{[\varpi]}}]\right).
\]
Since $M[\pi^{\infty}][{\textstyle \frac{1}{[\varpi]}}]\simeq M[\pi^{\infty}](K)$
is a finitely generated torsion module over the complete discrete
valuation ring $A(K)$, there is a unique sequence of integers 
\[
\inv_{A(K)}\left(M[\pi^{\infty}](K)\right)\eqd(n_{1}\geq\cdots\geq n_{s}>0)\quad\mbox{in}\quad\mathbb{N}_{\geq}^{s}
\]
for some $s\in\mathbb{N}$ such that $M[\pi^{\infty}][{\textstyle \frac{1}{[\varpi]}}]$
is isomorphic to 
\[
\oplus_{i=1}^{s}A_{n_{i}}(K)=\oplus_{i=1}^{s}A_{n_{i}}[{\textstyle \frac{1}{[\varpi]}}].
\]
Chasing denominators, we may modify any such isomorphism into one
that fits in a commutative diagram of $\pi^{\infty}$-torsion $A$-modules
\[
\begin{array}{ccc}
\oplus_{i=1}^{\ell}A_{n_{i}} & \hookrightarrow & M[\pi^{\infty}]\\
\cap &  & \downarrow\\
\oplus_{i=1}^{\ell}A_{n_{i}}[{\textstyle \frac{1}{[\varpi]}]} & \stackrel{\simeq}{\longrightarrow} & M[\pi^{\infty}][{\textstyle \frac{1}{[\varpi]}]}
\end{array}
\]
If $\pi^{n}M[\pi^{\infty}]=0$, the cokernel of the top map is a finitely
generated $A_{n}$-module $Q$ with $Q[\frac{1}{[\varpi]}]=0$, thus
$[\varpi]^{m}Q=0$ for $m\gg0$. Since $(\oplus_{i=1}^{\ell}A_{n_{i}})[\mathfrak{m}^{\infty}]=0$,
$M[\mathfrak{m}^{\infty}]$ embeds into $Q$, therefore also $[\varpi]^{m}M[\mathfrak{m}^{\infty}]=0$,
i.e.~$M[\mathfrak{m}^{\infty}]$ is the kernel of $[\varpi]^{m}$
acting on the finitely presented $A_{n}$-module $M[\pi^{\infty}]$.
It follows that $M[\mathfrak{m}^{\infty}]$ is itself finitely presented
over $A_{n}$ and $A$, since $A_{n}$ is a coherent ring by (the
easy case of) \cite[Proposition 3.24]{BaMoSc16}. We finally define
the subquotient 
\[
M_{t}\eqd M[\pi^{\infty}]/M[\mathfrak{m}^{\infty}].
\]
This is a finitely presented $A$-module killed by $\pi^{n}$ for
$n\gg0$.

\subsubsection{~\label{sub:DevissageOfModules}}

We will consider the following strictly full subcategories of $\Mod_{A,\ast}$:
\[
\begin{array}{lcl}
\Mod_{A,f} & \eqd & \left\{ \mbox{finite free \ensuremath{A}-modules}\right\} ,\\
\Mod_{A,\pi^{\infty}} & \eqd & \left\{ \mbox{finitely presented \ensuremath{A}-modules killed by \ensuremath{\pi^{n}\,}for \ensuremath{n\gg0}}\right\} \\
 & = & \left\{ M\in\Mod_{A,\ast}\mbox{ such that }M=M[\pi^{\infty}]\right\} ,\\
\Mod_{A,\mathfrak{m}^{\infty}} & \eqd & \left\{ \mbox{finitely presented \ensuremath{A}-modules killed by \ensuremath{(\pi,[\varpi])^{n}\,}for \ensuremath{n\gg0}}\right\} \\
 & = & \left\{ M\in\Mod_{A,\ast}\mbox{ such that }M=M[\mathfrak{m}^{\infty}]\right\} ,\\
\Mod_{A,t} & \eqd & \left\{ \mbox{finitely presented \ensuremath{A}-modules with \ensuremath{\pi\,}nilpotent and \ensuremath{[\varpi]\,}injective}\right\} \\
 & = & \left\{ M\in\Mod_{A,\ast}\mbox{ such that \ensuremath{M=M_{t}}}\right\} .
\end{array}
\]
Then any $M\in\Mod_{A,\ast}$ has a canonical and functorial \emph{dévissage
\[
\xyR{1pc}\xymatrix{ & M[\mathfrak{m}^{\infty}]\ar@{^{(}->}[d]\\
0\ar[r] & M[\pi^{\infty}]\ar[r]\ar@{->>}[d] & M\ar[r] & M_{f}\ar[r] & \overline{M}\ar[r] & 0\\
 & M_{t}
}
\]
}with everyone in the relevant subcategory. The projective dimension
of the nonzero $A$-modules in $\Mod_{A,f}$, $\Mod_{A,t}$ and $\Mod_{A,\mathfrak{m}^{\infty}}$
are respectively $0$, $1$ and $2$.

\subsubsection{~\label{sub:noncanseq}}

For $n\gg0$, $\pi^{n}$ kills $M[\pi^{\infty}]$ and $\overline{M}$,
thus for any $m\in M_{f}$, $\pi^{n}m$ is the image of some $m'\in M$
and $\pi^{n}m'\in M$ only depends upon $m$. This defines an embedding
$M_{f}\hookrightarrow M$ whose cokernel $Q$ is a finitely presented
$A$-module killed by $\pi^{2n}$:
\[
0\rightarrow M_{f}\rightarrow M\rightarrow Q\rightarrow0.
\]
If $M[\mathfrak{m}^{\infty}]=0$, then also $Q[\mathfrak{m}^{\infty}]=0$,
i.e.~$Q\in\Mod_{A,t}$.

\subsubsection{~\label{sub:FiltOnModAt=000026ModAm}}

Any $A$-module $M$ in $\Mod_{A,\pi^{\infty}}$ has yet another canonical
and functorial \emph{dévissage}, the finite non-decreasing filtration
by the finitely presented $A$-submodules $M[\pi^{n}]$ of $M$ whose
successive quotients $M[\pi^{n}]/M[\pi^{n-1}]\simeq\pi^{n-1}M[\pi^{n}]$
are finitely presented $\mathcal{O}_{K}^{\flat}$-modules. If $M$
belongs to $\Mod_{A,t}$, these subquotients are torsion free, thus
finite free over $\mathcal{O}_{K}^{\flat}$. If $M$ belongs to $\Mod_{A,\mathfrak{m}^{\infty}}$,
they are finitely presented torsion $\mathcal{O}_{K}^{\flat}$-modules,
thus themselves non-canonically isomorphic to direct sums of modules
of the form $\mathcal{O}_{K}^{\flat}(x):=\mathcal{O}_{K}^{\flat}/x\mathcal{O}_{K}^{\flat}$
with $x$ nonzero in $\mathcal{O}_{K}^{\flat}$.

\subsubsection{~\label{sub:Tor(A(K)LO_L)}}

For every $A$-module $N$ and any nonzero $x\in\mathcal{O}_{K}^{\flat}$,
the exact sequences 
\[
0\rightarrow A\stackrel{\pi}{\rightarrow}A\rightarrow\mathcal{O}_{K}^{\flat}\rightarrow0\quad\mbox{and}\quad0\rightarrow\mathcal{O}_{K}^{\flat}\stackrel{x}{\rightarrow}\mathcal{O}_{K}^{\flat}\rightarrow\mathcal{O}_{K}^{\flat}(x)\rightarrow0
\]
give $\Tor_{0}^{A}(N,\mathcal{O}_{K}^{\flat})=N/\pi N$, $\Tor_{1}^{A}(N,\mathcal{O}_{K}^{\flat})=N[\pi]$,
and an exact sequence
\[
0\rightarrow N[\pi]/xN[\pi]\rightarrow\Tor_{1}^{A}(N,\mathcal{O}_{K}^{\flat}(x))\rightarrow(N/\pi N)[x]\rightarrow0.
\]
It follows that for every $M\in\Mod_{A,\ast}$, 
\[
\Tor_{1}^{A}(A(K),M)=0\quad\mbox{and}\quad\Tor_{1}^{A}(L,M)=0
\]
since this holds for $M\in\{A,\mathcal{O}_{K}^{\flat},\mathcal{O}_{K}^{\flat}(x)\}$.
If moreover $M[\mathfrak{m}^{\infty}]=0$, then also 
\[
\Tor_{1}^{A}(\mathcal{O}_{L},M)=0
\]
since this holds for $M\in\{A,\mathcal{O}_{K}^{\flat}\}$.

\subsubsection{~}

The category $\Mod_{A,\ast}$ is stable under extensions in $\Mod_{A}$.
The next proposition implies that it inherits from $\Mod_{A}$ the
structure of a closed symmetric monoidal category, which just says
that~$\Mod_{A,\ast}$ is a $\otimes$-category with internal Homs. 
\begin{prop}
For every $M_{1}$ and $M_{2}$ in $\Mod_{A,\ast}$ and any $i\geq0$,
\[
\Tor_{i}^{A}(M_{1},M_{2})\quad\mbox{and}\quad\Ext_{A}^{i}(M_{1},M_{2})
\]
also belong to $\Mod_{A,\ast}$.\end{prop}
\begin{proof}
Fix a finite resolution $P_{\bullet}$ of $M_{1}$ by finite free
$A$-modules. Then 
\[
\Ext_{A}^{i}(M_{1},M_{2})=H^{i}\left(\Hom_{A}(P_{\bullet},M_{2})\right)\quad\mbox{and}\quad\Tor_{i}^{A}(M_{1},M_{2})=H^{i}\left(P_{\bullet}\otimes_{A}M_{2})\right).
\]
Since $\Hom_{A}(P_{\bullet},M_{2})$ and $P_{\bullet}\otimes_{A}M_{2}$
are perfect complexes, their cohomology groups are finitely presented
over $A$. Since moreover $A\rightarrow A[\frac{1}{\pi}]$ is flat,
\[
\Ext_{A}^{i}(M_{1},M_{2})[{\textstyle \frac{1}{\pi}}]=\Ext_{A[{\textstyle \frac{1}{\pi}}]}^{i}(M_{1}[{\textstyle \frac{1}{\pi}}],M_{2}[{\textstyle \frac{1}{\pi}}])=\begin{cases}
\Hom_{A[{\textstyle \frac{1}{\pi}}]}(M_{1}[{\textstyle \frac{1}{\pi}}],M_{2}[{\textstyle \frac{1}{\pi}}]) & \mbox{if }i=0,\\
0 & \mbox{if }i>0,
\end{cases}
\]
since $M_{1}[{\textstyle \frac{1}{\pi}}]$ is finite free over $A[{\textstyle \frac{1}{\pi}}]$,
and similarly 
\[
\Tor_{i}^{A}(M_{1},M_{2})[{\textstyle \frac{1}{\pi}}]=\Tor_{i}^{A[{\textstyle \frac{1}{\pi}}]}(M_{1}[{\textstyle \frac{1}{\pi}}],M_{2}[{\textstyle \frac{1}{\pi}}])=\begin{cases}
M_{1}[{\textstyle \frac{1}{\pi}}]\otimes_{A[{\textstyle \frac{1}{\pi}}]}M_{2}[{\textstyle \frac{1}{\pi}}] & \mbox{if }i=0,\\
0 & \mbox{if }i>0.
\end{cases}
\]
So all of these $A[{\textstyle \frac{1}{\pi}}]$-modules are indeed
finite and free. 
\end{proof}

\subsubsection{~\label{sub:ModAtquasiab}}

The categories $\Mod_{A,\pi^{\infty}}$ and $\Mod_{A,\mathfrak{m}^{\infty}}$
are weak Serre subcategories of $\Mod_{A}$: they are stable under
kernels, cokernels and extensions. In particular, they are both abelian.
The category $\Mod_{A,t}$ is also stable by extensions and kernels
in $\Mod_{A}$, but it is only quasi-abelian. In fact, the exact sequence
(for $M\in\Mod_{A,\pi^{\infty}}$) 
\[
0\rightarrow M[\mathfrak{m}^{\infty}]\rightarrow M\rightarrow M_{t}\rightarrow0
\]
 yields a cotilting torsion theory \cite{BoVdB03} on the abelian
category $\Mod_{A,\pi^{\infty}}$ with torsion class $\Mod_{A,\mathfrak{m}^{\infty}}$
and torsion-free class $\Mod_{A,t}$: any $M\in\Mod_{A,\pi^{\infty}}$
is a quotient of $A_{n}^{r}\in\Mod_{A,t}$ for some $n,r\in\mathbb{N}$,
and there is no nonzero morphism from an object in $\Mod_{A,\mathfrak{m}^{\infty}}$
to an object in $\Mod_{A,t}$. The kernel and coimage of a morphism
in $\Mod_{A,t}$ are the corresponding kernel and coimage in the abelian
category $\Mod_{A,\pi^{\infty}}$ or $\Mod_{A}$. The image and cokernel
of $f:M\rightarrow N$ in $\Mod_{A,t}$ are given by 
\[
\mathrm{im}_{\Mod_{A,t}}(f)=f(M)^{\mathrm{sat}}\quad\mbox{and}\quad\coker_{\Mod_{A,t}}(f)=(N/f(M))_{t}=N/f(M)^{\mathrm{sat}}
\]
where 
\[
f(M)^{\mathrm{sat}}/f(M)\eqd(N/f(M))[\mathfrak{m}^{\infty}]=(N/f(M))[[\varpi]^{\infty}].
\]
The morphism $f$ is strict if and only if $N/f(M)$ has no $[\varpi]$-torsion.
It is a mono-epi if and only if $f$ is injective and $N/f(M)$ is
killed by $[\varpi]^{n}$ for $n\gg0$. Finally, short exact sequences
in $\Mod_{A,t}$ remain exact in $\Mod_{A}$.

\subsubsection{~}

The categories $\Mod_{A,f}$, $\Mod_{A,\pi^{\infty}}$ and $\Mod_{A,\mathfrak{m}^{\infty}}$
are stable under the usual $\Ext$'s and $\Tor$'s in $\Mod_{A}$,
and so they are also $\otimes$-categories with internal Homs (but
only $\Mod_{A,f}$ has a neutral object). They are also stable under
symmetric and exterior powers (of rank $k\geq1$ for the torsion categories). 

The category $\Mod_{A,t}$ is stable under the internal Hom of $\Mod_{A}$,
but it is not stable under the $\otimes$-product of $\Mod_{A}$.
For instance, if $x\neq0$ belongs to $\mathfrak{m}_{K}^{\flat}$,
then 
\[
M=(\pi,[x])/(\pi^{2})
\]
 is a finitely generated ideal of $A_{2}$, so it belongs to $\Mod_{A,t}$,
but the image of $\pi$ in 
\[
M\otimes_{A}\mathcal{O}_{K}^{\flat}=M/\pi M=(\pi,[x])/(\pi^{2},\pi[x])
\]
is a nonzero element killed by $[x]\in\mathfrak{m}\setminus\pi A$.
We can nevertheless equip $\Mod_{A,t}$ with a tensor product compatible
with the usual internal Hom, given by 
\[
(M_{1},M_{2})\mapsto M_{1}\otimes_{t}M_{2}\eqd\left(M_{1}\otimes_{A}M_{2}\right)_{t}=\left(M_{1}\otimes_{A}M_{2}\right)/\left(M_{1}\otimes_{A}M_{2}\right)[\mathfrak{m}^{\infty}]
\]
With this definition, $\Mod_{A,t}$ becomes yet another $\otimes$-category
with internal Homs.

\subsubsection{~}

As explained in \ref{sub:basicsonModA*} or \ref{sub:Invariants},
there is an invariant 
\[
\inv_{t}:\sk\,\Mod_{A,\pi^{\infty}}\rightarrow\mathbb{N}_{\geq}^{\infty}
\]
defined as follows: for every $M\in\Mod_{A,\pi^{\infty}}$, 
\[
\inv_{t}(M)=(n_{1}\geq\cdots\geq n_{s})\iff M(K)\simeq\oplus_{i=1}^{s}A_{n_{i}}(K).
\]
Alternatively, $\inv_{t}(M)$ is the unique element $(n_{1}\geq\cdots\geq n_{s})$
of $\mathbb{N}_{\geq}^{\infty}$ such that 
\[
\forall n\geq1:\qquad\rank_{\mathcal{O}_{K}^{\flat}}\left(M[\pi^{n}]/M[\pi^{n-1}]\right)=\left|\left\{ i:n_{i}\geq n\right\} \right|.
\]
This follows from~\ref{sub:Tor(A(K)LO_L)}, which indeed implies
that for every $n\geq1$, 
\[
M(K)[\pi^{n}]/M(K)[\pi^{n-1}]\simeq M[\pi^{n}]/M[\pi^{n-1}]\otimes_{\mathcal{O}_{K}^{\flat}}K^{\flat}.
\]
This invariant yields a function $\rank_{t}:\sk\,\Mod_{A,\pi^{\infty}}\rightarrow\mathbb{N}$
defined by 
\[
\rank_{t}(M)=\deg(\inv_{t}(M))=\sum_{i=1}^{s}n_{i}=\length_{A(K)}M(K).
\]
Plainly, $\inv_{t}(M)=\inv_{t}(M_{t})$, $\rank_{t}(M)=\rank_{t}(M_{t})$
and
\[
\begin{array}{rcccl}
\rank_{t}(M)=0 & \iff & \inv_{t}(M)=0 & \iff & {\textstyle M}(K)=0\\
 & \iff & M_{t}=0 & \iff & M=M[\mathfrak{m}^{\infty}].
\end{array}
\]
Moreover by \ref{sub:Tor(A(K)LO_L)} and lemma~\ref{lem:InvTorsOModAndExactSeq},
for every exact sequence 
\[
0\rightarrow M_{1}\rightarrow M_{2}\rightarrow M_{3}\rightarrow0
\]
in $\Mod_{A,\pi^{\infty}}$, we have $\rank_{t}(M_{2})=\rank_{t}(M_{1})+\rank_{t}(M_{3})$
and 
\[
\inv_{t}(M_{2})\geq\inv_{t}(M_{1})\ast\inv_{t}(M_{3})\quad\mbox{in}\quad\mathbb{N}_{\geq}^{\infty}
\]
with equality if and only if the exact sequence 
\[
0\rightarrow M_{1}(K)\rightarrow M_{2}(K)\rightarrow M_{3}(K)\rightarrow0
\]
of finite length $A(K)$-modules is split. In particular, the function
\[
\rank_{t}:\sk\,\Mod_{A,t}\rightarrow\mathbb{N}
\]
is a rank function on $\Mod_{A,t}$ in the sense of \cite{Co16}:
it is additive on short exact sequences, nonzero on nonzero objects,
and constant on mono-epis in $\Mod_{A,t}$.

\subsubsection{~}

There is another invariant 
\[
\inv_{\pi^{\infty}}:\sk\,\Mod_{A,\pi^{\infty}}\rightarrow\mathbb{N}_{\geq}^{\infty}
\]
defined as follows: for every $M\in\Mod_{A,\pi^{\infty}}$, 
\[
\inv_{\pi^{\infty}}(M)=(n_{1}\geq\cdots\geq n_{s})\iff M(\mathcal{O}_{L})=\oplus_{i=1}^{s}\mathcal{O}_{L}/\pi^{n_{i}}\mathcal{O}_{L}.
\]
Using~\ref{sub:Tor(A(K)LO_L)} as above, we now find that if $M$
belongs to $\Mod_{A,t}$, then again 
\[
\forall n\geq1:\qquad\rank_{\mathcal{O}_{K}^{\flat}}\left(M[\pi^{n}]/M[\pi^{n-1}]\right)=\left|\left\{ i:n_{i}\geq n\right\} \right|.
\]
In particular, both invariants coincide on $\Mod_{A,t}$ and for any
exact sequence 
\[
0\rightarrow M_{1}\rightarrow M_{2}\rightarrow M_{3}\rightarrow0
\]
in $\Mod_{A,t}$, we thus also have by lemma~\ref{lem:InvTorsOModAndExactSeq}
\[
\inv_{\pi^{\infty}}(M_{2})\geq\inv_{\pi^{\infty}}(M_{1})\ast\inv_{\pi^{\infty}}(M_{3})\quad\mbox{in}\quad\mathbb{N}_{\geq}^{\infty}
\]
with equality if and only if the exact sequence of finite length $\mathcal{O}_{L}$-modules
\[
0\rightarrow M_{1}(\mathcal{O}_{L})\rightarrow M_{2}(\mathcal{O}_{L})\rightarrow M_{3}(\mathcal{O}_{L})\rightarrow0
\]
is split. For a general $M$ in $\Mod_{A,\pi^{\infty}}$, the exact
sequence
\[
0\rightarrow M[\mathfrak{m}^{\infty}](\mathcal{O}_{L})\rightarrow M(\mathcal{O}_{L})\rightarrow M_{t}(\mathcal{O}_{L})\rightarrow0
\]
then shows that 
\[
\inv_{\pi^{\infty}}(M)\geq\inv_{\pi^{\infty}}(M_{t})\ast\inv_{\pi^{\infty}}(M[\mathfrak{m}^{\infty}])=\inv_{t}(M)\ast\inv_{\pi^{\infty}}(M[\mathfrak{m}^{\infty}])
\]
with equality if and only if the exact sequence is split.

\subsubsection{~\label{sub:rankfree/tors}}

For $M\in\Mod_{A,\ast}$, let $I$ be the image of $M\rightarrow M_{f}$.
For any $n\geq1$, recall that $M_{n}=M/\pi^{n}M$, which is a finitely
presented $A_{n}$-module. The dévissage of $M$ from~\ref{sub:DevissageOfModules}
yields exact sequences of finitely presented $A_{n}$-modules
\[
0\rightarrow M[\pi^{\infty}]_{n}\rightarrow M_{n}\rightarrow I_{n}\rightarrow0\quad\mbox{and}\quad0\rightarrow\overline{M}[\pi^{n}]\rightarrow I_{n}\rightarrow M_{f,n}\rightarrow\overline{M}_{n}\rightarrow0.
\]
It follows that $I_{n}(K)\simeq M_{f,n}(K)$ is finite free over $A_{n}(K)$
and 
\[
M_{n}(K)\simeq M[\pi^{\infty}]_{n}(K)\oplus M_{f,n}(K)\simeq M_{t,n}(K)\oplus M_{f,n}(K).
\]
In particular, $\inv_{t}M_{n}=\inv_{t}M_{t,n}\ast\inv_{t}M_{f,n}$
and 
\[
\rank_{t}M_{n}=\rank_{t}M_{t,n}+n\,\rank_{A}M_{f}
\]
with $n\mapsto\rank_{t}M_{t,n}$ non-decreasing and equal to $\rank_{t}M_{t}$
for $n\gg0$.

\subsubsection{~\label{sub:defdegonModminfty}}

A \emph{good filtration} on a module $M$ in $\Mod_{A,\mathfrak{m}^{\infty}}$
is a sequence 
\[
0=M_{0}\subsetneq M_{1}\subsetneq\cdots\subsetneq M_{r}=M
\]
of $A$-submodules such that for all $i\in\{1,\cdots,r\}$, $M_{i}/M_{i-1}\simeq\mathcal{O}_{K}^{\flat}(x_{i})$
for some nonzero $x_{i}\in\mathcal{O}_{K}^{\flat}$ -- thus $M_{i}\in\Mod_{A,\mathfrak{m}^{\infty}}$
for all $i$. We have seen in \ref{sub:FiltOnModAt=000026ModAm}~that
any $M$ in $\Mod_{A,\mathfrak{m}^{\infty}}$ has such a good filtration.
We claim that the principal ideal 
\[
\delta(M)\eqd\mathcal{O}_{K}^{\flat}{\textstyle \prod_{i=1}^{r}}x_{i}
\]
 does not depend upon the chosen good filtration on $M$. Indeed if
\[
0=M'_{0}\subsetneq M'_{1}\subsetneq\cdots\subsetneq M'_{r'}=M
\]
is another good filtration with $M'_{i}/M'_{i-1}\simeq\mathcal{O}_{K}^{\flat}(y_{i})$,
$y_{i}\neq0$ in $\mathcal{O}_{K}^{\flat}$, set 
\[
M_{i,j}=M_{i-1}+M'_{j}\cap M_{i}\quad\mbox{and}\quad M'_{j,i}=M'_{j-1}+M_{i}\cap M'_{j}
\]
Then $j\mapsto\overline{M}_{i,j}=M_{i,j}/M_{i-1}$ and $i\mapsto\overline{M}_{j,i}=M'_{j,i}/M'_{j-1}$
are good filtrations on $M_{i}/M_{i-1}$ and $M'_{j}/M'_{j-1}$ respectively,
with 
\[
\overline{M}_{i,j}/\overline{M}_{i,j-1}\simeq\frac{M'_{j}\cap M_{i}}{M'_{j}\cap M_{i-1}+M'_{j-1}\cap M_{i}}\simeq\overline{M}_{j,i}/\overline{M}_{j,i-1}.
\]
It is therefore sufficient to treat the case where $M=\mathcal{O}_{K}^{\flat}(x)$
for some nonzero $x\in\mathcal{O}_{K}^{\flat}$, which follows from
lemma~\ref{lem:InvTorsOModAndExactSeq}. We thus obtain a generalized
length function, 
\[
\length_{\mathfrak{m}^{\infty}}:\sk\,\Mod_{A,\mathfrak{m}^{\infty}}\rightarrow\mathbb{R}_{+},\qquad\length_{\mathfrak{m}^{\infty}}(M)\eqd-\log_{q}\left|\delta(M)\right|
\]
which is plainly additive on short exact sequences in $\Mod_{A,\mathfrak{m}^{\infty}}$.
Here $\left|\delta\right|=\left|x\right|$ if $\delta=\mathcal{O}_{K}^{\flat}x$.
Note also that $\length_{\mathfrak{m}^{\infty}}(M)=0$ if and only
if $M=0$.

\subsubsection{~\label{sub:relatinglength2rank}}

For every $a\in A\setminus\pi A$, $M\mapsto M/aM$ is an exact functor
from $\Mod_{A,t}$ to $\Mod_{A,\mathfrak{m}^{\infty}}$ which maps
$M=\mathcal{O}_{K}^{\flat}$ to $\mathcal{O}_{K}^{\flat}(\overline{a})=\mathcal{O}_{K}^{\flat}/\overline{a}\mathcal{O}_{K}^{\flat}$,
with $\overline{a}=a\bmod\pi\in\mathcal{O}_{K}^{\flat}$. It follows
that for every $M\in\Mod_{A,t}$, we have the following formula: 
\[
\length_{\mathfrak{m}^{\infty}}\left(M/aM\right)=-\log_{q}\left|\overline{a}\right|\cdot\rank_{t}(M).
\]
Since $\log_{q}\left|\overline{\xi'}\right|=\log_{q}\left|\varpi^{q}\right|=-1$,
we obtain another formula for the rank on $\Mod_{A,t}$:
\[
\rank_{t}(M)=\length_{\mathfrak{m}^{\infty}}\left(M/\xi'M\right)\quad\mbox{in}\quad\mathbb{N}\subset\mathbb{R}_{+}.
\]

\subsubsection{~}

The functor $M\mapsto M[\frac{1}{\pi}]$ extends to the isogeny categories,
\[
-[{\textstyle \frac{1}{\pi}]:}\Mod_{A,f}\otimes E\rightarrow\Mod_{A,\ast}\otimes E\rightarrow\Mod_{A[\frac{1}{\pi}]}.
\]
The functor $\Mod_{A,f}\otimes E\rightarrow\Mod_{A,\ast}\otimes E$
is an equivalence of categories, with inverse induced by $M\mapsto M_{f}$.
The functor $\Mod_{A,\ast}\otimes E\rightarrow\Mod_{A[\frac{1}{\pi}]}$
is fully faithful with essential image the full subcategory $\Mod_{A[\frac{1}{\pi}],f}$
of finite free $A[\frac{1}{\pi}]$-modules.

\subsection{Categories of $\varphi$-$A$-modules}

\subsubsection{~}

Let $\Mod_{A}^{\varphi}$ be the category of $A$-modules $M$ equipped
with an $A[\xi^{\prime-1}]$-linear isomorphism $\varphi_{M}:(\varphi^{\ast}M)[\xi^{\prime-1}]\rightarrow M[\xi^{\prime-1}]$.
A morphism $(M_{1},\varphi_{1})\rightarrow(M_{2},\varphi_{2})$ is
an $A$-linear morphism $f:M_{1}\rightarrow M_{2}$ such that the
following diagram is commutative:
\[
\xyR{2pc}\xymatrix{(\varphi^{\ast}M_{1})[\xi^{\prime-1}]\ar[r]^{\varphi^{\ast}f}\ar[d]^{\varphi_{1}} & (\varphi^{\ast}M_{2})[\xi^{\prime-1}]\ar[d]^{\varphi_{2}}\\
M_{1}[\xi^{\prime-1}]\ar[r]^{f} & M_{2}[\xi^{\prime-1}]
}
\]
Its kernel and cokernels are given by $\left(\ker(f),\varphi_{1}^{\prime}\right)$
and $\left(\coker(f),\varphi_{2}^{\prime}\right)$ with
\[
\xymatrix{(\varphi^{\ast}\ker(f))[\xi^{\prime-1}]\ar@{^{(}->}[r]\ar[d]^{\varphi_{1}^{\prime}} & (\varphi^{\ast}M_{1})[\xi^{\prime-1}]\ar[r]^{\varphi^{\ast}f}\ar[d]^{\varphi_{1}} & (\varphi^{\ast}M_{2})[\xi^{\prime-1}]\ar[d]^{\varphi_{2}}\ar@{->>}[r] & (\varphi^{\ast}\coker(f))[\xi^{\prime-1}]\ar[d]^{\varphi_{2}^{\prime}}\\
\ker(f)[\xi^{\prime-1}]\ar@{^{(}->}[r] & M_{1}[\xi^{\prime-1}]\ar[r]^{f} & M_{2}[\xi^{\prime-1}]\ar@{->>}[r] & \coker[\xi^{\prime-1}]
}
\]
commutative. This makes sense since $M\mapsto M[\xi^{\prime-1}]$
and $M\mapsto\varphi^{\ast}M$ are exact. The category $\Mod_{A}^{\varphi}$
is abelian, and it is a $\otimes$-category: using the isomorphisms
\begin{eqnarray*}
\varphi^{\ast}(M_{1}\otimes_{A}M_{2})\left[\xi^{\prime-1}\right] & \simeq & \left(\varphi^{\ast}(M_{1})\left[\xi^{\prime-1}\right]\right)\otimes_{A\left[\xi^{\prime-1}\right]}\left(\varphi^{\ast}(M_{2})\left[\xi^{\prime-1}\right]\right),\\
\varphi^{\ast}\left(\Sym_{A}^{k}M\right)\left[\xi^{\prime-1}\right] & \simeq & \Sym_{A\left[\xi^{\prime-1}\right]}^{k}\left(\varphi^{\ast}\left(M\right)\left[\xi^{\prime-1}\right]\right),\\
\varphi^{\ast}\left(\Lambda_{A}^{k}M\right)\left[\xi^{\prime-1}\right] & \simeq & \Lambda_{A\left[\xi^{\prime-1}\right]}^{k}\left(\varphi^{\ast}\left(M\right)\left[\xi^{\prime-1}\right]\right),\\
\varphi^{\ast}(A)\left[\xi^{\prime-1}\right] & \simeq & A\left[\xi^{\prime-1}\right],
\end{eqnarray*}
the tensor product, symmetric and exterior powers, and neutral object
are 
\begin{eqnarray*}
(M_{1},\varphi_{1})\otimes(M_{2},\varphi_{2}) & \eqd & \left(M_{1}\otimes M_{2},\varphi_{1}\otimes\varphi_{2}\right),\\
\Sym^{k}(M,\varphi) & \eqd & \left(\Sym^{k}(M),\Sym^{k}(\varphi)\right),\\
\Lambda^{k}(M,\varphi) & \eqd & \left(\Lambda^{k}(M),\Lambda^{k}(\varphi)\right),\\
\mbox{and}\quad A & \eqd & (A,\mathrm{Id}).
\end{eqnarray*}

\subsubsection{~\label{sub:DefBKFMod}}

A \emph{Breuil-Kisin-Fargues module} or \emph{BKF-module} is an $A$-module
$M$ in $\Mod_{A,\ast}$ equipped with an $A[\xi^{\prime-1}]$-linear
isomorphism $\varphi_{M}:(\varphi^{\ast}M)[\xi^{\prime-1}]\rightarrow M[\xi^{\prime-1}]$.
This defines a strictly full subcategory $\Mod_{A,\ast}^{\varphi}$
of $\Mod_{A}^{\varphi}$. For $\star\in\{f,\pi^{\infty},\mathfrak{m}^{\infty},t\}$,
we denote by $\Mod_{A,\star}^{\varphi}$ the strictly full subcategory
of $\Mod_{A,\ast}^{\varphi}$ of all BKF-modules $(M,\varphi_{M})$
whose underlying $A$-module $M$ lies in the strictly full subcategory
$\Mod_{A,\star}$ of $\Mod_{A}$. Note that $\Mod_{A,\mathfrak{m}^{\infty}}^{\varphi}=\Mod_{A,\mathfrak{m}^{\infty}}$
since $M[\xi^{\prime-1}]=0$ for $M\in\Mod_{A,\mathfrak{m}^{\infty}}$.

Since $M\mapsto\varphi^{\ast}M[\xi^{\prime-1}]$ is exact, the functorial
\emph{dévissage} of objects in $\Mod_{A,\ast}$ yields an analogous
functorial \emph{dévissage} for any BKF-modules $(M,\varphi_{M})$
in $\Mod_{A,\ast}^{\varphi}$,\emph{
\[
\xyR{1pc}\xyC{1.5pc}\xymatrix{ & \left(M[\mathfrak{m}^{\infty}],\varphi_{M[\mathfrak{m}^{\infty}]}\right)\ar@{^{(}->}[d]\\
0\ar[r] & \left(M[\pi^{\infty}],\varphi_{M[\pi^{\infty}]}\right)\ar[r]\ar@{->>}[d] & \left(M,\varphi_{M}\right)\ar[r] & \left(M_{f},\varphi_{f}\right)\ar[r] & \left(\overline{M},\varphi_{\overline{M}}\right)\ar[r] & 0\\
 & \left(M_{t},\varphi_{M_{t}}\right)
}
\]
}with everyone in the relevant strictly full subcategory.

\subsubsection{~}

The categories $\Mod_{A,\pi^{\infty}}^{\varphi}$ and $\Mod_{A,\mathfrak{m}^{\infty}}^{\varphi}$
are weak Serre subcategories of $\Mod_{A}^{\varphi}$: they are stable
under kernels, cokernels and extensions. In particular, they are both
abelian. The category $\Mod_{A,t}^{\varphi}$ is also stable under
extensions and kernels in $\Mod_{A}^{\varphi}$, but it is only quasi-abelian.
This last statement now requires some argument, given below: for every
$M\in\Mod_{A,\pi^{\infty}}^{\varphi}$, the exact sequence 
\[
0\rightarrow M[\mathfrak{m}^{\infty}]\rightarrow M\rightarrow M_{t}\rightarrow0
\]
yields a torsion theory on the abelian category $\Mod_{A,\pi^{\infty}}^{\varphi}$
with torsion class $\Mod_{A,\mathfrak{m}^{\infty}}^{\varphi}$ and
torsion-free class $\Mod_{A,t}^{\varphi}$, but we do not know whether
this torsion theory is cotilting (is every object $M$ of of $\Mod_{A,\pi^{\infty}}^{\varphi}$
a quotient of some $N$ in $\Mod_{A,t}^{\varphi}$?), and thus we
can not appeal to the criterion of \cite[B.3]{BoVdB03} for quasi-abelian
categories, as we did for $\Mod_{A,t}$ in~\ref{sub:ModAtquasiab}.
Plainly, kernels and coimages in $\Mod_{A,t}^{\varphi}$ are the corresponding
kernels and coimages in $\Mod_{A}^{\varphi}$. The image and cokernel
of a morphism $f:(M,\varphi_{M})\rightarrow(N,\varphi_{N})$ in $\Mod_{A,t}^{\varphi}$
are respectively equal to 
\[
\left(f(M)^{\mathrm{sat}},\varphi_{f(M)^{\mathrm{sat}}}\right)\quad\mbox{and}\quad\left(\left(N/f(M)\right)_{t},\varphi_{(N/f(M))_{t}}\right)
\]
where the Frobeniuses are induced by $\varphi_{N}:\varphi^{\ast}(N)[\xi^{\prime-1}]\rightarrow N[\xi^{\prime-1}]$
on respectively 
\begin{eqnarray*}
\varphi^{\ast}\left(f(M)^{\mathrm{sat}}\right)[\xi^{\prime-1}] & = & \varphi^{\ast}\left(f(M)\right)[\xi^{\prime-1}]\\
\quad\mbox{and}\quad\varphi^{\ast}\left(\left(N/f(M)\right)_{t}\right)[\xi^{\prime-1}] & = & \left(\varphi^{\ast}\left(N\right)/\varphi^{\ast}\left(f(M)\right)\right)[\xi^{\prime-1}].
\end{eqnarray*}
Such a morphism is strict if and only if $N/f(M)$ has no $\mathfrak{m}^{\infty}$-torsion,
or $[\varpi]$-torsion. It is a monomorphism (resp.~an epimorphism)
if and only if $f:M\rightarrow N$ is injective (resp.~$N/f(M)$
is killed by $[\varpi]^{n}$ for $n\gg0$). It is a strict monomorphism
(resp.~a strict epimorphism) if and only if $f:M\rightarrow N$ is
injective and $N/f(M)$ has no $[\varpi]$-torsion (resp.~$f:M\rightarrow N$
is surjective). We have to show that these classes of morphisms are
respectively stable under arbitrary push-outs and pull-backs: this
follows from the analogous properties of the quasi-abelian category
$\Mod_{A,t}$, since the forgetful functor $\Mod_{A,t}^{\varphi}\rightarrow\Mod_{A,t}$
is strongly exact (i.e.~commutes with kernels and cokernels). Finally,
short exact sequences in $\Mod_{A,t}^{\varphi}$ remain exact in $\Mod_{A}^{\varphi}$.

\subsubsection{~}

Any BKF-module $(M,\varphi_{M})$ in $\Mod_{A,t}^{\varphi}$ has a
canonical functorial filtration by strict subobjects $(M[\pi^{n}],\varphi_{M[\pi^{n}]})$
such that $M[\pi^{n}]/M[\pi^{n-1}]\simeq\pi^{n-1}M[\pi^{n}]$ is a
finite free $\mathcal{O}_{K}^{\flat}$-module, and conversely, any
$(M,\varphi_{M})$ in $\Mod_{A}^{\varphi}$ which is a successive
extension of such BKF-modules belongs to $\Mod_{A,t}^{\varphi}$.

\subsubsection{~}

The categories $\Mod_{A,\ast}^{\varphi}$, $\Mod_{A,f}^{\varphi}$,
$\Mod_{A,\mathfrak{m}^{\infty}}^{\varphi}$ and $\Mod_{A,\pi^{\infty}}^{\varphi}$
are stable under the tensor product, symmetric and exterior powers
of $\Mod_{A}^{\varphi}$. The isomorphism 
\[
\varphi^{\ast}\left(\Hom_{A}(M_{1},M_{2})\right)\left[\xi^{\prime-1}\right]\simeq\Hom_{A[\xi^{-1}]}\left(\left(\varphi^{\ast}M_{1}\right)\left[\xi^{\prime-1}\right],\left(\varphi^{\ast}M_{2}\right)\left[\xi^{\prime-1}\right]\right)
\]
which is valid for any finitely presented $M_{1}$ also yields an
internal Hom, 
\[
\Hom\left(\left(M_{1},\varphi_{1}\right),\left(M_{2},\varphi_{2}\right)\right)\eqd\left(\Hom_{A}(M_{1},M_{2}),\Hom_{A[\xi^{\prime-1}]}\left(\varphi_{1}^{-1},\varphi_{2}\right)\right)
\]
on any of these categories. The subcategory $\Mod_{A,t}^{\varphi}$
of $\Mod_{A,\ast}^{\varphi}$ is stable under this internal Hom, but
it is not stable under the tensor product. As for $\Mod_{A,t}$, there
is a modified tensor product $(M_{1},\varphi_{1})\otimes_{t}(M_{2},\varphi_{2}):=(M_{1}\otimes_{t}M_{2},\varphi_{1}\otimes_{t}\varphi_{2})$
which turns $\Mod_{A,t}^{\varphi}$ into a genuine $\otimes$-category
with internal Hom's.

\subsubsection{~}

There is a Tate object $A\{1\}=\left(A\{1\},\varphi_{A\{1\}}\right)$
in $\Mod_{A,f}^{\varphi}$, defined in~\cite[4.24]{BaMoSc16}. The
$A$-module $A\{1\}$ is free of rank $1$, and $\varphi_{A\{1\}}:\varphi^{\ast}(A\{1\})[\xi^{\prime-1}]\rightarrow A\{1\}[\xi^{\prime-1}]$
maps $\varphi^{\ast}A\{1\}$ to $\xi^{\prime-1}A\{1\}$. For any BKF-module
$M$ and $n\in\mathbb{Z}$ we set 
\[
M\{n\}\eqd M\otimes A\{n\}\quad\mbox{with}\quad A\{n\}\eqd\begin{cases}
A\{1\}^{\otimes n} & \mbox{if }n\geq0,\\
A\{-n\}^{\vee} & \mbox{if }n\leq0.
\end{cases}
\]
If $M[\mathfrak{m}^{\infty}]=0$, then $M$ and $\varphi^{\ast}M$
have no $\xi^{\prime}$-torsion, thus $M\subset M[\xi^{\prime-1}]$
and $\varphi^{\ast}M\subset\varphi^{\ast}M[\xi^{\prime-1}]$. We then
say that $M$ is effective if $\varphi_{M}(\varphi^{\ast}M)\subset M$.
Plainly, 
\[
M\{-n\}\mbox{ is effective for every }n\gg0.
\]

\subsection{The Fargues filtration on $\Mod_{A,t}^{\varphi}$\label{sub:FarguesOnModAt}}

\subsubsection{~}

The rank function on $\Mod_{A,t}$ yields a rank function on $\Mod_{A,t}^{\varphi}$,
\[
\rank_{t}:\sk\,\Mod_{A,t}^{\varphi}\rightarrow\mathbb{N},\qquad\rank_{t}(M,\varphi_{M})\eqd\rank_{t}(M).
\]
In addition, the length function on $\Mod_{A,\mathfrak{m}^{\infty}}$
yields a degree function on $\Mod_{A,t}^{\varphi}$, 
\[
\deg_{t}:\sk\,\Mod_{A,t}^{\varphi}\rightarrow\mathbb{R}
\]
which is defined as follows. For every $(M,\varphi_{M})\in\Mod_{A,t}^{\varphi}$
and $n\gg0$, $M\{-n\}$ is effective and $\varphi_{M\{-n\}}$ maps
$\varphi^{\ast}M\{-n\}$ injectively into $M\{-n\}$ with cokernel
\[
Q^{n}(M,\varphi_{M})\simeq M/\xi^{\prime n}\varphi_{M}\left(\varphi^{\ast}M\right),\qquad Q^{n}(M,\varphi_{M})\in\Mod_{A,\mathfrak{m}^{\infty}}.
\]
From the short exact sequences 
\[
0\rightarrow Q^{n}(M,\varphi_{M})\stackrel{\xi^{\prime}}{\longrightarrow}Q^{n+1}(M,\varphi_{M})\rightarrow M/\xi'M\rightarrow0
\]
and \ref{sub:relatinglength2rank}, we thus obtain that 
\[
\deg_{t}(M,\varphi_{M})\eqd n\,\rank_{t}(M)-\length_{\mathfrak{m}^{\infty}}Q^{n}(M,\varphi_{M})
\]
does not depend upon $n\gg0$. Plainly, 
\[
\deg_{t}(M,\varphi_{M})\{n\}=\deg_{t}(M,\varphi_{M})+n\,\rank_{t}(M)
\]
for every $n\in\mathbb{Z}$. A short exact sequence 
\[
0\rightarrow(M_{1},\varphi_{M_{1}})\rightarrow(M_{2},\varphi_{M_{2}})\rightarrow(M_{3},\varphi_{M_{3}})\rightarrow0
\]
in $\Mod_{A,t}^{\varphi}$ yields, for every $n\gg0$, a commutative
diagram
\[
\xyC{1pc}\xymatrix{ & 0\ar[d] & 0\ar[d] & 0\ar[d]\\
0\ar[r] & \varphi^{\ast}M_{1}\{-n\}\ar[r]\ar[d] & \varphi^{\ast}M_{2}\{-n\}\ar[r]\ar[d] & \varphi^{\ast}M_{3}\{-n\}\ar[r]\ar[d] & 0\\
0\ar[r] & M_{1}\{-n\}\ar[r]\ar[d] & M_{2}\{-n\}\ar[r]\ar[d] & M_{3}\{-n\}\ar[r]\ar[d] & 0\\
0\ar[r] & Q^{n}(M_{1},\varphi_{M_{1}})\ar[r]\ar[d] & Q^{n}(M_{2},\varphi_{M_{2}})\ar[r]\ar[d] & Q^{n}(M_{3},\varphi_{M_{3}})\ar[r]\ar[d] & 0\\
 & 0 & 0 & 0
}
\]
 with exact rows and columns, from which easily follows that 
\[
\deg_{t}(M_{2},\varphi_{M_{2}})=\deg_{t}(M_{1},\varphi_{M_{1}})+\deg_{t}(M_{3},\varphi_{M_{3}}).
\]
Similarly a mono-epi $f:(M,\varphi_{M})\rightarrow(N,\varphi_{N})$
in $\Mod_{A,t}^{\varphi}$ yields an exact sequence 
\[
0\rightarrow\ker(\varphi^{\ast}Q\rightarrow Q)\rightarrow Q^{n}(M,\varphi_{M})\rightarrow Q^{n}(N,\varphi_{N})\rightarrow\coker(\varphi^{\ast}Q\rightarrow Q)\rightarrow0
\]
in $\Mod_{A,\mathfrak{m}^{\infty}}$, where $Q=N/f(M)$ and $\varphi^{\ast}Q\rightarrow Q$
is induced by $\xi^{\prime n}\varphi_{N}$. Thus
\begin{eqnarray*}
\deg_{t}(N,\varphi_{N})-\deg_{t}(M,\varphi_{M}) & = & \length_{\mathfrak{m}^{\infty}}\varphi^{\ast}Q-\length_{\mathfrak{m}^{\infty}}Q\\
 & = & (q-1)\cdot\length_{\mathfrak{m}^{\infty}}Q
\end{eqnarray*}
and $\deg_{t}(M,\varphi_{M})\leq\deg_{t}(N,\varphi_{N})$ with equality
if and only if $f$ is an isomorphism.

\subsubsection{~}

These rank and degree functions induce a Harder-Narasimhan theory
on the quasi-abelian category $\Mod_{A,t}^{\varphi}$~\cite{An09,Co16}.
A BKF-module $M$ in $\Mod_{A,t}^{\varphi}$ is semi-stable of slope
$\mu\in\mathbb{R}$ if and only if for every strict subobject $N$
of $M$, $\deg_{t}(N)\leq\mu\,\rank_{t}(N)$ with equality for $N=M$.
With this definition, the trivial BKF-module is semi-stable of slope
$\mu$ for all $\mu\in\mathbb{R}$. The semi-stable BKF-modules of
slope $\mu$ form an abelian full subcategory of $\Mod_{A,t}^{\varphi}$,
and every BKF-module $M$ in $\Mod_{A,t}^{\varphi}$ has a unique
decreasing $\mathbb{R}$-filtration $\mathcal{F}$ by strict subobjects
$\mathcal{F}^{\geq\gamma}$ with $\Gr_{\mathcal{F}}^{\gamma}:=\mathcal{F}^{\geq\gamma}/\mathcal{F}^{>\gamma}$
semi-stable of slope $\gamma$ for all $\gamma\in\mathbb{R}$, where
$\mathcal{F}^{>\gamma}:=\cup_{\gamma'>\gamma}\mathcal{F}^{\geq\gamma'}$.
We call $\mathcal{F}_{F}(M):=\mathcal{F}$ the Fargues filtration
of $M$. It depends functorially upon $M$ and there is no nonzero
morphism from $M_{1}$ to $M_{2}$ if $M_{1}$ and $M_{2}$ are semi-stable
of slope $\mu_{1}$ and $\mu_{2}<\mu_{1}$. The Fargues type of $M$
is the type $t_{F}(M)\in\mathbb{R}_{\geq}^{r}$ of $\mathcal{F}_{F}(M)$,
with $r=\rank_{t}(M)$. Finally, we denote by $\Gr_{F}^{\bullet}(M)$
the associated graded object in $\Mod_{A,t}^{\varphi}$.
\begin{prop}
\label{prop:PseudoIsogInModt}If $M_{1}\rightarrow M_{2}$ is a mono-epi
in $\Mod_{A,t}^{\varphi}$ with cokernel $Q$ in $\Mod_{A,\mathfrak{m}^{\infty}}^{\varphi}$
and $r=\rank_{t}M_{1}=\rank_{t}M_{2}$, then for every $s\in[0,r]$,
\[
0\leq t_{F}(M_{2})(s)-t_{F}(M_{1})(s)\leq(q-1)\cdot\length_{\mathfrak{m}^{\infty}}Q
\]
with equality on the left (resp.~right) for $s=0$ (resp.~$s=r$).
In particular, 
\[
0\leq\left\{ \begin{array}{c}
t_{F}^{\max}(M_{2})-t_{F}^{\max}(M_{1})\\
t_{F}^{\min}(M_{2})-t_{F}^{\min}(M_{1})
\end{array}\right\} \leq(q-1)\cdot\length_{\mathfrak{m}^{\infty}}Q.
\]
\end{prop}
\begin{proof}
Set $f_{i}=t_{F}(M_{i})$ and $(r_{i},d_{i})(\gamma)=(\rank_{t},\deg_{t})(\mathcal{F}_{F}^{\gamma}(M_{i}))$
for $\gamma\in\mathbb{R}$ and $i\in\{1,2\}$. It is sufficient to
show that for every $\gamma\in\mathbb{R}$, 
\[
d_{1}(\gamma)\leq f_{2}(r_{1}(\gamma))\quad\mbox{and}\quad d_{2}(\gamma)\leq f_{1}(r_{2}(\gamma))+(q-1)\cdot\length_{\mathfrak{m}^{\infty}}Q.
\]
For the first inequality, let $\mathcal{F}_{F}^{\gamma}(M_{1})^{\mathrm{sat}}$
be the image of $\mathcal{F}_{F}^{\gamma}(M_{1})$ in $M_{2}$. Then
\[
d_{1}(\gamma)=\deg_{t}\mathcal{F}_{F}^{\gamma}(M_{1})\leq\deg_{t}\mathcal{F}_{F}^{\gamma}(M_{1})^{\mathrm{sat}}\leq f_{2}(r_{1}(\gamma))
\]
since $\mathcal{F}_{F}^{\gamma}(M_{1})\rightarrow\mathcal{F}_{F}^{\gamma}(M_{1})^{\mathrm{sat}}$
is a mono-epi and $\mathcal{F}_{F}^{\gamma}(M_{1})^{\mathrm{sat}}$
is a strict subobject of rank $r_{1}(\gamma)$ in $M_{2}$. For the
second inequality, let $\mathcal{F}_{F}^{\gamma}(M_{2})^{\prime}$
and $Q^{\gamma}$ be respectively the kernel and image of $\mathcal{F}_{F}^{\gamma}(M_{2})\rightarrow Q$.
Then 
\begin{eqnarray*}
d_{2}(\gamma)=\deg_{t}\mathcal{F}_{F}^{\gamma}(M_{2}) & = & \deg_{t}\mathcal{F}_{F}^{\gamma}(M_{2})^{\prime}+(q-1)\cdot\length_{\mathfrak{m}^{\infty}}Q^{\gamma}\\
 & \leq & f_{1}(r_{2}(\gamma))+(q-1)\cdot\length_{\mathfrak{m}^{\infty}}Q
\end{eqnarray*}
since $\mathcal{F}_{F}^{\gamma}(M_{2})^{\prime}$ is a strict subobject
of rank $r_{2}(\gamma)$ in $M_{1}$ and $Q^{\gamma}\subset Q$. \end{proof}
\begin{prop}
\label{prop:ExactSeqInModt}Let $0\rightarrow M_{1}\rightarrow M_{2}\rightarrow M_{3}\rightarrow0$
be an exact sequence in $\Mod_{A,t}^{\varphi}$, set $r_{i}=\rank_{t}M_{i}$
and view $t_{F}(M_{i})$ as a concave function $f_{i}:[0,r_{i}]\rightarrow\mathbb{R}$.
Then 
\[
f_{1}\ast f_{3}(s)\geq f_{2}(s)\geq\begin{cases}
f_{1}(s) & \mbox{if }0\leq s\leq r_{1}\\
f_{1}(r_{1})+f_{3}(s-r_{1}) & \mbox{if }r_{1}\leq s\leq r_{2}
\end{cases}
\]
with equality for $s=0$ and $s=r_{2}$. In particular,\textup{
\[
\begin{array}{ccccc}
t_{F}^{\max}(M_{1}) & \leq & t_{F}^{\max}(M_{2}) & \leq & \max\left\{ t_{F}^{\max}(M_{1}),t_{F}^{\max}(M_{3})\right\} ,\\
t_{F}^{\min}(M_{3}) & \geq & t_{F}^{\min}(M_{2}) & \geq & \min\left\{ t_{F}^{\min}(M_{1}),t_{F}^{\min}(M_{3})\right\} ,
\end{array}
\]
\[
\mbox{and}\qquad\qquad t_{F}(M_{2})\leq t_{F}(M_{1})\ast t_{F}(M_{3})\quad\mbox{in}\quad\mathbb{R}_{\geq}^{r_{2}}.\qquad\qquad
\]
}\textup{\emph{Moreover, $t_{F}(M_{2})=t_{F}(M_{1})\ast t_{F}(M_{3})$
if and only if for every $\gamma\in\mathbb{R}$, 
\[
0\rightarrow\mathcal{F}_{F}^{\gamma}(M_{1})\rightarrow\mathcal{F}_{F}^{\gamma}(M_{2})\rightarrow\mathcal{F}_{F}^{\gamma}(M_{3})\rightarrow0
\]
is exact.}}\end{prop}
\begin{proof}
These are standard properties of Harder-Narasimhan filtrations on
quasi-abelian categories, see for instance~\cite[Proposition 21]{Co16}
or \cite[4.4.4]{An09}. \end{proof}
\begin{prop}
\label{prop:tF(twist)onModt}For every $M\in\Mod_{A,t}^{\varphi}$
of rank $r\in\mathbb{N}$ and any $n\in\mathbb{Z}$, 
\[
\mathcal{F}_{F}^{\gamma}(M\{n\})=\mathcal{F}_{F}^{\gamma-n}(M)\{n\}
\]
for every $\gamma\in\mathbb{R}$, hence 
\[
t_{F}(M\{n\})=t_{F}(M)+(n,\cdots,n)\quad\mbox{in}\quad\mathbb{R}_{\geq}^{r}.
\]
\end{prop}
\begin{proof}
This is obvious: the map $N\mapsto N\{n\}$ induces a bijection between
strict subobjects of $M$ and strict subobjects of $M\{n\}$, with
$\mu(N\{n\})=\mu(N)+n$. 
\end{proof}

\subsubsection{~\label{sub:ExactSeqInModpi}}

For $M\in\Mod_{A,\pi^{\infty}}^{\varphi}$, we set $t_{F}(M)=t_{F}(M_{t})$.
An exact sequence 
\[
0\rightarrow M_{1}\rightarrow M_{2}\rightarrow M_{3}\rightarrow0
\]
in $\Mod_{A,\pi^{\infty}}^{\varphi}$ gives rise to three exact sequences:
\[
0\rightarrow M_{1}[\mathfrak{m}^{\infty}]\rightarrow M_{2}[\mathfrak{m}^{\infty}]\rightarrow M_{3}[\mathfrak{m}^{\infty}]\rightarrow Q\rightarrow0
\]
\[
0\rightarrow M_{4}\rightarrow M_{2,t}\rightarrow M_{3,t}\rightarrow0
\]
\[
0\rightarrow M_{1,t}\rightarrow M_{4}\rightarrow Q\rightarrow0
\]
with $Q\in\Mod_{A,\mathfrak{m}^{\infty}}^{\varphi}$ and $M_{4}\in\Mod_{A,t}^{\varphi}$.
Set $\ell_{i}=\length_{\mathfrak{m}^{\infty}}M_{i}[\mathfrak{m}^{\infty}]$,
$r_{i}=\rank_{t}(M_{i})$, $f_{i}=t_{F}(M_{i})$ and $\ell_{Q}=\length_{\mathfrak{m}^{\infty}}Q$.
We thus have the following relations: 
\[
r_{1}=r_{4},\quad r_{1}+r_{3}=r_{2},\quad\ell_{Q}=\ell_{1}-\ell_{2}+\ell_{3}
\]
\[
f_{4}\ast f_{3}(s)\geq f_{2}(s)\geq\begin{cases}
f_{4}(s) & \mbox{for }0\leq s\leq r_{1},\\
f_{4}(r_{1})+f_{3}(s-r_{1}) & \mbox{for }r_{1}\leq s\leq r_{2}.
\end{cases}
\]
\[
f_{1}(s)\leq f_{4}(s)\leq f_{1}(s)+(q-1)\ell_{Q}\quad\mbox{for }0\leq s\leq r_{1}.
\]
Set $c_{i}=\max\left\{ \left|t_{F}^{\min}(M_{i})\right|,\left|t_{F}^{\max}(M_{i})\right|\right\} $
so that $f_{i}$ is $c_{i}$-Lipschitzian and
\[
c_{2}\leq\max\left\{ c_{3},c_{4}\right\} ,\quad\left|c_{1}-c_{4}\right|\leq(q-1)\ell_{Q}.
\]
Moreover, we have
\[
f_{4}\ast f_{3}(s)\leq\begin{cases}
f_{4}(s)+c_{4}r_{3}+\max(f_{3}) & \mbox{for }0\leq s\leq r_{1},\\
f_{3}(s-r_{1})+c_{3}r_{1}+\max(f_{4}) & \mbox{for }r_{1}\leq s\leq r_{2}
\end{cases}
\]
We obtain the following inequalities: for $0\leq s\leq r_{1}$, 
\[
0\le f_{2}(s)-f_{1}(s)\leq\left(c_{1}+c_{3}\right)r_{3}+(q-1)(\ell_{1}+\ell_{3})(r_{3}+1)
\]
and for $r_{1}\leq s\leq r_{2}$, 
\[
-c_{1}r_{1}\leq f_{2}(s)-f_{3}(s-r_{1})\leq\left(c_{1}+c_{3}+(q-1)(\ell_{1}+\ell_{3})\right)r_{1}
\]
which also implies that for $0\leq s\leq r_{3}$, 
\[
\left|f_{2}(s)-f_{3}(s)\right|\leq\max\left\{ \begin{array}{c}
c_{1}+2c_{3}+(q-1)(\ell_{1}+\ell_{3}),\\
2c_{1}+c_{3}+3(q-1)(\ell_{1}+\ell_{3})
\end{array}\right\} \cdot r_{1}.
\]

\subsection{The Fargues type on $\Mod_{A,\ast}^{\varphi}$\label{sub:FarguesTypeBKF}}

\subsubsection{~}

For any $M\in\Mod_{A}^{\varphi}$ and $n\geq1$, consider the exact
sequence
\[
\xyC{2pc}\xymatrix{0\ar[r] & M[\pi^{n}]\ar[r] & M\ar[r]^{\pi^{n}} & M\ar[r] & M_{n}\ar[r] & 0}
.
\]
Suppose that $M$ is a BKF-module, i.e.~belongs to $\Mod_{A,\ast}^{\varphi}$.
Then $M_{n}$ and $M[\pi^{n}]$ both belong to $\Mod_{A,\pi^{\infty}}^{\varphi}$.
Moreover, $\rank_{t}M_{n}\geq n\,\rank_{A}M$ by \ref{sub:rankfree/tors}.
Viewing $t_{F}(M_{n})$ as a concave function on $[0,\rank_{t}M_{n}]$,
we may thus define 
\[
t_{F,n}(M):[0,\rank_{A}M]\rightarrow\mathbb{R},\qquad t_{F,n}(M)(s)={\textstyle \frac{1}{n}}t_{F}(M_{n})(ns).
\]

\begin{prop}
\label{prop:DeftFinfty}There is a constant $C(M)$ such that the
functions $t_{F,n}(M)$ are $C(M)$-Lipschitzian. They converge uniformly
to a continuous concave function 
\[
t_{F,\infty}(M):[0,\rank_{A}M]\rightarrow\mathbb{R}.
\]
If $M_{1},M_{2}\in\Mod_{A,\ast}^{\varphi}$ become isomorphic in the
isogeny category $\Mod_{A,\ast}^{\varphi}\otimes E$, then 
\[
t_{F,\infty}(M_{1})=t_{F,\infty}(M_{2}).
\]
\end{prop}
\begin{proof}
Let $r=\rank_{A}M=\rank_{A}M_{f}$ and set $f_{n}=f_{n}(M)=t_{F,n}(M)$. 

Suppose first that $M$ is free. Then for every $n,m\geq1$, the exact
sequence 
\[
0\rightarrow M_{n}\stackrel{\pi^{m}}{\longrightarrow}M_{n+m}\rightarrow M_{m}\rightarrow0
\]
in $\Mod_{A,t}^{\varphi}$ gives the inequality 
\[
t_{F}(X_{n+m})\leq t_{F}(X_{n})\ast t_{F}(X_{m})\quad\mbox{in}\quad\mathbb{R}_{\geq}^{r}.
\]
It follows that for every $n,k\geq1$ and $0\leq s\leq r$, 
\[
f_{nk}(s)\leq f_{n}(s)\quad\mbox{with equality for }s\in\{0,r\}.
\]
In particular, $f_{n}(s)\leq f_{1}(s)$ with equality for $s\in\{0,r\}$,
and the slopes of the continuous piecewise linear functions $f_{n}$
are uniformly bounded by the constant 
\[
C=C(M)=\max\left\{ \left|t_{F}^{\min}(M_{1})\right|,\left|t_{F}^{\max}(M_{1})\right|\right\} .
\]
Fix $n_{0},n\geq1$. For $n=n_{0}q_{n}+r_{n}$ with $q_{n}\geq0$
and $0\leq r_{n}<n_{0}$, we have 
\[
t_{F}(X_{n})\leq t_{F}(X_{n_{0}q_{n}})\ast t_{F}(X_{r_{n}})\leq t_{F}(X_{n_{0}})^{\ast^{q_{n}}}\ast t_{F}(X_{r_{n}})
\]
from which we obtain that for $0\leq s\leq r$, 
\[
f_{n}(s)\leq(1-{\textstyle \frac{r_{n}}{n})}f_{n_{0}}(s^{\prime})+{\textstyle \frac{r_{n}}{n}}f_{r_{n}}(s^{\prime\prime})
\]
for some $s^{\prime},s^{\prime\prime}\in[0,r]$ with $n_{0}q_{n}s^{\prime}+r_{n}s^{\prime\prime}=ns$.
But then $s^{\prime}-s=\frac{r_{n}}{n}(s^{\prime}-s^{\prime\prime})$,
thus 
\[
f_{n}(s)\leq(1-{\textstyle \frac{r_{n}}{n})}f_{n_{0}}(s)+{\textstyle \frac{r_{n}}{n}}\left(2rC{\textstyle (1-\frac{r_{n}}{n})}+\sup(f_{1})\right).
\]
Therefore $\limsup f_{n}(s)\leq f_{n_{0}}(s)$ and this being true
for all $n_{0}\geq1$, 
\[
\limsup f_{n}(s)\leq\liminf f_{n}(s)
\]
i.e.~$f_{n}(s)$ converges to some limit $f_{\infty}(s)\in\mathbb{R}$.
Since all the $f_{n}$'s are $C$-Lipschitzian concave, so is $f_{\infty}=f_{\infty}(M)$
and the convergence is uniform. 

Suppose next that $M$ is torsion free, so that $0\rightarrow M\rightarrow M_{f}\rightarrow\overline{M}\rightarrow0$
is exact and for $n\gg0$ (such that $\pi^{n}\overline{M}=0$), we
obtain an exact sequence 
\[
0\rightarrow\overline{M}\rightarrow M_{n}\rightarrow M_{f,n}\rightarrow\overline{M}\rightarrow0
\]
which identifies $\overline{M}$ and $M_{n}[\mathfrak{m}^{\infty}]$
(since $M_{f,n}[\mathfrak{m}^{\infty}]=0$), i.e.
\[
0\rightarrow M_{n,t}\rightarrow M_{f,n}\rightarrow\overline{M}\rightarrow0
\]
is exact. Our claim now follows from proposition~\ref{prop:PseudoIsogInModt},
with the constant 
\[
C(M)=C(M_{f})+(q-1)\length_{\mathfrak{m}^{\infty}}\overline{M}
\]
and the limit $f_{\infty}(M)=f_{\infty}(M_{f})$.

For the general case, let $I$ be the image of $M\rightarrow M_{f}$,
so that $I$ is a torsion free BKF-module. This time for $n\gg0$,
we have an exact sequence 
\[
0\rightarrow M[\pi^{\infty}]\rightarrow M_{n}\rightarrow I_{n}\rightarrow0.
\]
We have just seen that $I_{n}[\mathfrak{m}^{\infty}]=\overline{I}=\overline{M}$
for $n\gg0$. Our claim now follows from the discussion of section~\ref{sub:ExactSeqInModpi}
with the constant 
\[
C(M)=\max\left\{ C(I),C(M_{t})+(q-1)\length_{\mathfrak{m}^{\infty}}M[\mathfrak{m}^{\infty}]\oplus\overline{M}\right\} 
\]
and the limit $f_{\infty}(M)=f_{\infty}(I)=f_{\infty}(M_{f})$. Here
\begin{eqnarray*}
C(I) & = & \max\left\{ \left|t_{F}^{\min}(M_{f,1})\right|,\left|t_{F}^{\max}(M_{f,1})\right|\right\} +(q-1)\length_{\mathfrak{m}^{\infty}}\overline{M},\\
C(M_{t}) & = & \max\left\{ \left|t_{F}^{\min}(M_{t})\right|,\left|t_{F}^{\max}(M_{t})\right|\right\} .
\end{eqnarray*}

It remains to establish that $M\mapsto t_{F,\infty}(M)$ is constant
on isogeny classes, and we already know that $t_{F,\infty}(M)=t_{F,\infty}(M_{f})$.
We thus have to show that if 
\[
0\rightarrow M_{1}\rightarrow M_{2}\rightarrow Q\rightarrow0
\]
is an exact sequence in $\Mod_{A}^{\varphi}$ with $M_{1}$, $M_{2}$
finite free and $Q$ torsion, then $t_{F,\infty}(M_{1})$ equals $t_{F,\infty}(M_{2})$.
For $n\gg0$ (such that $\pi^{n}Q=0$), we obtain exact sequences
\[
0\rightarrow Q\rightarrow M_{1,n}\rightarrow M_{2,n}\rightarrow Q\rightarrow0.
\]
Splitting them in two short exact sequences and using again the computations
of section~\ref{sub:ExactSeqInModpi} yields the desired equality. \end{proof}
\begin{prop}
\label{prop:tFinfty(twist)onMod*}For any BKF-module $M$ of rank
$r\in\mathbb{N}$ and any $n\in\mathbb{Z}$, 
\[
\forall s\in[0,r]:\qquad t_{F,\infty}(M\{n\})(s)=t_{F,\infty}(M)(s)+sn.
\]
\end{prop}
\begin{proof}
This follows from proposition~\ref{prop:tF(twist)onModt}.
\end{proof}

\subsubsection{~\label{sub:TypeHN}}

The first part of the proof of proposition~\ref{prop:DeftFinfty}
shows that
\begin{prop}
\label{prop:ComptFinftytFn}For a finite free BKF-module $M$ of rank
$r\in\mathbb{N}$, 
\[
t_{F,\infty}(M)(s)\leq t_{F,n}(M)(s)\leq t_{F,1}(M)(s)
\]
for every $n\geq1$ and $s\in[0,r]$ with equality for $s\in\{0,r\}$. \end{prop}
\begin{defn}
\label{Def:HNType}We say that a finite free BKF-module $M$ is of
HN-type if 
\[
t_{F,\infty}(M)=t_{F,1}(M).
\]

\end{defn}
\noindent Thus if $M$ is of HN-type and rank $r\in\mathbb{N}$,
then $t_{F,\infty}(M)\in\mathbb{R}_{\geq}^{r}$.
\begin{prop}
\label{prop:TypeHN}Let $M$ be a finite free BKF-module of HN-type.
Then
\begin{enumerate}
\item For every $\gamma\in\mathbb{R}$ and $n,m\geq1$, the exact sequence
\[
\xymatrix{0\ar[r] & M_{n}\ar[r]^{\pi^{m}} & M_{n+m}\ar[r] & M_{m}\ar[r] & 0}
\]
induces an exact sequence
\[
\xymatrix{0\ar[r] & \mathcal{F}_{F}^{\gamma}(M_{n})\ar[r]^{\pi^{m}} & \mathcal{F}_{F}^{\gamma}(M_{n+m})\ar[r] & \mathcal{F}_{F}^{\gamma}(M_{m})\ar[r] & 0}
\]

\item The formula $\mathcal{F}_{F}^{\gamma}(M)=\underleftarrow{\lim}\mathcal{F}_{F}^{\gamma}(M_{n})$
defines an $\mathbb{R}$-filtration on $M$ by finite free BKF-submodules
whose underlying $A$-submodules are direct summands: the quotient
$\Gr_{F}^{\gamma}(M)=\mathcal{F}_{F}^{\geq\gamma}(M)/\mathcal{F}_{F}^{>\gamma}(M)$
is a finite free BKF-module. 
\item For every $\gamma\in\mathbb{R}$ and $n\geq1$, 
\[
\mathcal{F}_{F}^{\gamma}(M)_{n}=\mathcal{F}_{F}^{\gamma}(M_{n})\quad\mbox{and}\quad\Gr_{F}^{\gamma}(M)_{n}=\Gr_{F}^{\gamma}(M_{n}).
\]
In particular, the type of the $\mathbb{R}$-filtration $\mathcal{F}_{F}^{\bullet}(M)$
is given by 
\[
\mathbf{t}\left(\mathcal{F}_{F}^{\bullet}(M)\right)=t_{F}(M_{1})=t_{F,1}(M)=t_{F,\infty}(M).
\]

\end{enumerate}
\end{prop}
\begin{proof}
$(1)$ Since $t_{F,\infty}(M)=t_{F,1}(M)$, also $t_{F,n}(M)=t_{F,1}(M)$
for every $n\geq1$, thus 
\[
t_{F}(M_{n+m})=t_{F}(M_{n})\ast t_{F}(M_{m})
\]
for every $n,m\geq1$, from which $(1)$ immediately follows by proposition~\ref{prop:ExactSeqInModt}. 

$(2)$ and $(3)$: This follows from $(1)$ by a standard argument:
consider for $n,m\geq1$ and $\gamma\in\mathbb{R}$ the commutative
diagram with exact rows and columns
\[
\xymatrix{ & 0\ar[d] & 0\ar[d] & 0\ar[d]\\
0\ar[r] & \mathcal{F}_{F}^{\gamma}(M_{n})\ar[d]\ar[r]^{\pi^{m}} & \mathcal{F}_{F}^{\gamma}(M_{n+m})\ar[d]\ar[r] & \mathcal{F}_{F}^{\gamma}(M_{m})\ar[d]\ar[r] & 0\\
0\ar[r] & M_{n}\ar[d]\ar[r]^{\pi^{m}} & M_{n+m}\ar[d]\ar[r] & M_{m}\ar[d]\ar[r] & 0\\
0\ar[r] & \mathcal{G}_{F}^{\gamma}(M_{n})\ar[d]\ar[r]^{\pi^{m}} & \mathcal{G}_{F}^{\gamma}(M_{n+m})\ar[d]\ar[r] & \mathcal{G}_{F}^{\gamma}(M_{m})\ar[d]\ar[r] & 0\\
 & 0 & 0 & 0
}
\]
Taking the projective limit over $n$, and since every one is Mittag-Leffler
surjective, we obtain a commutative diagram of $A$-modules with exact
rows and columns
\[
\xymatrix{ & 0\ar[d] & 0\ar[d] & 0\ar[d]\\
0\ar[r] & \mathcal{F}_{F}^{\gamma}(M)\ar[d]\ar[r]^{\pi^{m}} & \mathcal{F}_{F}^{\gamma}(M)\ar[d]\ar[r] & \mathcal{F}_{F}^{\gamma}(M_{m})\ar[d]\ar[r] & 0\\
0\ar[r] & M\ar[d]\ar[r]^{\pi^{m}} & M\ar[d]\ar[r] & M_{m}\ar[d]\ar[r] & 0\\
0\ar[r] & \mathcal{G}_{F}^{\gamma}(M)\ar[d]\ar[r]^{\pi^{m}} & \mathcal{G}_{F}^{\gamma}(M)\ar[d]\ar[r] & \mathcal{G}_{F}^{\gamma}(M_{m})\ar[d]\ar[r] & 0\\
 & 0 & 0 & 0
}
\]
Since $\mathcal{F}_{F}^{\gamma}(M)=\underleftarrow{\lim}\,\mathcal{F}_{F}^{\gamma}(M_{n})$,
the first row tells us that $\mathcal{F}_{F}^{\gamma}(M)$ is separated
and complete in the $\pi$-adic topology, with $\mathcal{F}_{F}^{\gamma}(M)_{1}\simeq\mathcal{F}_{F}^{\gamma}(M_{1})$
finite free over $A_{1}=\mathcal{O}_{K}^{\flat}$, say of rank $s\in\mathbb{N}$.
Pick a morphism $\alpha:A^{s}\rightarrow\mathcal{F}_{F}^{\gamma}(M)$
reducing to an isomorphism modulo $\pi$. By the topological version
of Nakayama's lemma, $\alpha$ is surjective, and $\mathcal{F}_{F}^{\gamma}(M)$
is finitely generated over $A$. Playing the same game with the third
row, we obtain a surjective morphism $\beta:A^{s'}\twoheadrightarrow\mathcal{G}_{F}^{\gamma}(M)$
reducing to an isomorphism modulo $\pi$. But now the kernel $N$
of $\beta$ has to be finitely generated over $A$ since $\mathcal{G}_{F}^{\gamma}(M)$
is finitely presented over $A$ by the second column. Applying $\Tor_{\bullet}^{A}(-,\mathcal{O}_{K}^{\flat})$
to the resulting short exact sequence $0\rightarrow N\rightarrow A^{s^{\prime}}\rightarrow\mathcal{G}_{F}^{\gamma}(M)\rightarrow0$,
we find that 
\[
N\otimes\mathcal{O}_{K}^{\flat}\simeq\Tor_{1}^{A}\left(\mathcal{G}_{F}^{\gamma}(M),\mathcal{O}_{K}^{\flat}\right)\simeq\mathcal{G}_{F}^{\gamma}(M)[\pi],
\]
which is trivial by the third row, thus $N=0$ by the classical version
of Nakayama's lemma. It follows that $\beta$ is an isomorphism, $\mathcal{G}_{F}^{\gamma}(M)$
is free, the middle column is split (in $\Mod_{A}$), and $\mathcal{F}_{F}^{\gamma}(M)$
is also free, being finite projective over the local ring $A$. The
remaining assertions of $(2)$ and $(3)$ easily follow.\end{proof}
\begin{rem}
For a finite free BKF-module $M$ of HN-type and $n\in\mathbb{Z}$,
the Tate twist $M\{n\}$ is also of HN-type and $\mathcal{F}_{F}^{\gamma}(M\{n\})=\mathcal{F}_{F}^{\gamma-n}(M)\{n\}$
by proposition~\ref{prop:tF(twist)onModt}.\end{rem}
\begin{defn}
We say that a finite free BKF-module $M$ is semi-stable (of slope
$\gamma\in\mathbb{R}$) if $M_{1}$ is semi-stable (of slope $\gamma\in\mathbb{R}$).\end{defn}
\begin{example}
Any finite free BKF-module $M$ of rank $1$ is semi-stable of slope
$\deg_{t}(M_{1})$, thus $A$ is semi-stable of slope $0$ and $A\{1\}$
is semi-stable of slope $1$. 
\end{example}
By~proposition~\ref{prop:ComptFinftytFn}, a finite free BKF-module
$M$ is semi-stable (of slope $\gamma$) if and only if $M_{n}$ is
semi-stable (of slope $\gamma$) for every $n\geq1$, in which case
$M$ is of HN-type and $t_{F,\infty}(M)=t_{F,1}(M)=t_{F}(M_{1})$
is isoclinic (of slope $\gamma$). By proposition~\ref{prop:TypeHN},
a finite free BKF-module $M$ of HN-type has a canonical filtration
$\mathcal{F}_{F}(M)$ whose graded pieces are finite free semi-stable
BKF-modules with decreasing slopes. Conversely, any finite free BKF-module
which has such a filtration is of HN-type (by unicity of the Fargues
filtration on $\Mod_{A,t}^{\varphi}$) and its filtration is the canonical
one.

\subsubsection{~}

We denote by $\Mod_{A,f}^{\varphi,\ast}$ the strictly full subcategory
of $\Mod_{A,f}^{\varphi}$ whose objects are the finite free BKF-modules
of HN-type. The functoriality of the Fargues filtration on $\Mod_{A,t}^{\varphi}$
implies that $M\mapsto\mathcal{F}_{F}(M)$ is functorial on $\Mod_{A,f}^{\varphi,\ast}$.
 
\begin{prop}
\label{prop:catOfHNTypeTensor}The subcategory $\Mod_{A,f}^{\varphi,\ast}$
of $\Mod_{A,f}^{\varphi}$ is stable under $\otimes$-products and
inner Homs and the $\mathbb{R}$-filtration $\mathcal{F}_{F}$ on
$\Mod_{A,f}^{\varphi,\ast}$ is compatible with them.\end{prop}
\begin{proof}
The Fargues filtrations on $M_{1},M_{2}\in\Mod_{A,f}^{\varphi,\ast}$
induce $\mathbb{R}$-filtrations on $M_{1}\otimes M_{2}$ and $\Hom(M_{1},M_{2})$
whose graded pieces are the finite free BKF-modules 
\[
\oplus_{\gamma_{1}+\gamma_{2}=\gamma}\Gr_{F}^{\gamma_{1}}(M_{1})\otimes\Gr_{F}^{\gamma_{2}}(M_{2})\quad\mbox{and}\quad\oplus_{\gamma_{2}-\gamma_{1}=\gamma}\Hom(\Gr_{F}^{\gamma_{1}}(M_{1}),\Gr_{F}^{\gamma_{2}}(M_{2})).
\]
We thus have to show that if $M_{1}$ and $M_{2}$ are semi-stable
of slope $\gamma_{1}$ and $\gamma_{2}$, then $P=M_{1}\otimes M_{2}$
and $H=\Hom(M_{1},M_{2})$ are semi-stable of slope $\gamma_{1}+\gamma_{2}$
and $\gamma_{2}-\gamma_{1}$. Since $P_{1}=M_{1,1}\otimes M_{2,1}$
and $H_{1}=\Hom(M_{1,1},M_{2,1})$, we need to establish the analogous
statement for BKF-modules which are finite free over $A_{1}=\mathcal{O}_{K}^{\flat}$.
This is a special case of~\cite[\S 5.3]{Co16}, see also section~\ref{sub:RisOkflat}
below.
\end{proof}

\subsection{Categories of $\varphi$-$R$-modules}

For any $A$-algebra $R$ equipped with a ring isomorphism $\varphi:R\rightarrow R$
compatible with $\varphi:A\rightarrow A$, we may analogously define
the abelian $\otimes$-category $\Mod_{R}^{\varphi}$ and its full
$\otimes$-subcategories $\Mod_{R,\ast}^{\varphi}$ and $\Mod_{R,f}^{\varphi}$.
They come equipped with $\otimes$-functors $\Mod_{A,?}^{\varphi}\rightarrow\Mod_{R,?}^{\varphi}$
for $?\in\{\emptyset,\ast,f\}$, which are exact when $A\rightarrow R$
is flat. In this section, we discuss the following cases: 
\[
R\in\left\{ \mathcal{O}_{K}^{\flat},L,\mathcal{O}_{L},A(K),A\left[{\textstyle \frac{1}{\pi}}\right]\right\} .
\]

\subsubsection{$\boxed{R=\mathcal{O}_{K}^{\flat}}$.\label{sub:RisOkflat}}

In this case, $\Mod_{R,f}^{\varphi}$ is the full subcategory of $\Mod_{A,t}^{\varphi}$
made of all BKF-modules killed by $\pi$. This is the quasi-abelian
category of all finite free $\mathcal{O}_{K}^{\flat}$-modules $M$
equipped with an isomorphism $\varphi_{M}:\varphi^{\ast}M\otimes K^{\flat}\rightarrow M\otimes K^{\flat}$,
or equivalently, with a $\varphi$-semilinear isomorphism $\phi_{M}:M\otimes K^{\flat}\rightarrow M\otimes K^{\flat}$.
As a subcategory of $\Mod_{A,\ast}^{\varphi}$, it is stable under
tensor products, internal Homs, symmetric and exterior powers, and
it has a neutral object of its own. Using the isomorphisms 
\begin{eqnarray*}
\varphi^{\ast}(M_{1}\otimes M_{2})\otimes K^{\flat} & \simeq & \left(\varphi^{\ast}(M_{1})\otimes K^{\flat}\right)\otimes_{K^{\flat}}\left(\varphi^{\ast}(M_{2})\otimes K^{\flat}\right),\\
\varphi^{\ast}\left(\Hom_{\mathcal{O}_{K}^{\flat}}(M_{1},M_{2})\right)\otimes K^{\flat} & \simeq & \Hom_{K^{\flat}}\left(\varphi^{\ast}(M_{1})\otimes K^{\flat},\varphi^{\ast}(M_{2})\otimes K^{\flat}\right),\\
\varphi^{\ast}(\mathcal{O}_{K}^{\flat})\otimes K^{\flat} & \simeq & K^{\flat},
\end{eqnarray*}
the tensor products, internal Homs and neutral object in $\Mod_{R,f}^{\varphi}$
are given by 
\begin{eqnarray*}
(M_{1},\varphi_{1})\otimes(M_{2},\varphi_{2}) & \eqd & \left(M_{1}\otimes M_{2},\varphi_{1}\otimes\varphi_{2}\right),\\
\Hom\left((M_{1},\varphi_{1}),(M_{2},\varphi_{2})\right) & \eqd & \left(\Hom_{\mathcal{O}_{K}^{\flat}}(M_{1},M_{2}),\Hom_{K^{\flat}}(\varphi_{1}^{-1},\varphi_{2})\right),\\
\mathcal{O}_{K}^{\flat} & \eqd & (\mathcal{O}_{K}^{\flat},\mathrm{Id}).
\end{eqnarray*}
The rank and degree functions on $\Mod_{A,t}^{\varphi}$ induce rank
and degree functions on $\Mod_{R,f}^{\varphi}$, and the corresponding
Harder-Narasimhan (Fargues) filtrations $\mathcal{F}_{F}$ are compatible
since the essential image of $\Mod_{R,f}^{\varphi}\hookrightarrow\Mod_{A,t}^{\varphi}$
is stable under strict subobjects. The rank of $(M,\varphi_{M})\in\Mod_{R,f}^{\varphi}$
is the usual rank of the finite free $\mathcal{O}_{K}^{\flat}$-module
$M$, and its degree is the degree of the Hodge $\mathbb{R}$-filtration
\[
\mathcal{F}_{H}(M,\varphi_{M})\eqd\mathcal{F}\left(M,\varphi_{M}\left(\varphi^{\ast}M\right)\right)
\]
which is induced by the $\mathcal{O}_{K}^{\flat}$-lattice $\varphi_{M}(\varphi^{\ast}M)$
of $M\otimes K^{\flat}$ on the residue $M\otimes\mathbb{F}$ of $M$.
The Hodge type of $(M,\varphi_{M})$ is the type $t_{H}(M,\varphi_{M})$
of $\mathcal{F}_{H}(M,\varphi_{M})$, so that
\[
t_{H}(M,\varphi_{M})=\mathbf{d}\left(M,\varphi_{M}(\varphi^{\ast}M)\right)\quad\mbox{in}\quad\mathbb{R}_{\geq}^{r}
\]
where $r=\rank\, M$. The next proposition then follows from \cite[\S 5.3]{Co16}:
\begin{prop}
\label{prop:FarguesonO_Kflattensor}The restriction of the Fargues
filtration to the subcategory $\Mod_{R,f}^{\varphi}$ of $\Mod_{A,t}^{\varphi}$
is compatible with tensor products, duals, symmetric and exterior
powers. 
\begin{prop}
\label{prop:HodgeonO_Kflat}The Hodge filtration $\mathcal{F}_{H}:\Mod_{R,f}^{\varphi}\rightarrow\Fil_{\mathbb{F}}^{\mathbb{R}}$
is compatible with tensor products, duals, symmetric and exterior
powers. For every exact sequence 
\[
0\rightarrow(M_{1},\varphi_{1})\rightarrow(M_{2},\varphi_{2})\rightarrow(M_{3},\varphi_{3})\rightarrow0
\]
in $\Mod_{R,f}^{\varphi}$ with $r_{i}=\rank M_{i}$ (so that $r_{2}=r_{1}+r_{3}$),
we have 
\[
t_{H}(M_{1},\varphi_{1})\ast t_{H}(M_{3},\varphi_{3})\leq t_{H}(M_{2},\varphi_{2})\quad\mbox{in}\quad\mathbb{R}_{\geq}^{r_{2}}
\]
with equality if and only if for every $\gamma\in\mathbb{R}$, the
complex of $\mathbb{F}$-vector spaces
\[
0\rightarrow\mathcal{F}_{H}^{\gamma}(M_{1},\varphi_{1})\rightarrow\mathcal{F}_{H}^{\gamma}(M_{2},\varphi_{2})\rightarrow\mathcal{F}_{H}^{\gamma}(M_{3},\varphi_{3})\rightarrow0
\]
is exact.
\end{prop}
\end{prop}
\begin{proof}
This follows from \ref{sub:TensorProd} and lemma~\ref{lem:RelPos=000026ExcSeq}.\end{proof}
\begin{cor}
\label{cor:tF=000026tH4M1}For every $(M,\varphi)$ in $\Mod_{R,f}^{\varphi}$
of rank $r\in\mathbb{N}$, 
\[
t_{F}(M,\varphi)\leq t_{H}(M,\varphi)\quad\mbox{in}\quad\mathbb{R}_{\geq}^{r}.
\]
\end{cor}
\begin{proof}
Let $X^{\bullet}=\oplus_{\gamma}X^{\gamma}$ be the $\mathbb{R}$-graded
object of $\Mod_{R,f}^{\varphi}$ attached to the Fargues filtration
of $X=(M,\varphi)$. Then by propositions~\ref{prop:ExactSeqInModt}
and \ref{prop:HodgeonO_Kflat}, 
\[
t_{F}(X)=t_{F}(X^{\bullet})=\ast_{\gamma}t_{F}(X^{\gamma})\quad\mbox{and}\quad t_{H}(X)\geq t_{H}(X^{\bullet})=\ast_{\gamma}t_{H}(X^{\lambda}).
\]
We may thus assume that $X$ is semi-stable, in which case the result
is obvious since the concave polygons $t_{F}(X)$ and $t_{H}(X)$
have the same terminal points.
\end{proof}
\noindent Let $\mathcal{O}_{K}^{\flat}\{n\}:=A\{n\}\otimes\mathcal{O}_{K}^{\flat}$
and $\mathbb{F}\{n\}:=A\{n\}\otimes\mathbb{F}$, so that $M\{n\}=M\otimes\mathcal{O}_{K}^{\flat}\{n\}$
for every $M$ in $\Mod_{R,f}^{\varphi}$. The map $X\mapsto X\{n\}=X\otimes\mathbb{F}\{n\}$
then induces a bijection between $\mathbb{F}$-subspaces $X$ of $M\otimes\mathbb{F}$
and $X\{n\}$ of $M\{n\}\otimes\mathbb{F}=M\otimes\mathbb{F}\{n\}$. 
\begin{prop}
\label{prop:tH(twist)4O_K^flat}For every $M\in\Mod_{R,f}^{\varphi}$
of rank $r\in\mathbb{N}$ and any $n\in\mathbb{Z}$, 
\[
\mathcal{F}_{H}^{\gamma}(M\{n\})=\mathcal{F}_{H}^{\gamma-n}(M)\{n\}\quad\mbox{inside}\quad M\{n\}\otimes\mathbb{F}
\]
for every $\gamma\in\mathbb{R}$, hence 
\[
t_{H}(M\{n\})=t_{H}(M)+(n,\cdots,n)\quad\mbox{in}\quad\mathbb{R}_{\geq}^{r}.
\]
\end{prop}
\begin{proof}
By definition, $\mathcal{F}_{H}^{\gamma}(M\{n\})$ equals 
\[
\frac{M\otimes\mathcal{O}_{K}^{\flat}\{n\}\cap\left(I^{\gamma}\cdot\varphi_{M}\left(\varphi^{\ast}M\right)\otimes(\xi^{\prime}\bmod\pi)^{-n}\mathcal{O}_{K}^{\flat}\{n\}\right)+\mathfrak{m}_{K}^{\flat}\cdot M\otimes\mathcal{O}_{K}^{\flat}\{n\}}{\mathfrak{m}_{K}^{\flat}\cdot M\otimes\mathcal{O}_{K}^{\flat}\{n\}}
\]
where $I^{\gamma}=\left\{ x\in K^{\flat}:\left|x\right|\leq q^{-\gamma}\right\} $,
and $\xi^{\prime}\bmod\pi=\varpi^{q}$, i.e.
\begin{eqnarray*}
\mathcal{F}_{H}^{\gamma}(M\{n\}) & = & \frac{M\cap\left(I^{\gamma}\varpi^{-qn}\cdot\varphi_{M}\left(\varphi^{\ast}M\right)\right)+\mathfrak{m}_{K}^{\flat}\cdot M}{\mathfrak{m}_{K}^{\flat}\cdot M}\otimes\mathbb{F}\{n\}\\
 & = & \mathcal{F}_{H}^{\gamma-n}(M)\{n\}
\end{eqnarray*}
since $\left|\varpi^{-qn}\right|=q^{n}$, which proves the proposition.
\end{proof}

\subsubsection{$\boxed{R=L}$\label{sub:RisL}}

Here, $\Mod_{L,\ast}^{\varphi}=\Mod_{L,f}^{\varphi}$ is the tannakian
category of $E$-isocrystals over $\mathbb{F}$, i.e.~finite dimensional
vector spaces $D$ over $L$ equipped with an isomorphism $\varphi_{D}:\varphi^{\ast}D\rightarrow D$,
or equivalently, with a $\varphi$-semilinear automorphism $\phi_{D}:D\rightarrow D$.
The Dieudonné-Manin classification gives a slope decomposition 
\[
(D,\varphi_{D})=\oplus_{\lambda\in\mathbb{Q}}(D_{\lambda},\varphi_{D_{\lambda}}).
\]
For $\lambda=\frac{d}{h}$ with $d\in\mathbb{Z}$ and $h\in\mathbb{N}^{\ast}$
relatively prime, $D_{\lambda}$ is the union of the finitely generated
$\mathcal{O}_{L}$-submodules $X$ of $D$ such that $\phi_{D}^{(h)}(X)=\pi^{d}X$.
This Newton decomposition is functorial, compatible with all tensor
product constructions, thus 
\[
\mathcal{G}_{N}:\Mod_{L,f}^{\varphi}\rightarrow\Gra_{L}^{\mathbb{Q}},\qquad\mathcal{G}_{N}^{\lambda}(D,\varphi_{D})\eqd D_{\lambda}
\]
is an exact $\otimes$-functor, and so are the corresponding opposed
Newton $\mathbb{Q}$-filtrations
\[
\mathcal{F}_{N},\mathcal{F}_{N}^{\iota}:\Mod_{L,f}^{\varphi}\rightarrow\Fil_{L}^{\mathbb{Q}}
\]
which are given by the usual formulas
\[
\mathcal{F}_{N}^{\lambda}(D,\varphi_{D})\eqd\oplus_{\lambda^{\prime}\geq\lambda}D_{\lambda^{\prime}}\quad\mbox{and}\quad\mathcal{F}_{N}^{\iota\lambda}(D,\varphi_{D})\eqd\oplus_{\lambda^{\prime}\geq\lambda}D_{-\lambda^{\prime}}.
\]
We denote by $t_{N}(D,\varphi_{D})$ and $t_{N}^{\iota}(D,\varphi_{D})$
the corresponding opposed types. Both Newton filtrations are Harder-Narasimhan
filtrations, for the obvious rank function on $\Mod_{L,f}^{\varphi}$
and for the opposed degree functions which are respectively given
by 
\[
\deg_{N}(D,\varphi_{D})\eqd\deg\mathcal{F}_{N}(D,\varphi_{D})\quad\mbox{and}\quad\deg_{N}^{\iota}(D,\varphi_{D})\eqd\deg\mathcal{F}_{N}^{\iota}(D,\varphi_{D}).
\]
These degree functions are $\mathbb{Z}$-valued! If the residue field
$\mathbb{F}$ is algebraically closed, the category $\Mod_{L,f}^{\varphi}$
is even semi-simple, with one simple object $D_{\lambda}^{\circ}$
for each slope $\lambda\in\mathbb{Q}$. If $\lambda=\frac{d}{h}$
as above, then $\rank(D_{\lambda}^{\circ})=h$ and $\deg_{N}(D_{\lambda}^{\circ})=d=-\deg_{N}^{\iota}(D_{\lambda}^{\circ})$.

Since $\varphi_{A\{1\}}(\varphi^{\ast}A\{1\})=\xi^{\prime-1}A\{1\}$
and $\xi'$ maps to a uniformizer in $L$, we have $\phi_{L\{1\}}(\mathcal{O}_{L}\{1\})=\pi^{-1}\mathcal{O}_{L}\{1\}$
for the $\mathcal{O}_{L}$-lattice $\mathcal{O}_{L}\{1\}:=A\{1\}\otimes\mathcal{O}_{L}$
in the Tate object $L\{1\}:=A\{1\}\otimes L$ of $\Mod_{L,f}^{\varphi}$.
It follows that
\[
\deg_{N}(L\{1\})=-1\quad\mbox{and}\quad\deg_{N}^{\iota}(L\{1\})=+1.
\]
For $D$ in $\Mod_{L,f}^{\varphi}$ and $n\in\mathbb{Z}$, we set
$D\{n\}:=D\otimes L\{1\}^{\otimes n}$ as usual. Then
\[
\mathcal{G}_{N}^{\gamma}(D\{n\})=\mathcal{G}_{N}^{\gamma+n}(D)\{n\}
\]
for every $\gamma\in\mathbb{Q}$, therefore
\[
\mathcal{F}_{N}^{\gamma}(D\{n\})=\mathcal{F}_{N}^{\gamma+n}(D)\{n\}\quad\mbox{and}\quad\mathcal{F}_{N}^{\iota\gamma}(D\{n\})=\mathcal{F}_{N}^{\iota\gamma-n}(D)\{n\}.
\]
In particular, we have the following equalities in $\mathbb{Q}_{\geq}^{r}$
where $r=\dim_{L}D$: 
\[
t_{N}(D\{n\})=t_{N}(D)-(n,\cdots,n)\quad\mbox{and}\quad t_{N}^{\iota}(D\{n\})=t_{N}^{\iota}(D)+(n,\cdots,n).
\]

\subsubsection{$\boxed{R=\mathcal{O}_{L}}$\label{sub:RisOL}}

The category $\Mod_{\mathcal{O}_{L},f}^{\varphi}$ is now the category
of $\mathcal{O}_{E}$-crystals over $\mathbb{F}$, or $\mathcal{O}_{L}$-lattices
in $E$-isocrystals over $\mathbb{F}$, whose objects are finite free
$\mathcal{O}_{L}$-modules $M$ equipped with an isomorphism $\varphi_{M}:\varphi^{\ast}M\otimes L\rightarrow M\otimes L$.
It is a quasi-abelian $\mathcal{O}_{E}$-linear rigid $\otimes$-category,
with an exact faithful $\otimes$-functor $-\otimes L:\Mod_{\mathcal{O}_{L},f}^{\varphi}\rightarrow\Mod_{L,f}^{\varphi}$.
Since $\mathcal{O}_{L}$ is a discrete valuation ring, there is also
a Hodge $\mathbb{Z}$-filtration, defined by 
\[
\mathcal{F}_{H}(M,\varphi_{M})\eqd\mathcal{F}\left(M,\varphi_{M}(\varphi^{\ast}M)\right),
\]
a $\mathbb{Z}$-filtration on $M\otimes\mathbb{F}$, whose type will
be denoted by $t_{H}(M,\varphi_{M})$, so that 
\[
t_{H}(M,\varphi_{M})=\mathbf{d}\left(M,\varphi_{M}(\varphi^{\ast}M)\right)\quad\mbox{in}\quad\mathbb{Z}_{\geq}^{r}
\]
where $r=\rank(M)$. As before, we have the following proposition:
\begin{prop}
\label{prop:HodgeonO_L}The Hodge filtration $\mathcal{F}_{H}:\Mod_{\mathcal{O}_{L},f}^{\varphi}\rightarrow\Fil_{\mathbb{F}}^{\mathbb{Z}}$
is compatible with tensor products, duals, symmetric and exterior
powers. For every exact sequence 
\[
0\rightarrow(M_{1},\varphi_{1})\rightarrow(M_{2},\varphi_{2})\rightarrow(M_{3},\varphi_{3})\rightarrow0
\]
in $\Mod_{\mathcal{O}_{L},f}^{\varphi}$ with $r_{i}=\rank\, M_{i}$
(so that $r_{2}=r_{1}+r_{3}$), we have 
\[
t_{H}(M_{1},\varphi_{1})\ast t_{H}(M_{3},\varphi_{3})\leq t_{H}(M_{2},\varphi_{2})\quad\mbox{in}\quad\mathbb{Z}_{\geq}^{r_{2}}
\]
with equality if and only if for every $\gamma\in\mathbb{Z}$, the
complex of $\mathbb{F}$-vector spaces
\[
0\rightarrow\mathcal{F}_{H}^{\gamma}(M_{1},\varphi_{1})\rightarrow\mathcal{F}_{H}^{\gamma}(M_{2},\varphi_{2})\rightarrow\mathcal{F}_{H}^{\gamma}(M_{3},\varphi_{3})\rightarrow0
\]
is exact.\end{prop}
\begin{cor}
\label{cor:Mazurinequality}(Mazur's inequality) For every $X$ in
$\Mod_{\mathcal{O}_{L},f}^{\varphi}$ of rank $r\in\mathbb{N}$, 
\[
t_{N}^{\iota}(X\otimes L)\leq t_{H}(X)\quad\mbox{in}\quad\mathbb{Q}_{\geq}^{r}.
\]
\end{cor}
\begin{proof}
We first show that $\mathcal{F}_{N}^{\iota}(X\otimes L)$ and $\mathcal{F}_{H}(X)$
have the same degree. Since both filtrations are compatible with exterior
powers, we may assume that the rank of $X=(M,\varphi_{M})$ equals
$1$. Then $\varphi_{M}(\varphi^{\ast}M)=\pi^{-d}M$ for some $d\in\mathbb{Z}$,
thus indeed $\deg\mathcal{F}_{N}^{\iota}(X\otimes L)=d=\deg\mathcal{F}_{H}(X)$.
Returning to the general case, both polygons thus have the same terminal
points. We now follow the proof of corollary~\ref{cor:tF=000026tH4M1}.
Let $X^{\bullet}=\oplus_{\gamma}X^{\gamma}$ be the $\mathbb{Q}$-graded
object of $\Mod_{\mathcal{O}_{L},f}^{\varphi}$ attached to the filtration
on $X$ induced by $\mathcal{F}_{N}^{\iota}(X\otimes L)$. Then by
exactness of $\mathcal{F}_{N}^{\iota}$ and the previous proposition
\[
t_{N}^{\iota}(X\otimes L)=t_{N}^{\iota}(X^{\bullet}\otimes L)=\ast_{\gamma}t_{N}^{\iota}(X^{\gamma}\otimes L)\quad\mbox{and}\quad t_{H}(X)\geq t_{H}(X^{\bullet})=\ast_{\gamma}t_{H}(X^{\lambda}).
\]
We may thus assume that $X\otimes L$ is semi-stable (i.e.~isoclinic),
in which case the result is obvious since $t_{N}^{\iota}(X\otimes L)$
and $t_{H}(X)$ have the same terminal points. 
\end{proof}
\noindent We have already defined the Tate object $\mathcal{O}_{L}\{1\}=A\{1\}\otimes\mathcal{O}_{L}$,
giving rise to Tate twists $M\{n\}:=M\otimes\mathcal{O}_{L}\{1\}^{\otimes n}$
for every $M\in\Mod_{\mathcal{O}_{L},f}^{\varphi}$ and $n\in\mathbb{Z}$,
with a bijection $X\mapsto X\{n\}:=X\otimes\mathbb{F}\{n\}$ between
$\mathbb{F}$-subspaces of $M\otimes\mathbb{F}$ and $M\{n\}\otimes\mathbb{F}=M\otimes\mathbb{F}\{n\}$. 
\begin{prop}
\label{prop:tH(twist)4O_L}For every $M\in\Mod_{\mathcal{O}_{L},f}^{\varphi}$
of rank $r\in\mathbb{N}$ and $n\in\mathbb{Z}$, 
\[
\mathcal{F}_{H}^{\gamma}(M\{n\})=\mathcal{F}_{H}^{\gamma-n}(M)\{n\}\quad\mbox{inside}\quad M\{n\}\otimes\mathbb{F}
\]
for every $\gamma\in\mathbb{Z}$, hence 
\[
t_{H}(M\{n\})=t_{H}(M)+(n,\cdots,n)\quad\mbox{in}\quad\mathbb{Z}_{\geq}^{r}.
\]
\end{prop}
\begin{proof}
This is similar to proposition~\ref{prop:tH(twist)4O_K^flat}. It
also follows from the compatibility of the Hodge filtration with tensor
products (proposition~\ref{prop:HodgeonO_L}), along with the formula
\[
\varphi_{L\{1\}}(\varphi^{\ast}\mathcal{O}_{L}\{1\})=\pi^{-1}\mathcal{O}_{L}\{1\}
\]
which shows that $\mathcal{F}_{H}(\mathcal{O}_{L}\{1\})$ has a single
jump at $1$. 
\end{proof}

\subsubsection{$\boxed{R=A(K)=W_{\mathcal{O}_{E}}(K^{\flat})}$\label{sub:RisW(Kflat)}}

Then $\Mod_{R,\ast}^{\varphi}$ is the abelian $\otimes$-category
of finitely generated $R$-modules $M$ equipped with an isomorphism
$\varphi_{M}:\varphi^{\ast}M\rightarrow M$, or equivalently, with
a $\varphi$-semilinear automorphism $\phi_{M}:M\rightarrow M$. If
$\overline{K}^{\flat}$ is an algebraic closure of $K^{\flat}$ with
Galois group $\Gamma=\Gal(\overline{K}^{\flat}/K^{\flat})$ and $\overline{R}=W_{\mathcal{O}_{E}}(\overline{K}^{\flat})$,
the formulas 
\begin{eqnarray*}
(M,\varphi_{M}) & \mapsto & (T,\rho)=\left(\left(M\otimes_{R}\overline{R}\right)^{\phi_{M}\otimes\varphi=1},\mathrm{Id}\otimes\rho_{\overline{R}}\right)\\
(T,\rho) & \mapsto & (M,\varphi_{M})=\left(\left(T\otimes_{\mathcal{O}_{E}}\overline{R},\rho\otimes\rho_{\overline{R}}\right)^{\Gamma},\mathrm{Id}\otimes\varphi\right)
\end{eqnarray*}
yield equivalences of $\otimes$-categories between $\Mod_{R,\ast}^{\varphi}$
and the category $\Rep_{\mathcal{O}_{E},\ast}(\Gamma$) of continuous
representations $\rho:\Gamma\rightarrow\Aut_{\mathcal{O}_{E}}(T)$
on finitely generated $\mathcal{O}_{E}$-modules $T$~\cite[1.2.6]{Fo90}.
Here $\rho_{\overline{R}}:\Gamma\rightarrow\Aut_{R,\varphi}(\overline{R})$
is induced by the functoriality of $W_{\mathcal{O}_{E}}(-)$.

\subsubsection{$\boxed{{\textstyle R=A[\frac{1}{\pi}]}}$.\label{sub:RisA=00005B1/pi=00005D}}

The category $\Mod_{R,\ast}^{\varphi}=\Mod_{R,f}^{\varphi}$ is the
rigid $E$-linear $\otimes$-category of finite free $A[\frac{1}{\pi}]$-modules
$M$ with an isomorphism $\varphi_{M}:\varphi^{\ast}M[\xi^{\prime-1}]\rightarrow M[\xi^{\prime-1}]$.
Since $\xi'=\varphi(\xi)$, the Frobenius of $A$ induces an isomorphism
$\varphi:B_{dR}^{+}\rightarrow B_{dR}^{\prime+}$ of discrete valuation
rings between the completion of the local rings of $A[\frac{1}{\pi}]$
at the maximal ideals $A[\frac{1}{\pi}]\xi=\ker(\theta:A[\frac{1}{\pi}]\twoheadrightarrow K)$
and $A[\frac{1}{\pi}]\xi'=\varphi(A[\frac{1}{\pi}]\xi)$, along with
the induced isomorphisms $\varphi:K\rightarrow K'$ and $\varphi:B_{dR}\rightarrow B_{dR}^{\prime}$
between the residue and fraction fields of $B_{dR}^{+}$ and $B_{dR}^{\prime+}$.
For $(M,\varphi_{M})$ in $\Mod_{R,f}^{\varphi}$, the commutative
diagram
\[
\xyC{4pc}\xymatrix{M\left[\xi^{-1}\right]\ar[d]_{\varphi}\ar[r]^{(\varphi^{-1})^{\ast}(\varphi_{M})} & ((\varphi^{-1})^{\ast}M)\left[\xi^{-1}\right]\ar[d]^{\varphi}\\
(\varphi^{\ast}M)\left[\xi^{\prime-1}\right]\ar[r]^{\varphi_{M}} & M\left[\xi^{\prime-1}\right]
}
\]
extends to a commutative diagram 
\[
\xyC{4pc}\xymatrix{M\otimes B_{dR}\ar[d]_{\varphi\otimes\varphi}\ar[r]^{(\varphi^{-1})^{\ast}(\varphi_{M})} & ((\varphi^{-1})^{\ast}M)\otimes B_{dR}\ar[d]^{\varphi\otimes\varphi}\\
(\varphi^{\ast}M)\otimes B_{dR}^{\prime}\ar[r]^{\varphi_{M}} & M\otimes B_{dR}^{\prime}
}
\]
Then $M\otimes B_{dR}^{+}$ is a $B_{dR}^{+}$-lattice in $M\otimes B_{dR}$
and similarly for the other three vertices. Each line of our diagram
thus yields a pair of $\mathbb{Z}$-filtrations on the residue (over
$K$ or $K'$) of its vertices, which have opposed types in $\mathbb{Z}_{\geq}^{r}$
where $r$ is the rank of $M$, and the two pairs match along the
$\varphi$-equivariant isomorphisms which are induced by the vertical
maps. In particular, the Hodge $\mathbb{Z}$-filtrations 
\begin{eqnarray*}
\mathcal{F}_{H}^{\iota}(M,\varphi_{M}) & \eqd & \mathcal{F}\left(M\otimes B_{dR}^{+}\,,\left((\varphi^{-1})^{\ast}(\varphi_{M})\right)^{-1}\left(\left((\varphi^{-1})^{\ast}M\right)\otimes B_{dR}^{+}\right)\right)\\
\mbox{and}\quad\mathcal{F}_{H}(M,\varphi_{M}) & \eqd & \mathcal{F}\left(M\otimes B_{dR}^{\prime+}\,,\varphi_{M}\left((\varphi^{\ast}M)\otimes B_{dR}^{\prime+}\right)\right)
\end{eqnarray*}
on respectively $M\otimes_{A}K$ and $M\otimes_{A}K'$ have opposed
types
\[
t_{H}^{\iota}(M,\varphi_{M})\quad\mbox{and}\quad t_{H}(M,\varphi_{M})\quad\mbox{in}\quad\mathbb{Z}_{\geq}^{r}.
\]

\begin{prop}
\label{prop:Hodge4A=00005B1/pi=00005D_tensorexact}The Hodge filtration
$\mathcal{F}_{H}:\Mod_{R,f}^{\varphi}\rightarrow\Fil_{K^{\prime}}^{\mathbb{Z}}$
is compatible with tensor products, duals, symmetric and exterior
powers. For every exact sequence 
\[
0\rightarrow(M_{1},\varphi_{1})\rightarrow(M_{2},\varphi_{2})\rightarrow(M_{3},\varphi_{3})\rightarrow0
\]
 in $\Mod_{R,f}^{\varphi}$ with $r_{i}=\rank M_{i}$ (so that $r_{2}=r_{1}+r_{3}$),
we have 
\[
t_{H}(M_{1},\varphi_{1})\ast t_{H}(M_{3},\varphi_{3})\leq t_{H}(M_{2},\varphi_{2})\quad\mbox{in}\quad\mathbb{Z}_{\geq}^{r_{2}}
\]
with equality if and only if for every $\gamma\in\mathbb{R}$, the
complex of $K'$-vector spaces
\[
0\rightarrow\mathcal{F}_{H}^{\gamma}(M_{1},\varphi_{1})\rightarrow\mathcal{F}_{H}^{\gamma}(M_{2},\varphi_{2})\rightarrow\mathcal{F}_{H}^{\gamma}(M_{3},\varphi_{3})\rightarrow0
\]
is exact. The Hodge filtration $\mathcal{F}_{H}^{\iota}:\Mod_{R,f}^{\varphi}\rightarrow\Fil_{K}^{\mathbb{Z}}$
has analogous properties.
\end{prop}
\noindent For the Tate object $A[\frac{1}{\pi}]\{1\}:=A\{1\}[\frac{1}{\pi}]$,
$t_{H}=1=-t_{H}^{\iota}$, thus again:
\begin{prop}
\label{prop:t_H(twist)4A=00005B1/pi=00005D}For every $M$ in $\Mod_{R,f}^{\varphi}$
of rank $r\in\mathbb{N}$ and $n\in\mathbb{Z}$, 
\[
\mathcal{F}_{H}^{\gamma}(M\{n\})=\mathcal{F}_{H}^{\gamma-n}(M)\{n\}\quad\mbox{and}\quad\mathcal{F}_{H}^{\iota\gamma}(M\{n\})=\mathcal{F}_{H}^{\iota\gamma+n}(M)\{n\}
\]
for every $\gamma\in\mathbb{Z}$, therefore
\[
\begin{array}{rcl}
t_{H}(M\{n\}) & = & t_{H}(M)+(n,\cdots,n)\\
t_{H}^{\iota}(M\{n\}) & = & t_{H}^{\iota}(M)-(n,\cdots,n)
\end{array}\quad\mbox{in}\quad\mathbb{Z}_{\geq}^{r}.
\]

\end{prop}
\noindent The $\otimes$-functor $\Mod_{A,\ast}^{\varphi}\rightarrow\Mod_{R,f}^{\varphi}$
identifies the isogeny category $\Mod_{A,\ast}^{\varphi}\otimes E$
with a full subcategory of $\Mod_{R,f}^{\varphi}$. We may thus unambiguously
denote by $X\mapsto X\otimes E$ or $X[\frac{1}{\pi}]$ the $\otimes$-functor
from $\Mod_{A,\ast}^{\varphi}$ to either $\Mod_{A,\ast}^{\varphi}\otimes E$
or $\Mod_{R,f}^{\varphi}$.
\begin{prop}
\label{prop:CompHodgeA=00005B1/pi=00005DOLA1}For a finite free BKF-module
$M$ of rank $r\in\mathbb{N}$,
\[
\begin{array}{rcll}
t_{H}(M\otimes\mathcal{O}_{L}) & \leq & t_{H}\left(M\otimes E\right) & \mbox{in}\quad\mathbb{Z}_{\geq}^{r},\\
t_{H}(M\otimes\mathcal{O}_{K}^{\flat}) & \leq & t_{H}\left(M\otimes E\right) & \mbox{in}\quad\mathbb{R}_{\geq}^{r}.
\end{array}
\]
\end{prop}
\begin{proof}
Using the compatibility with Tate twists (propositions~\ref{prop:tH(twist)4O_K^flat},
\ref{prop:tH(twist)4O_L}, and \ref{prop:t_H(twist)4A=00005B1/pi=00005D}),
we may assume that $M\subset M'=\varphi_{M}(\varphi^{\ast}M)$ in
$M[\xi^{\prime-1}]$. Then $Q=M'/M$ is a perfect $A$-module of projective
dimension $\leq1$ which is killed by a power of $\xi^{\prime}$,
say $\xi^{\prime n}Q=0$. For $0\leq i\leq n$, let $M^{i}$ be the
inverse image of $Q^{i}=Q[\xi^{\prime i}]$ in $M'$, so that 
\[
M=M^{0}\subset M^{1}\subset\cdots\subset M^{n}=M'\quad\mbox{with}\quad M^{i}/M^{0}=Q^{i}.
\]
We first claim that each $M^{i}$ is finite free over $A$. By descending
induction on $i$, it is sufficient to establish that the following
$A$-module has projective dimension $1$: 
\[
X^{i}=M^{i}/M^{i-1}\simeq Q^{i}/Q^{i-1}\simeq\xi^{\prime i-1}Q[\xi^{\prime i}]\subset Q[\xi^{\prime}]\subset Q.
\]
We will show that it is finite free over $A(1)=A/A\xi^{\prime}$.
Since $A(1)\simeq A/A\xi\simeq\mathcal{O}_{K}$ is a valuation ring,
we just have to verify that $X^{i}$ is finitely generated and torsion-free
over $A(1)$. Since $Q$ is finitely presented over $A$, it is finitely
presented over $A(n)=A/A\xi^{\prime n}$, which is a coherent ring
by \cite[3.26]{BaMoSc16}, thus $Q^{i}=Q[\xi^{\prime i}]$ is finitely
presented over $A(n)$ and $A$ for all $i$, and so is $X^{i}\simeq Q^{i}/Q^{i-1}$.
On the other hand, $Q[\mathfrak{m}^{\infty}]=0$ by~\ref{sub:ProjDimModA},
thus also $X^{i}[\mathfrak{m}^{\infty}]=0$, which means that $X^{i}$
is indeed torsion-free as an $A(1)$-module. We denote by $x_{i}$
the rank of $X^{i}$ over $A(1)$. 

Let $S$ be any one of the valuations rings $B_{dR}^{\prime+}$, $\mathcal{O}_{K}^{\flat}$
or $\mathcal{O}_{L}$. Then $\Tor_{1}^{A}(X^{i},S)=0$ since $\Tor_{1}^{A}(A(1),S)=S[\xi^{\prime}]=0$.
We thus obtain a sequence of $S$-lattices 
\[
M\otimes S=M^{0}\otimes S\subset M^{1}\otimes S\subset\cdots\subset M^{n}\otimes S=M^{\prime}\otimes S
\]
inside $M\otimes\Frac(S)$. The triangular inequality of lemma~\ref{lem:triangIneq}
then yields
\[
\mathbf{d}\left(M\otimes S,M'\otimes S\right)\leq\sum_{i=1}^{n}\mathbf{d}\left(M^{i-1}\otimes S,M^{i}\otimes S\right)\quad\mbox{in}\quad\mathbb{R}_{\geq}^{r}.
\]
Since $M^{i}\otimes S/M^{i-1}\otimes S\simeq X^{i}\otimes S\simeq(S/\xi_{S}^{\prime}S)^{x_{i}}$
where $\xi_{S}^{\prime}$ is the image of $\xi^{\prime}$ in $S$
and since also $\left|\xi_{S}^{\prime}\right|=q^{-1}$ in all three
cases for the normalized absolute value on $S$, 
\[
\mathbf{d}\left(M^{i-1}\otimes S,M^{i}\otimes S\right)=(1,\cdots,1,0,\cdots,0)\quad\mbox{in}\quad\mathbb{Z}_{\geq}^{r}\subset\mathbb{R}_{\geq}^{r}
\]
with exactly $x_{i}$ one's. Now observe that by definition of our
various Hodge types, 
\[
\mathbf{d}\left(M\otimes S,M'\otimes S\right)=\begin{cases}
t_{H}\left(M\otimes E\right) & \mbox{for }S=B_{dR}^{\prime+},\\
t_{H}(M\otimes\mathcal{O}_{L}) & \mbox{for }S=\mathcal{O}_{L},\\
t_{H}(M\otimes\mathcal{O}_{K}^{\flat}) & \mbox{for }S=\mathcal{O}_{K}^{\flat}.
\end{cases}
\]
To establish the proposition, it is now sufficient to show that for
$S=B_{dR}^{\prime+}$, actually
\[
\mathbf{d}\left(M\otimes S,M'\otimes S\right)=\sum_{i=1}^{n}\mathbf{d}\left(M^{i}\otimes S,M^{i-1}\otimes S\right)\quad\mbox{in}\quad\mathbb{Z}_{\geq}^{r}.
\]
Since $S$ is the completion of a Noetherian local ring of $A$, it
is flat over $A$, thus 
\[
\frac{M^{i}\otimes S}{M^{0}\otimes S}=Q[\xi^{\prime i}]\otimes S=Q\otimes S[\xi_{S}^{\prime i}]=\frac{M'\otimes S}{M\otimes S}[\xi_{S}^{\prime i}],
\]
which means that $M^{i}\otimes S=(M^{\prime}\otimes S)\cap\xi_{S}^{\prime-i}(M\otimes S)$
in $M\otimes B_{dR}^{\prime}$. If 
\[
\mathbf{d}\left(M\otimes S,M'\otimes S\right)=(n_{1}\geq\cdots\geq n_{r})\quad\mbox{in}\quad\mathbb{Z}_{\geq}^{r},
\]
there exists an $S$-basis $(e_{1},\cdots,e_{r})$ of $M\otimes S$
such that $(\xi_{S}^{\prime-n_{1}}e_{1},\cdots,\xi_{S}^{\prime-n_{r}}e_{r})$
is an $S$-basis of $M^{\prime}\otimes S$. Then $(\xi_{S}^{\prime-\min(n_{1},i)}e_{1},\cdots,\xi_{S}^{\prime-\min(n_{r},i)}e_{r})$
is an $S$-basis of $M^{i}\otimes S$, $x_{i}=\max\{j:n_{j}\geq i\}$
and indeed $n_{j}=\sharp\{i:x_{i}\geq j\}$ for all $j\in\{1,\cdots,r\}$.\end{proof}
\begin{rem}
With notations as above (and for a finite free BKF-module $M$ such
that $M^{\vee}$ is effective), the proof shows that we have equality
when $n=1$, i.e.~$\xi'Q=0$, i.e.~$t_{H}(M\otimes E)=(1,\cdots,1,0,\cdots0)$
is minuscule. More generally for $S\in\{\mathcal{O}_{L},\mathcal{O}_{K}^{\flat}\}$,
$t_{H}(M\otimes E)=t_{H}(M\otimes S)$ if $\Tor_{j}^{A}(S,Q/\xi^{\prime i}Q)=0$
for $1\leq i\leq n$ and $j\in\{1,2\}$.
\begin{rem}
For a finite free BKF-module $M\in\Mod_{A,f}^{\varphi}$ of rank $r\in\mathbb{N}$,
we thus have
\[
t_{F,\infty}(M)\leq t_{F}(M\otimes\mathcal{O}_{K}^{\flat})\leq t_{H}(M\otimes\mathcal{O}_{K}^{\flat})\leq t_{H}(M[{\textstyle \frac{1}{\pi}]\geq t_{H}(M\otimes\mathcal{O}_{L})\geq t_{N}^{\iota}(M\otimes L)}
\]
by propositions~\ref{prop:ComptFinftytFn}, \ref{prop:CompHodgeA=00005B1/pi=00005DOLA1}
and corollaries \ref{cor:tF=000026tH4M1} and \ref{cor:Mazurinequality}.
In particular,
\[
\begin{aligned}t_{F,\infty}(M)(r) & =\deg_{t}(M\otimes\mathcal{O}_{K}^{\flat})=\deg\mathcal{F}_{H}(M\otimes\mathcal{O}_{K}^{\flat})\\
 & =\deg\mathcal{F}_{H}(M{\textstyle [\frac{1}{\pi}]})\\
 & =\deg\mathcal{F}_{H}(M\otimes\mathcal{O}_{L})=\deg_{N}^{\iota}(M\otimes L)
\end{aligned}
\]
and this apriori real number actually belongs to $\mathbb{Z}$. We
call it the degree of $M$. 
\end{rem}
\end{rem}

\section{The functors of Fargues\label{sec:FunctorsOfFargues}}

Suppose from now on that $K=C$ is algebraically closed. In this section,
we will define and study the following commutative diagram of covariant
$\otimes$-functors:
\[
\xymatrix{\Mod_{A,f}^{\varphi}\ar[rr]^{\HTT'}\ar@{^{(}->}[d] &  & \HT_{\mathcal{O}_{E}}^{B_{dR}}\ar@{^{(}->}[d]\\
\Mod_{A,f}^{\varphi}\otimes E\ar[r]^{\underline{\mathcal{E}}}\ar@{^{(}->}[d] & \Modif_{X}^{ad}\ar[r]^{\HTT}\ar@{^{(}->}[d] & \HT_{E}^{B_{dR}}\\
\Mod_{A[\frac{1}{\pi}]}^{\varphi}\ar[r]^{\underline{\mathcal{E}}} & \Modif_{X}
}
\]
In this diagram, the first two lines are equivalences of $\otimes$-categories,
the top vertical arrows are faithful and the bottom ones fully faithful.
The construction of $\underline{\mathcal{E}}$ which is given below
is a covariant version of the analytic construction of \cite{Fa15}.
A slightly twisted version of it was sketched in Scholze's course
\cite{ScWe15} -- for stukhas with one paw at $\mathfrak{m}=A\xi$.
Our variant is meant to match the normalized construction of $\HTT'$
in \cite{BaMoSc16}, where the paw was twisted from $\mathfrak{m}$
to $\mathfrak{m}'=A\xi'$. Following \cite{BaMoSc16}, we fix a compatible
system of $p$-power roots of unity, $\zeta_{p^{r}}\in\mathcal{O}_{C}^{\times}$
for $r\geq1$, and set 
\[
\epsilon=(1,\zeta_{p},\zeta_{p^{2}},\cdots)\in\mathcal{O}_{C}^{\flat,\times},\qquad\mu=[\epsilon]-1\in\mathfrak{m}\subset A,
\]
\[
\xi={\textstyle \frac{\mu}{\varphi^{-1}(\mu)}}=1+[\epsilon^{1/q}]+\cdots+[\epsilon^{1/q}]^{q-1}\in\mathfrak{m}\subset A,
\]
\[
\xi^{\prime}=\varphi(\xi)={\textstyle \frac{\varphi(\mu)}{\mu}}=1+[\epsilon]+\cdots+[\epsilon]^{q-1}\in\mathfrak{m}\subset A,
\]
\[
\varpi=\xi\bmod\pi=1+\epsilon^{1/q}+\cdots+(\epsilon^{1/q})^{q-1}\in\mathfrak{m}_{C}^{\flat}\subset\mathcal{O}_{C}^{\flat},
\]
\[
\varpi^{q}=\xi^{\prime}\bmod\pi=1+\epsilon+\cdots+\epsilon^{q-1}\in\mathfrak{m}_{C}^{\flat}\subset\mathcal{O}_{C}^{\flat}.
\]
As suggested by the notations, $\xi$ is a generator of $\ker(\theta:A\twoheadrightarrow\mathcal{O}_{C})$.
We have
\[
\varphi^{-1}(\mu)\mid\mu\mid\varphi(\mu)\,\,\mbox{in }A\quad\mbox{thus}\quad A[{\textstyle \frac{1}{\varphi^{-1}(\mu)}]\subset A[\frac{1}{\mu}]\subset A[\frac{1}{\varphi(\mu)}].}
\]
Moreover, $\theta(\varphi^{-1}(\mu))=\zeta_{q}-1\neq0$, and therefore
$\xi\nmid\varphi^{-1}(\mu)$ and $\xi^{\prime}\nmid\mu$.

\subsection{Modifications of vector bundles on the curve}

\subsubsection{The Fargues-Fontaine curve\label{sub:The-Fargues-Fontaine-curve}}

Let $X=X_{C^{\flat},E}$ be the Fargues-Fontaine curve attached to
$(C^{\flat},E)$ \cite{FaFo15}. This is an integral noetherian regular
$1$-dimensional scheme over $E$ which is a complete curve in the
sense of \cite[5.1.3]{FaFo15}: the degree function on divisors factors
through a degree function on the Picard group, $\deg:\Pic(X)\rightarrow\mathbb{N}$.
We denote by $\eta$ the generic point of $X$ and by $E(X)=\mathcal{O}_{X,\eta}$
the field of rational functions on $X$. In addition, there is a distinguished
closed point $\infty\in\left|X\right|$ with completed local ring
$\mathcal{O}_{X,\infty}^{\wedge}$ canonically isomorphic to the ring
$B_{dR}^{+}$ of section~\ref{sub:RisA=00005B1/pi=00005D}.

\subsubsection{Vector bundles on the curve}

Let $\Bun_{X}$ be the $E$-linear $\otimes$-category of vector bundles
$\mathcal{E}$ on $X$. Since $X$ is a regular curve, it is a quasi-abelian
category whose short exact sequences remain exact in the larger category
of all sheaves on $X$, and the generic fiber $\mathcal{E}\mapsto\mathcal{E}_{\eta}$
yields an exact and faithful $\otimes$-functor 
\[
(-)_{\eta}:\Bun_{X}\rightarrow\Vect_{E(X)}
\]
which induces an isomorphism between the poset $\Sub(\mathcal{E})$
of strict subobjects of $\mathcal{E}$ in $\Bun_{X}$ and the poset
$\Sub(\mathcal{E}_{\eta})$ of $E(X)$-subspaces of $\mathcal{E}_{\eta}$.

\subsubsection{Newton slope filtrations\label{sub:Newton4BundlesOnCurve}}

The usual rank and degree functions 
\[
\rank:\sk\,\Bun_{X}\rightarrow\mathbb{N}\quad\mbox{and}\quad\deg:\sk\,\Bun_{X}\rightarrow\mathbb{Z}
\]
are additive on short exact sequences in $\Bun_{X}$, and they are
respectively constant and non-decreasing on mono-epis in $\Bun_{X}$.
More precisely, if $f:\mathcal{E}_{1}\rightarrow\mathcal{E}_{2}$
is a mono-epi, then $\rank(\mathcal{E}_{1})=\rank(\mathcal{E}_{2})$
and $\deg(\mathcal{E}_{1})\leq\deg(\mathcal{E}_{2})$ with equality
if and only if $f$ is an isomorphism. These functions yield a Harder-Narasimhan
filtration on $\Bun_{X}$, the Newton filtration $\mathcal{F}_{N}$
with slopes $\mu=\deg/\rank$ in $\mathbb{Q}$. The filtration $\mathcal{F}_{N}(\mathcal{E})$
on $\mathcal{E}\in\Bun_{X}$ is non-canonically split. More precisely
for every $\mu\in\mathbb{Q}$, the full subcategory of semi-stable
vector bundles of slope $\mu$ is abelian, equivalent to the category
of right $D_{\mu}$-vector spaces, where $D_{\mu}$ is the semi-simple
division $E$-algebra whose invariant is the class of $\mu$ in $\mathbb{Q}/\mathbb{Z}$.
We denote by $\mathcal{O}_{X}(\mu)$ its unique simple object. Then
for every vector bundle $\mathcal{E}$ on $X$, there is unique sequence
$\mu_{1}\geq\cdots\geq\mu_{s}$ in $\mathbb{Q}$ for which there is
a (non-unique) isomorphism $\oplus_{i=1}^{s}\mathcal{O}_{X}(\mu_{i})\simeq\mathcal{E}$,
and any such isomorphism maps $\oplus_{i:\mu_{i}\geq\gamma}\mathcal{O}_{X}(\mu_{i})$
to $\mathcal{F}_{N}^{\gamma}(\mathcal{E})$ for every $\gamma\in\mathbb{Q}$.
We denote by $t_{N}(\mathcal{E})\in\mathbb{Q}_{\geq}^{r}$ the type
of $\mathcal{F}_{N}(\mathcal{E})$, where $r=\rank(\mathcal{E})$. 
\begin{prop}
The Newton filtration is compatible with tensor products, duals, symmetric
and exterior powers in $\Bun_{X}$. For any exact sequence in $\Bun_{X}$,
\[
0\rightarrow\mathcal{E}_{1}\rightarrow\mathcal{E}_{2}\rightarrow\mathcal{E}_{3}\rightarrow0
\]
set $r_{i}=\rank\,\mathcal{E}_{i}$ and view $t_{N}(\mathcal{E}_{i})$
as a concave function $f_{i}:[0,r_{i}]\rightarrow\mathbb{R}$. Then
\[
f_{1}\ast f_{3}(s)\geq f_{2}(s)\geq\begin{cases}
f_{1}(s) & \mbox{if }0\leq s\leq r_{1}\\
f_{1}(r_{1})+f_{3}(s-r_{1}) & \mbox{if }r_{1}\leq s\leq r_{2}
\end{cases}
\]
with equality for $s=0$ and $s=r_{2}$. In particular,\textup{
\[
\begin{array}{ccccc}
t_{N}^{\max}(\mathcal{E}_{1}) & \leq & t_{N}^{\max}(\mathcal{E}_{2}) & \leq & \max\left\{ t_{N}^{\max}(\mathcal{E}_{1}),t_{N}^{\max}(\mathcal{E}_{3})\right\} ,\\
t_{N}^{\min}(\mathcal{E}_{3}) & \geq & t_{N}^{\min}(\mathcal{E}_{2}) & \geq & \min\left\{ t_{N}^{\min}(\mathcal{E}_{1}),t_{N}^{\min}(\mathcal{E}_{3})\right\} ,
\end{array}
\]
\[
\mbox{and}\qquad\qquad t_{N}(\mathcal{E}_{2})\leq t_{N}(\mathcal{E}_{1})\ast t_{N}(\mathcal{E}_{3})\quad\mbox{in}\quad\mathbb{Q}_{\geq}^{r_{2}}.\qquad\qquad
\]
}\textup{\emph{Moreover, $t_{N}(\mathcal{E}_{2})=t_{N}(\mathcal{E}_{1})\ast t_{N}(\mathcal{E}_{3})$
if and only if for every $\gamma\in\mathbb{Q}$, 
\[
0\rightarrow\mathcal{F}_{N}^{\gamma}(\mathcal{E}_{1})\rightarrow\mathcal{F}_{N}^{\gamma}(\mathcal{E}_{2})\rightarrow\mathcal{F}_{N}^{\gamma}(\mathcal{E}_{3})\rightarrow0
\]
is exact.}}\end{prop}
\begin{proof}
The compatibility of $\mathcal{F}_{N}$ with $\otimes$-products and
duals comes from~\cite[5.6.23]{FaFo15}. Since $\Bun_{X}$ is an
$E$-linear category, the compatibility of $\mathcal{F}_{N}$ with
symmetric and exterior powers follows from its additivity and compatibility
with $\otimes$-products. For the remaining assertions, see~\cite[Proposition 21]{Co16}
or \cite[4.4.4]{An09}.
\end{proof}

\subsubsection{Modifications of vector bundles}

We denote by $\Modif_{X}$ the category of triples 
\[
\underline{\mathcal{E}}=(\mathcal{E}_{1},\mathcal{E}_{2},f)
\]
where $\mathcal{E}_{1}$ and $\mathcal{E}_{2}$ are vector bundles
on $X$ while $f$ is an isomorphism
\[
f:\mathcal{E}_{1}\vert_{X\setminus\{\infty\}}\rightarrow\mathcal{E}_{2}\vert_{X\setminus\{\infty\}}.
\]
A morphism $F:\underline{\mathcal{E}}\rightarrow\underline{\mathcal{E}}'$
is a pair of morphisms $F_{i}:\mathcal{E}_{i}\rightarrow\mathcal{E}_{i}^{\prime}$
with $F_{2}\circ f=f'\circ F_{1}$. This defines a quasi-abelian $E$-linear
rigid $\otimes$-category with a Tate twist. The kernels and cokernels
are induced by those of $\Bun_{X}$. The neutral object is the trivial
modification $\underline{\mathcal{O}}_{X}=(\mathcal{O}_{X},\mathcal{O}_{X},\mathrm{Id})$,
the tensor product and duals are given by 
\[
\underline{\mathcal{E}}\otimes\underline{\mathcal{E}}^{\prime}\eqd(\mathcal{E}_{1}\otimes\mathcal{E}_{1}^{\prime},\mathcal{E}_{2}\otimes\mathcal{E}_{2}^{\prime},f\otimes f')\quad\mbox{and}\quad\underline{\mathcal{E}}^{\vee}\eqd(\mathcal{E}_{1}^{\vee},\mathcal{E}_{2}^{\vee},f^{\vee-1}).
\]
The Tate twist is $\underline{\mathcal{E}}\{i\}:=\underline{\mathcal{E}}\otimes\underline{\mathcal{O}}_{X}\{i\}$
where $\underline{\mathcal{O}}_{X}\{i\}:=\underline{\mathcal{O}}_{X}\{1\}^{\otimes i}$
with 
\[
\underline{\mathcal{O}}_{X}\{1\}\eqd\left(\mathcal{O}_{X}\otimes_{E}E(1),\mathcal{O}_{X}(1)\otimes_{E}E(1),\mathrm{can}\otimes\mathrm{Id}\right).
\]
Here $E(1)=E\otimes_{\mathbb{Z}_{p}}\mathbb{Z}_{p}(1)$ with $\mathbb{Z}_{p}(1)=\underleftarrow{\lim}\,\mu_{p^{n}}(C)$
and $\mathrm{can}:\mathcal{O}_{X}\hookrightarrow\mathcal{O}_{X}(1)$
is the canonical morphism, dual to the embedding $\mathcal{I}(\infty)\hookrightarrow\mathcal{O}_{X}$.
There are also symmetric and exterior powers, given by the following
formulae: for every $k\geq0$, 
\begin{eqnarray*}
\Sym^{k}(\underline{\mathcal{E}}) & \eqd & (\Sym^{k}\mathcal{E}_{1},\Sym^{k}\mathcal{E}_{2},\Sym^{k}f),\\
\Lambda^{k}(\underline{\mathcal{E}}) & \eqd & (\Lambda^{k}\mathcal{E}_{1},\Lambda^{k}\mathcal{E}_{2},\Lambda^{k}f).
\end{eqnarray*}
The generic fiber $\underline{\mathcal{E}}\mapsto\mathcal{E}_{1,\eta}$
yields an exact faithful $\otimes$-functor 
\[
(-)_{1,\eta}:\Modif_{X}\rightarrow\Vect_{E(X)}
\]
which induces an isomorphism between the poset $\Sub(\underline{\mathcal{E}})$
of strict subobjects of $\underline{\mathcal{E}}$ in $\Modif_{X}$
and the poset $\Sub(\mathcal{E}_{1,\eta})$ of $E(X)$-subspaces of
$\mathcal{E}_{1,\eta}$. We say that a modification $\underline{\mathcal{E}}=(\mathcal{E}_{1},\mathcal{E}_{2},f)$
is effective if $f$ extends to a (necessarily unique) morphism $f:\mathcal{E}_{1}\rightarrow\mathcal{E}_{2}$,
which is then a mono-epi in $\Bun_{X}$. For every $\underline{\mathcal{E}}$
in $\Modif_{X}$, 
\[
\underline{\mathcal{E}}\{i\}\mbox{ is effective for }i\gg0.
\]

\subsubsection{Hodge and Newton filtrations\label{sub:Hodge-and-NewtonForModif}}

For $\underline{\mathcal{E}}=(\mathcal{E}_{1},\mathcal{E}_{2},f)$
as above, we denote by 
\[
f_{dR}:\mathcal{E}_{1,dR}^{+}[\xi^{-1}]\rightarrow\mathcal{E}_{2,dR}^{+}[\xi^{-1}]
\]
the $B_{dR}$-isomorphism induced by $f$, where $\mathcal{E}_{i,dR}^{+}=\mathcal{E}_{i,\infty}^{\wedge}$
is the completed local stalk at $\infty$. For $i\in\{1,2\}$, the
Hodge filtration $\mathcal{F}_{H,i}(\underline{\mathcal{E}})$ is
the $\mathbb{Z}$-filtration induced by $\mathcal{E}_{3-i,dR}^{+}$
on the residue $\mathcal{E}_{i}(\infty)=\mathcal{E}_{i,dR}^{+}/\xi\mathcal{E}_{i,dR}^{+}$
of $\mathcal{E}_{i}$. Thus for every $\gamma\in\mathbb{Z}$, 
\[
\mathcal{F}_{H,1}^{\gamma}\eqd\frac{f_{dR}^{-1}(\xi^{\gamma}\mathcal{E}_{2,dR}^{+})\cap\mathcal{E}_{1,dR}^{+}+\xi\mathcal{E}_{1,dR}^{+}}{\xi\mathcal{E}_{1,dR}^{+}},\quad\mathcal{F}_{H,2}^{\gamma}\eqd\frac{f_{dR}(\xi^{\gamma}\mathcal{E}_{1,dR}^{+})\cap\mathcal{E}_{2,dR}^{+}+\xi\mathcal{E}_{2,dR}^{+}}{\xi\mathcal{E}_{2,dR}^{+}}.
\]
These are filtrations with opposed types $t_{H,i}(\underline{\mathcal{E}})\in\mathbb{\mathbb{Z}}_{\geq}^{r}$,
where 
\[
r=\rank(\underline{\mathcal{E}})=\rank(\mathcal{E}_{1})=\rank(\mathcal{E}_{2}).
\]
We denote by $\mathcal{F}_{N,i}(\underline{\mathcal{E}})$ the Newton
filtration on $\mathcal{E}_{i}$ with type $t_{N,i}(\mathcal{\underline{E}})\in\mathbb{Q}_{\geq}^{r}$.
Thus
\[
\begin{array}{rcl}
t_{N,1}\left(\underline{\mathcal{E}}^{\vee}\right) & = & t_{N,1}\left(\underline{\mathcal{E}}\right)^{\iota}\\
t_{N,2}\left(\underline{\mathcal{E}}^{\vee}\right) & = & t_{N,2}\left(\underline{\mathcal{E}}\right)^{\iota}\\
t_{H,1}\left(\underline{\mathcal{E}}^{\vee}\right) & = & t_{H,1}\left(\underline{\mathcal{E}}\right)^{\iota}\\
t_{H,2}\left(\underline{\mathcal{E}}^{\vee}\right) & = & t_{H,2}\left(\underline{\mathcal{E}}\right)^{\iota}
\end{array}\quad\mbox{and}\quad\begin{array}{rcl}
t_{N,1}\left(\underline{\mathcal{E}}\{i\}\right) & = & t_{N,1}\left(\underline{\mathcal{E}}\right)\\
t_{N,2}\left(\underline{\mathcal{E}}\{i\}\right) & = & t_{N,2}\left(\underline{\mathcal{E}}\right)+(i,\cdots,i)\\
t_{H,1}\left(\underline{\mathcal{E}}\{i\}\right) & = & t_{H,1}\left(\underline{\mathcal{E}}\right)+(i,\cdots,i)\\
t_{H,2}\left(\underline{\mathcal{E}}\{i\}\right) & = & t_{H,2}\left(\underline{\mathcal{E}}\right)-(i,\cdots,i)
\end{array}
\]
The filtrations $\mathcal{F}_{N,i}$ and $\mathcal{F}_{H,i}$ are
compatible with tensor products, duals, symmetric and exterior powers.
In particular for every $0\leq k\leq r$, 
\[
t_{H,i}^{\max}(\Lambda^{k}\underline{\mathcal{E}})=t_{H,i}(\underline{\mathcal{E}})(k)\quad\mbox{and}\quad t_{N,i}^{\max}(\Lambda^{k}\underline{\mathcal{E}})=t_{N,i}(\underline{\mathcal{E}})(k)
\]
viewing the right hand side terms as functions on $[0,r]$. Also,
$\underline{\mathcal{E}}$ is effective if and only if the slopes
of $\mathcal{F}_{H,1}$ (resp. $\mathcal{F}_{H,2}$) are non-negative
(resp.~non-positive), in which case $t_{H,1}(\underline{\mathcal{E}})$
is the type $t(\mathcal{Q})$ of the torsion $\mathcal{O}_{X}$-module
$\mathcal{Q}=\mathcal{E}_{2}/f(\mathcal{E}_{1})$ supported at $\infty$,
which means that if $t_{H,1}=(n_{1}\geq\cdots\geq n_{r})\in\mathbb{N}_{\geq}^{r}$,
then 
\[
\mathcal{E}_{2}/f(\mathcal{E}_{1})\simeq\mathcal{O}_{X,\infty}/\mathfrak{m}_{\infty}^{n_{1}}\oplus\cdots\oplus\mathcal{O}_{X,\infty}/\mathfrak{m}_{\infty}^{n_{r}}\simeq B_{dR}^{+}/\xi^{n_{1}}B_{dR}^{+}\oplus\cdots\oplus B_{dR}^{+}/\xi^{n_{r}}B_{dR}^{+}.
\]

\begin{prop}
\label{prop:CompNewtonHodgeModif}For every modification $\underline{\mathcal{E}}$
on $X$ of rank $r\in\mathbb{N}$, 
\[
t_{N,2}(\underline{\mathcal{E}})\leq t_{N,1}(\underline{\mathcal{E}})+t_{H,1}(\underline{\mathcal{E}})\quad\mbox{in}\quad\mathbb{Q}_{\geq}^{r}.
\]
\end{prop}
\begin{proof}
Using a Tate twist, we may assume that $\underline{\mathcal{E}}$
is effective. The left and right-hand side concave polygons then already
have the same terminal point, since $\deg(\mathcal{E}_{2})=\deg(\mathcal{E}_{1})+\deg(\mathcal{Q})$
where $\mathcal{Q}=\mathcal{E}_{2}/f(\mathcal{E}_{1})$. By the formula
for the exterior powers, it is then sufficient to establish that 
\[
t_{N}^{\max}(\mathcal{E}_{2})\leq t_{N}^{\max}(\mathcal{E}_{1})+t^{\max}(\mathcal{Q}).
\]
Let $\underline{\mathcal{E}}^{\prime}=(\mathcal{E}_{1}^{\prime},\mathcal{E}_{2}^{\prime},f^{\prime})$
where $\mathcal{E}_{2}^{\prime}$ is the first (smallest) step of
$\mathcal{F}_{N}(\mathcal{E}_{2})$, $\mathcal{E}_{1}^{\prime}=f^{-1}(\mathcal{E}_{2}^{\prime})$
and $f^{\prime}=f\vert\mathcal{E}_{1}^{\prime}$. Set $\mathcal{Q}'=\mathcal{E}_{2}^{\prime}/f'(\mathcal{E}_{1}^{\prime})$.
Then $\mathcal{E}_{2}^{\prime}$ is semi-stable of slope $\mu=t_{N}^{\max}(\mathcal{E}_{2})$
and $\deg\mathcal{E}_{2}^{\prime}=\deg\mathcal{E}_{1}^{\prime}+\deg\mathcal{Q}'$,
thus $t_{N}(\mathcal{E}_{2}^{\prime})\leq t_{N}(\mathcal{E}_{1}^{\prime})+t(\mathcal{Q}')$
by concavity of the sum and equality of the terminal points. Considering
the first (largest) slopes, we find that $\mu\leq t_{N}^{\max}(\mathcal{E}_{1}^{\prime})+t^{\max}(\mathcal{Q}')$.
But $\mathcal{E}_{1}^{\prime}\subset\mathcal{E}_{1}$ and $\mathcal{Q}'\subset\mathcal{Q}$,
thus 
\[
t_{N}^{\max}(\mathcal{E}_{1}^{\prime})\leq t_{N}^{\max}(\mathcal{E}_{1})\quad\mbox{and}\quad t^{\max}(\mathcal{Q}')\leq t^{\max}(\mathcal{Q}).
\]
This yields the desired inequality.
\end{proof}

\subsubsection{Admissible modifications\label{sub:AdmModifDef}}

Let $\Modif_{X}^{ad}$ be the full subcategory of $\Modif_{X}$ whose
objects are the modifications $\underline{\mathcal{E}}=(\mathcal{E}_{1},\mathcal{E}_{2},f)$
such that $\mathcal{E}_{1}$ is semi-stable of slope $0$, i.e.~$t_{N,1}(\underline{\mathcal{E}})=t_{N}(\mathcal{E}_{1})=0$.
This is a quasi-abelian $E$-linear rigid $\otimes$-category with
Tate twists. The kernels, cokernels, duals, $\otimes$-products, Tate
twist, symmetric and exterior powers are induced by those of $\Modif_{X}$.
On $\Modif_{X}^{ad}$, we set 
\[
\mathcal{F}_{N}\eqd\mathcal{F}_{N,2},\quad\mathcal{F}_{H}\eqd\mathcal{F}_{H,1},\quad t_{N}\eqd t_{N,2}\quad\mbox{and}\quad t_{H}\eqd t_{H,1}.
\]

\begin{prop}
\label{prop:OnAdmModif_NewtonHodge}For every admissible modification
$\underline{\mathcal{E}}$ of rank $r\in\mathbb{N}$, 
\[
t_{N}(\underline{\mathcal{E}})\leq t_{H}(\underline{\mathcal{E}})\quad\mbox{in}\quad\mathbb{Q}_{\geq}^{r}.
\]
\end{prop}
\begin{proof}
This is the special case of proposition~\ref{prop:CompNewtonHodgeModif}
where $t_{N,1}(\underline{\mathcal{E}})=0$.
\end{proof}
\noindent The restriction of the generic fiber functor $(-)_{1,\eta}:\Modif_{X}\rightarrow\Vect_{E(X)}$
to the full subcategory $\Modif_{X}^{ad}$ of $\Mod_{X}$ descends
to an exact $E$-linear faithful $\otimes$-functor
\[
\omega:\Modif_{X}^{ad}\rightarrow\Vect_{E},\qquad\omega(\underline{\mathcal{E}})=\Gamma(X,\mathcal{E}_{1})
\]
inducing an isomorphism between the poset $\Sub^{ad}(\underline{\mathcal{E}})\subset\Sub(\underline{\mathcal{E}})$
of strict subobjects of $\underline{\mathcal{E}}$ in $\Modif_{X}^{ad}$
and the poset $\Sub(\omega(\underline{\mathcal{E}}))\subset\Sub(\mathcal{E}_{1,\eta})$
of $E$-subspaces of $\omega(\underline{\mathcal{E}})$.

\subsubsection{\label{sub:DefFarguesFiltrOnModif}The Fargues filtration}

The rank and degree functions 
\[
\rank:\sk\,\Modif_{X}^{ad}\rightarrow\mathbb{N}\quad\mbox{and}\quad\deg:\sk\,\Modif_{X}^{ad}\rightarrow\mathbb{Z}
\]
which are respectively defined by
\[
\begin{aligned}\rank(\underline{\mathcal{E}}) & \eqd\rank(\mathcal{E}_{1})=\rank(\mathcal{E}_{2})=\dim_{E}\omega(\underline{\mathcal{E}})\\
\deg(\underline{\mathcal{E}}) & \eqd\deg\mathcal{E}_{2}=\deg\mathcal{F}_{N}(\underline{\mathcal{E}})=\deg\mathcal{F}_{H}(\underline{\mathcal{E}})
\end{aligned}
\]
$ $are additive on short exact sequences in $\Modif_{X}^{ad}$, and
they are respectively constant and non-decreasing on mono-epis in
$\Modif_{X}^{ad}$. More precisely if $F=(F_{1},F_{2})$ is a mono-epi
$F:\underline{\mathcal{E}}\rightarrow\underline{\mathcal{E}}^{\prime},$
then $F_{1}:\mathcal{E}_{1}\rightarrow\mathcal{E}'_{1}$ is an isomorphism
and $F_{2}:\mathcal{E}_{2}\rightarrow\mathcal{E}_{2}^{\prime}$ is
a mono-epi in $\Bun_{X}$, thus $\deg(\underline{\mathcal{E}})=\deg(\mathcal{E}_{2})\leq\deg(\mathcal{E}_{2}^{\prime})=\deg(\underline{\mathcal{E}}^{\prime})$
with equality if and only if $F_{2}$ is an isomorphism in $\Bun_{X}$,
which amounts to $F=(F_{1},F_{2})$ being an isomorphism in $\Modif_{X}^{ad}$.
These rank and degree functions thus induce a Harder-Narasimhan filtration
on $\Modif_{X}^{ad}$, the Fargues filtration $\mathcal{F}_{F}$ with
slopes $\mu=\deg/\rank$ in $\mathbb{Q}$, and the full subcategory
of $\Modif_{X}^{ad}$ of semi-stable objects of slope $\mu$ is abelian.
We denote by $t_{F}(\underline{\mathcal{E}})$ the type of $\mathcal{F}_{F}(\underline{\mathcal{E}})$. 
\begin{prop}
\label{prop:OnAdmModif_FarguesExactness}Let $0\rightarrow\underline{\mathcal{E}}_{1}\rightarrow\underline{\mathcal{E}}_{2}\rightarrow\underline{\mathcal{E}}_{3}\rightarrow0$
be an exact sequence in $\Modif_{X}^{ad}$, set $r_{i}=\rank\,\underline{\mathcal{E}}_{i}$
and view $t_{F}(M_{i})$ as a concave function $f_{i}:[0,r_{i}]\rightarrow\mathbb{R}$.
Then 
\[
f_{1}\ast f_{3}(s)\geq f_{2}(s)\geq\begin{cases}
f_{1}(s) & \mbox{if }0\leq s\leq r_{1}\\
f_{1}(r_{1})+f_{3}(s-r_{1}) & \mbox{if }r_{1}\leq s\leq r_{2}
\end{cases}
\]
with equality for $s=0$ and $s=r_{2}$. In particular,\textup{
\[
\begin{array}{ccccc}
t_{F}^{\max}(\underline{\mathcal{E}}_{1}) & \leq & t_{F}^{\max}(\underline{\mathcal{E}}_{2}) & \leq & \max\left\{ t_{F}^{\max}(\underline{\mathcal{E}}_{1}),t_{F}^{\max}(\underline{\mathcal{E}}_{3})\right\} ,\\
t_{F}^{\min}(\underline{\mathcal{E}}_{3}) & \geq & t_{F}^{\min}(\underline{\mathcal{E}}_{2}) & \geq & \min\left\{ t_{F}^{\min}(\underline{\mathcal{E}}_{1}),t_{F}^{\min}(\underline{\mathcal{E}}_{3})\right\} ,
\end{array}
\]
\[
\mbox{and}\qquad\qquad t_{F}(\underline{\mathcal{E}}_{2})\leq t_{F}(\underline{\mathcal{E}}_{1})\ast t_{F}(\underline{\mathcal{E}}_{3})\quad\mbox{in}\quad\mathbb{Q}_{\geq}^{r_{2}}.\qquad\qquad
\]
}\textup{\emph{Moreover, $t_{F}(\underline{\mathcal{E}}_{2})=t_{F}(\underline{\mathcal{E}}_{1})\ast t_{F}(\underline{\mathcal{E}}_{3})$
if and only if for every $\gamma\in\mathbb{Q}$, 
\[
0\rightarrow\mathcal{F}_{F}^{\gamma}(\underline{\mathcal{E}}_{1})\rightarrow\mathcal{F}_{F}^{\gamma}(\underline{\mathcal{E}}_{2})\rightarrow\mathcal{F}_{F}^{\gamma}(\underline{\mathcal{E}}_{3})\rightarrow0
\]
is exact.}}\end{prop}
\begin{proof}
Again, see~\cite[Proposition 21]{Co16} or \cite[4.4.4]{An09}. \end{proof}
\begin{prop}
\label{prop:OnAdmModif_CompFarguesNewton}For every admissible modification
$\underline{\mathcal{E}}$ of rank $r\in\mathbb{N}$, 
\[
t_{F}(\underline{\mathcal{E}})\leq t_{N}(\underline{\mathcal{E}})\quad\mbox{in}\quad\mathbb{Q}_{\geq}^{r}.
\]
\end{prop}
\begin{proof}
The breaks of the concave polygon $t_{F}(\underline{\mathcal{E}})$
have coordinates
\[
\left(\rank,\deg\right)\left(\mathcal{F}_{F}^{\gamma}(\underline{\mathcal{E}})_{2}\right)\in\left\{ 0,\cdots,r\right\} \times\mathbb{Z}
\]
for $\gamma\in\mathbb{Q}$, where $\mathcal{F}_{F}^{\gamma}(\underline{\mathcal{E}})_{2}$
is a strict subobject of $\mathcal{E}_{2}$ in $\Bun_{X}$, equal
to $\mathcal{E}_{2}$ for $\gamma\ll0$. Thus by definition of $\mathcal{F}_{N}(\mathcal{E}_{2})$,
we find that $t_{F}(\underline{\mathcal{E}})$ lies below $t_{N}(\mathcal{E}_{2})=t_{N}(\underline{\mathcal{E}})$
and both polygons have the same terminal points, which proves the
proposition.
\end{proof}

\subsubsection{~}

Let $\underline{\mathcal{E}}=(\mathcal{E}_{1},\mathcal{E}_{2},\alpha)$
be an admissible modification and set $V=\Gamma(X,\mathcal{E}_{1})$,
so that $\mathcal{E}_{1,\eta}=V_{E(X)}$ and $\mathcal{E}_{1}(\infty)=V_{C}$.
We view $\mathcal{F}_{H}=\mathcal{F}_{H}(\underline{\mathcal{E}})$
as an element of $\mathbf{F}(V_{C})$, $\mathcal{F}_{N}^{\ast}=\alpha_{\eta}^{-1}(\mathcal{F}_{N}(\mathcal{E}_{2})_{\eta})$
as an element of $\mathbf{F}(V_{E(X)})$ and $\mathcal{F}_{F}^{\ast}=\Gamma(X,\mathcal{F}_{F}(\underline{\mathcal{E}})_{1})$
as an element of $\mathbf{F}(V)$. For every $\mathcal{F}\in\mathbf{F}(V_{E(X)})$,
define 
\[
\left\langle \mathcal{E}_{1},\mathcal{F}\right\rangle \eqd\sum_{\gamma\in\mathbb{R}}\gamma\deg\Gr_{\mathcal{F}}^{\gamma}\left(\mathcal{E}_{1}\right)\quad\mbox{and}\quad\left\langle \mathcal{E}_{2},\mathcal{F}\right\rangle \eqd\sum_{\gamma\in\mathbb{R}}\gamma\deg\Gr_{\mathcal{F}}^{\gamma}\left(\mathcal{E}_{2}\right).
\]
Here $\Gr_{\mathcal{F}}^{\gamma}(\mathcal{E}_{i}):=\mathcal{F}^{\geq\gamma}(\mathcal{E}_{i})/\mathcal{F}^{>\gamma}(\mathcal{E}_{i})$
where $\mathcal{F}^{\geq\gamma}(\mathcal{E}_{i})$ and $\mathcal{F}^{>\gamma}(\mathcal{E}_{i})$
are the strict subobjects of $\mathcal{E}_{i}$ with generic fiber
$\mathcal{F}^{\geq\gamma}$ and $\mathcal{F}^{>\gamma}$ in $V_{E(X)}=\mathcal{E}_{1,\eta}$
if $i=1$, or $\alpha_{\eta}(\mathcal{F}^{\geq\gamma})$ and $\alpha_{\eta}(\mathcal{F}^{>\gamma})$
in $\mathcal{E}_{2,\eta}$ if $i=2$. Thus whenever $\{\gamma_{s}>\cdots>\gamma_{0}\}\subset\mathbb{R}$
contains 
\[
\Jump(\mathcal{F})\eqd\left\{ \gamma\in\mathbb{R}:\Gr_{\mathcal{F}}^{\gamma}\neq0\right\} ,
\]
we have for any $i\in\{1,2\}$ the following equality: 
\[
\left\langle \mathcal{E}_{i},\mathcal{F}\right\rangle =\gamma_{0}\deg(\mathcal{E}_{i})+\sum_{j=1}^{s}\left(\gamma_{j}-\gamma_{j-1}\right)\deg\mathcal{F}^{\gamma_{j}}\left(\mathcal{E}_{i}\right).
\]
Since $\mathcal{E}_{1}$ is semi-stable of slope $0$, $\left\langle \mathcal{E}_{1},\mathcal{F}\right\rangle \leq0$
with equality if and only if each $\mathcal{F}^{\gamma_{j}}(\mathcal{E}_{1})$
is of degree $0$. We thus obtain: for every $\mathcal{F}\in\mathbf{F}(V_{E(X)})$,
\begin{eqnarray*}
\left\langle \mathcal{E}_{1},\mathcal{F}\right\rangle \geq0 & \iff & \left\langle \mathcal{E}_{1},\mathcal{F}\right\rangle =0,\\
 & \iff & \forall\gamma\in\mathbb{R},\,\mathcal{F}^{\gamma}(\mathcal{E}_{1})\mbox{ is semi-stable of slope }0,\\
 & \iff & \mathcal{F}\in\mathbf{F}(V).
\end{eqnarray*}

\begin{prop}
\label{prop:EquivCond4Fargues=00003DNewton}With notations as above,
the following conditions are equivalent:
\begin{eqnarray*}
\mathcal{F}_{F}^{\ast}=\mathcal{F}_{N}^{\ast} & \iff & \mathcal{F}_{N}^{\ast}\in\mathbf{F}(V),\\
 & \iff & \left\langle \mathcal{E}_{1},\mathcal{F}_{N}^{\ast}\right\rangle \geq0,\\
 & \iff & \left\langle \mathcal{E}_{1},\mathcal{F}_{N}^{\ast}\right\rangle =0,\\
 & \iff & \forall\gamma\in\mathbb{R},\,(\mathcal{F}_{N}^{\ast})^{\gamma}(\mathcal{E}_{1})\mbox{ is semi-stable of slope }0.
\end{eqnarray*}
If $\underline{\mathcal{E}}=(\mathcal{E}_{1},\mathcal{E}_{2},\alpha)$
is effective, $(\mathcal{F}_{N}^{\ast})^{\gamma}(\mathcal{E}_{1})=\alpha^{-1}(\mathcal{F}_{N}^{\gamma}(\mathcal{E}_{2}))$
thus also
\[
\mathcal{F}_{F}^{\ast}=\mathcal{F}_{N}^{\ast}\iff\forall\gamma\in\mathbb{R},\,\alpha^{-1}(\mathcal{F}_{N}^{\gamma}(\mathcal{E}_{2}))\mbox{ is semi-stable of slope }0.
\]
\end{prop}
\begin{proof}
By~\cite[Proposition 6]{Co16}: $(1)$ $\mathcal{F}_{N}^{\ast}$
is the unique element $\mathcal{F}$ of $\mathbf{F}(V_{E(X)})$ such
that $\left\langle \mathcal{E}_{2},\mathcal{G}\right\rangle \leq\left\langle \mathcal{F},\mathcal{G}\right\rangle $
for every $\mathcal{G}\in\mathbf{F}(V_{E(X)})$ with equality for
$\mathcal{G}=\mathcal{F}$, and $(2)$ $\mathcal{F}_{F}^{\ast}$ is
the unique element $f$ of $\mathbf{F}(V)$ such that $\left\langle \mathcal{E}_{2},g\right\rangle \leq\left\langle f,g\right\rangle $
for every $g\in\mathbf{F}(V)$ with equality for $g=f$. Thus $\mathcal{F}_{F}^{\ast}=\mathcal{F}_{N}^{\ast}\Leftrightarrow\mathcal{F}_{N}^{\ast}\in\mathbf{F}(V)$
and the proposition follows.
\end{proof}

\subsection{Hodge-Tate modules\label{sub:HodgeTateModDef}}

\subsubsection{~}

Let $\HT_{E}^{B_{dR}}$ be the category of pairs $(V,\Xi)$ where
$V$ is a finite $E$-vector space and $\Xi$ is a $B_{dR}^{+}$-lattice
in $V_{dR}=V\otimes_{E}B_{dR}$. A morphism $F:(V,\Xi)\rightarrow(V',\Xi')$
is an $E$-linear morphism $f:V\rightarrow V'$ whose $B_{dR}$-linear
extension $f_{dR}:V_{dR}\rightarrow V_{dR}^{\prime}$ satisfies $f_{dR}(\Xi)\subset\Xi'$.
The kernel and cokernel of $F$ are given by
\[
\ker(F)=\left(\ker(f),\ker(f_{dR})\cap\Xi\right)\quad\mbox{and}\quad\mbox{coker}(F)=\left(V'/\mathrm{im}(f),\Xi'/\mbox{im}(f_{dR})\cap\Xi'\right).
\]
This defines a quasi-abelian rigid $E$-linear $\otimes$-category
with tensor product 
\[
(V_{1},\Xi_{1})\otimes(V_{2},\Xi_{2})\eqd(V_{1}\otimes_{E}V_{2},\Xi_{1}\otimes_{B_{dR}^{+}}\Xi_{2}),
\]
neutral object $(E,B_{dR}^{+})$ and duals, symmetric and exterior
powers given by 
\[
(V,\Xi)\eqd(V^{\vee},\Xi^{\vee}),\quad\Sym^{k}(V,\Xi)\eqd(\Sym^{k}V,\Sym^{k}\Xi),\quad\Lambda^{k}(V,\Xi)\eqd(\Lambda^{k}V,\Lambda^{k}\Xi)
\]
where the tensor product constructions are over $E$ or $B_{dR}^{+}$.

\subsubsection{~}

There is an (exact) $\otimes$-equivalence of $\otimes$-categories
\[
\HTT:\Modif_{X}^{ad}\rightarrow\HT_{E}^{B_{dR}},\qquad\HTT(\mathcal{E}_{1},\mathcal{E}_{2},f)\eqd\left(\Gamma(X,\mathcal{E}_{1}),f_{dR}^{-1}(\mathcal{E}_{2,dR}^{+})\right).
\]
The Hodge filtration $\mathcal{F}_{H}(V,\Xi)$ is the $\mathbb{Z}$-filtration
which is induced by $\Xi$ on the residue $V_{C}=V\otimes_{E}C$ of
the standard lattice $V_{dR}^{+}=V\otimes_{E}B_{dR}^{+}$ of $V_{dR}$:
for $\gamma\in\mathbb{Z}$, 
\[
\mathcal{F}_{H}^{\gamma}(V,\Xi)\eqd\frac{V_{dR}^{+}\cap\xi^{\gamma}\Xi+\xi V_{dR}^{+}}{\xi V_{dR}^{+}}\quad\mbox{in}\quad V_{C}=\frac{V_{dR}^{+}}{\xi V_{dR}^{+}}.
\]
We denote by $t_{H}(V,\Xi)$ the type of $\mathcal{F}_{H}(V,\Xi)$.
The rank and degree functions
\[
\rank:\sk\,\HT_{E}^{B_{dR}}\rightarrow\mathbb{N}\quad\mbox{and}\quad\deg:\sk\,\HT_{E}^{B_{dR}}\rightarrow\mathbb{Z}
\]
are respectively given by
\[
\begin{aligned}\rank(V,\Xi) & \eqd\dim_{E}(V)=\rank_{B_{dR}^{+}}(\Xi),\\
\deg(V,\Xi) & \eqd\nu(V_{dR}^{+},\Xi)=\deg\mathcal{F}_{H}(V,\Xi).
\end{aligned}
\]
We denote by $\mathcal{F}_{F}(V,\Xi)$ the corresponding Fargues $\mathbb{Q}$-filtration,
with type $t_{F}(V,\Xi)$ in $\mathbb{Q}_{\geq}^{r}$ if $r=\dim_{E}V$.
The Tate object is $\HTT(\underline{\mathcal{O}}_{X}\{1\})=\left(E(1),\xi^{-1}E(1)_{dR}^{+}\right)$. 
\begin{prop}
Let $f:(V_{1},\Xi_{1})\rightarrow(V_{2},\Xi_{2})$ be a mono-epi in
$\HT_{E}^{B_{dR}}$, so that $f:V_{1}\rightarrow V_{2}$ is an isomorphism
and $f_{dR}:\Xi_{1}\rightarrow\Xi_{2}$ is injective with cokernel
$Q$ of finite length. If $r=\dim_{E}V_{1}=\dim_{E}V_{2}$, then for
every $s\in[0,r]$, 
\[
0\leq t_{F}(V_{2},\Xi_{2})(s)-t_{F}(V_{1},\Xi_{1})(s)\leq\length_{B_{dR}^{+}}(Q).
\]
with equality on the left (resp.~right) for $s=0$ (resp.~$s=r$).
In particular, 
\[
0\leq\left\{ \begin{array}{c}
t_{F}^{\max}(V_{2},\Xi_{2})-t_{F}^{\max}(V_{1},\Xi_{1})\\
t_{F}^{\min}(V_{2},\Xi_{2})-t_{F}^{\min}(V_{1},\Xi_{1})
\end{array}\right\} \leq\length_{B_{dR}^{+}}(Q).
\]
\end{prop}
\begin{proof}
This is analogous to proposition~\ref{prop:PseudoIsogInModt}.
\end{proof}

\subsubsection{~}

There is also an exact and fully faithful $\otimes$-functor from
the category $\HT_{E}^{B_{dR}}$ to the quasi-abelian $\otimes$-category
denoted by $\Norm_{E}^{B_{dR}}$ in \cite[\S 5.2]{Co16}, which maps
$(V,\Xi)$ to $(V,\alpha_{\Xi})$ where $\alpha_{\Xi}:V_{dR}\rightarrow\mathbb{R}_{+}$
is the gauge norm of the $B_{dR}^{+}$-lattice $\Xi\subset V_{dR}$.
This functor is plainly compatible with the rank and degree functions
of both categories (for the appropriate normalization of the valuation
on $B_{dR}$), and its essential image is stable under strict subobjects.
It is therefore also compatible with the corresponding Harder-Narasimhan
filtrations. Since the Harder-Narasimhan filtration on $\Norm_{E}^{B_{dR}}$
is compatible with tensor products, duals, symmetric and exterior
powers by \cite[Proposition 22]{Co16}, we obtain the following proposition:
\begin{prop}
\label{prop:OnHT+AdModif_Farguestensor}The Fargues filtrations $\mathcal{F}_{F}$
on $\HT_{E}^{B_{dR}}$ and $\Modif_{X}^{ad}$ are compatible with
tensor products, duals, symmetric and exterior powers.
\end{prop}

\subsubsection{~}

Fix an admissible modification $\underline{\mathcal{E}}$ of rank
$r$ and set $\HTT(\underline{\mathcal{E}})=\left(V,\Xi\right)$.
Then 
\[
\mathcal{F}_{H}=\mathcal{F}_{H}(\underline{\mathcal{E}})=\mathcal{F}_{H}(V,\Xi)
\]
is the $\mathbb{Z}$-filtration on $V_{C}=\mathcal{E}_{1}(\infty)$
which is denoted by $\loc(\alpha_{\Xi})$ in \cite[6.4]{Co16}, where
$\alpha_{\Xi}$ is the gauge norm of the $B_{dR}^{+}$-lattice $\Xi\subset V_{dR}$.
For any $\mathcal{F}\in\mathbf{F}(V_{E(X)})$, we set
\[
\left\langle \underline{\mathcal{E}},\mathcal{F}\right\rangle \eqd\left\langle \mathcal{E}_{2},\mathcal{F}\right\rangle -\left\langle \mathcal{E}_{1},\mathcal{F}\right\rangle .
\]
Thus if $\Jump(\mathcal{F})\subset\{\gamma_{s}>\cdots>\gamma_{0}\}\subset\mathbb{R}$
for some $s\in\mathbb{N}$, then
\[
\left\langle \underline{\mathcal{E}},\mathcal{F}\right\rangle =\gamma_{0}\deg(\mathcal{E}_{2})+\sum_{i=1}^{s}(\gamma_{i}-\gamma_{i-1})\left(\deg\mathcal{F}^{\gamma_{i}}(\mathcal{E}_{2})-\deg\mathcal{F}^{\gamma_{i}}(\mathcal{E}_{1})\right).
\]
Suppose first that $\underline{\mathcal{E}}=(\mathcal{E}_{1},\mathcal{E}_{2},f)$
is effective. Then $\mathcal{F}^{\gamma}(\mathcal{E}_{1})=f^{-1}(\mathcal{F}^{\gamma}(\mathcal{E}_{2}))$
and 
\begin{eqnarray*}
\left\langle \underline{\mathcal{E}},\mathcal{F}\right\rangle  & = & \gamma_{0}\deg(\mathcal{Q})+\sum_{i=1}^{s}(\gamma_{i}-\gamma_{i-1})\left(\deg\mathcal{F}^{\gamma_{i}}(\mathcal{Q})\right)\\
 & = & \sum_{\gamma\in\mathbb{R}}\gamma\deg\Gr_{\mathcal{F}}^{\gamma}(\mathcal{Q})
\end{eqnarray*}
where $\mathcal{F}^{\gamma}(\mathcal{Q})$ is the image of $\mathcal{F}^{\gamma}(\mathcal{E}_{2})$
in the torsion sheaf $\mathcal{Q}=\mathcal{E}_{2}/f(\mathcal{E}_{1})$
on $X$ and $\Gr_{\mathcal{F}}^{\gamma}(\mathcal{Q})=\mathcal{F}^{\geq\gamma}(\mathcal{Q})/\mathcal{F}^{>\gamma}(\mathcal{Q})$.
These are skyscraper sheaves supported at $\infty$, with 
\[
\Gamma(X,\mathcal{Q})=\Xi/V_{dR}^{+}\quad\mbox{and}\quad\Gamma(X,\mathcal{F}^{\gamma}(Q))=\Xi\cap\mathcal{F}_{dR}^{\gamma}/V_{dR}^{+}\cap\mathcal{F}_{dR}^{\gamma}
\]
where $\mathcal{F}_{dR}\in\mathbf{F}(V_{dR})$ is the base change
of $\mathcal{F}$ through $E(X)\hookrightarrow B_{dR}$. Therefore
\[
\left\langle \underline{\mathcal{E}},\mathcal{F}\right\rangle =\sum\gamma\nu\left(\Gr_{\mathcal{F}_{dR}}^{\gamma}V_{dR}^{+},\Gr_{\mathcal{F}_{dR}}^{\gamma}\Xi\right)=\left\langle \overrightarrow{\circ\alpha_{\Xi}},\mathcal{F}_{dR}\right\rangle 
\]
where $\circ$ is the gauge norm of $V_{dR}^{+}\subset V_{dR}$ and
the right-hand side term is the Busemann scalar product, see \cite[6.4.15]{Co16}.
This formula still holds true for a non-necessarily effective admissible
modification $\underline{\mathcal{E}}$, since indeed for every $i\in\mathbb{Z}$,
\[
\left\langle \underline{\mathcal{E}}\{i\},\mathcal{F}\right\rangle =\left\langle \underline{\mathcal{E}},\mathcal{F}\right\rangle +i\deg\mathcal{F}\quad\mbox{and}\quad\left\langle \overrightarrow{\circ\alpha_{\xi^{-i}\Xi}},\mathcal{F}_{dR}\right\rangle =\left\langle \overrightarrow{\circ\alpha_{\Xi}},\mathcal{F}_{dR}\right\rangle +i\deg(\mathcal{F}_{dR}).
\]
Returning thus to the general case, we now obtain: 
\[
\left\langle \underline{\mathcal{E}},\mathcal{F}\right\rangle =\left\langle \overrightarrow{\circ\alpha_{\Xi}},\mathcal{F}_{dR}\right\rangle \leq\left\langle \loc(\alpha_{\Xi}),\loc(\mathcal{F}_{dR})\right\rangle =\left\langle \mathcal{F}_{H},\mathcal{F}_{C}\right\rangle \leq\left\langle t_{H},\mathbf{t}(\mathcal{F})\right\rangle .
\]
Here $t_{H}=\mathbf{t}(\mathcal{F}_{H})$ is the Hodge type of $\underline{\mathcal{E}}$
and $\mathcal{F}_{C}=\loc(\mathcal{F}_{dR})$ is the $\mathbb{R}$-filtration
on $V_{C}=V_{dR}^{+}/\xi V_{dR}^{+}$ which is induced by the $\mathbb{R}$-filtration
$\mathcal{F}_{dR}$ on $V_{dR}$, so that 
\[
\mathbf{t}(\mathcal{F}_{C})=\mathbf{t}(\mathcal{F}_{dR})=\mathbf{t}(\mathcal{F})
\]
in $\mathbb{R}_{\geq}^{r}$. The last pairing is the standard scalar
product on $\mathbb{R}_{\geq}^{r}\subset\mathbb{R}^{r}$, and the
two inequalities come from~\cite[4.2 \& 5.5]{Co14}. For $\mathcal{F}=\mathcal{F}_{N}^{\ast}=\mathcal{F}_{N}^{\ast}(\underline{\mathcal{E}})$,
we obtain
\[
\left\langle \underline{\mathcal{E}},\mathcal{F}_{N}^{\ast}\right\rangle \leq\left\langle t_{H},t_{N}\right\rangle 
\]
where $t_{N}=\mathbf{t}(\mathcal{F}_{N})$ is the Newton type of $\underline{\mathcal{E}}$.
Now we have already seen that 
\[
\left\langle \underline{\mathcal{E}},\mathcal{F}_{N}^{\ast}\right\rangle =\left\langle \mathcal{E}_{2},\mathcal{F}_{N}^{\ast}\right\rangle -\left\langle \mathcal{E}_{1},\mathcal{F}_{N}^{\ast}\right\rangle =\left\Vert t_{N}\right\Vert ^{2}-\left\langle \mathcal{E}_{1},\mathcal{F}_{N}^{\ast}\right\rangle 
\]
with $\left\langle \mathcal{E}_{1},\mathcal{F}_{N}^{\ast}\right\rangle \leq0$,
and we thus obtain the following inequalities: 
\[
\left\Vert t_{N}\right\Vert ^{2}-\left\langle t_{H},t_{N}\right\rangle \leq\left\langle \mathcal{E}_{1},\mathcal{F}_{N}^{\ast}\right\rangle \leq0.
\]

\begin{prop}
\label{prop:Crit4Fargues=00003DNewton}With notations as above, $\left\Vert t_{N}\right\Vert ^{2}\leq\left\langle t_{H},t_{N}\right\rangle $
and 
\[
\left\Vert t_{N}\right\Vert ^{2}=\left\langle t_{H},t_{N}\right\rangle \Longrightarrow\left\langle \mathcal{E}_{1},\mathcal{F}_{N}^{\ast}\right\rangle =0\iff\mathcal{F}_{N}^{\ast}=\mathcal{F}_{F}.
\]
\end{prop}
\begin{proof}
This now follows from~proposition~\ref{prop:EquivCond4Fargues=00003DNewton}.
\end{proof}

\subsubsection{~}

Let $\HT_{\mathcal{O}_{E}}^{B_{dR}}$ be the category of pairs $(T,\Xi)$
where $T$ is a finite free $\mathcal{O}_{E}$-module and $\Xi$ is
a $B_{dR}^{+}$-lattice in $V_{dR}=T\otimes_{\mathcal{O}_{E}}B_{dR}=V\otimes_{E}B_{dR}$,
where $V=T\otimes_{\mathcal{O}_{E}}E$. A morphism $F:(T,\Xi)\rightarrow(T',\Xi')$
in $\HT_{\mathcal{O}_{E}}^{B_{dR}}$ is an $\mathcal{O}_{E}$-linear
morphism $f:T\rightarrow T'$ whose $B_{dR}$-linear extension $f_{dR}:V_{dR}\rightarrow V_{dR}^{\prime}$
satisfies $f_{dR}(\Xi)\subset\Xi'$. Any such morphism has a kernel
and a cokernel, which are respectively given by 
\[
\left(\ker(f),\ker(f_{dR})\cap\Xi\right)\quad\mbox{and}\quad\left(T'/f(T)^{sat},\Xi'/f_{dR}(V_{dR})\cap\Xi'\right)
\]
where $f(T)^{sat}/f(T)$ is the torsion submodule of $T'/f(T)$. This
defines a quasi-abelian rigid $\mathcal{O}_{E}$-linear $\otimes$-category
with tensor product 
\[
(T_{1},\Xi_{1})\otimes(T_{2},\Xi_{2})\eqd(T_{1}\otimes_{\mathcal{O}_{E}}T_{2},\Xi_{1}\otimes_{B_{dR}^{+}}\Xi_{2}),
\]
neutral object $(\mathcal{O}_{E},B_{dR}^{+})$ and duals, symmetric
and exterior powers given by 
\[
(T,\Xi)^{\vee}\eqd(T^{\vee},\Xi^{\vee}),\quad\Sym^{k}(T,\Xi)\eqd(\Sym^{k}T,\Sym^{k}\Xi),\quad\Lambda^{k}(T,\Xi)\eqd(\Lambda^{k}T,\Lambda^{k}\Xi)
\]
where the tensor product constructions are over $\mathcal{O}_{E}$
or $B_{dR}^{+}$. There is also a Tate twist in $\HT_{\mathcal{O}_{E}}^{B_{dR}}$,
corresponding to the Tate object $(\mathcal{O}_{E}(1),\xi^{-1}E(1)_{dR}^{+})$.

\subsubsection{~}

The exact and faithful $\otimes$-functor 
\[
\HT_{\mathcal{O}_{E}}^{B_{dR}}\rightarrow\HT_{E}^{B_{dR}}\qquad(T,\Xi)\mapsto(V,\Xi)\quad\mbox{with}\quad V=T\otimes_{\mathcal{O}_{E}}E
\]
induces a $\otimes$-equivalence of $\otimes$-categories 
\[
\HT_{\mathcal{O}_{E}}^{B_{dR}}\otimes E\rightarrow\HT_{E}^{B_{dR}}.
\]

\subsection{The Bhatt-Morrow-Scholze functor\label{sub:The-Bhatt-Morrow-Scholze-functor}}

\subsubsection{~}

Let $(M,\varphi_{M})$ be a finite free BKF-module over $A$. Then
$M\otimes_{A}A(C)$ is a finite free étale $\varphi$-module over
$A(C)=W_{\mathcal{O}_{E}}(C^{\flat})$, thus by \ref{sub:RisW(Kflat)},
\[
T\eqd\left\{ x\in M\otimes_{A}A(C):\phi_{M\otimes A(C)}(x)=x\right\} 
\]
is finite free over $\mathcal{O}_{E}$ and $T\hookrightarrow M\otimes A(C)$
extends to a $\varphi$-equivariant isomorphism
\[
T\otimes_{\mathcal{O}_{E}}A(C)\stackrel{\simeq}{\longrightarrow}M\otimes_{A}A(C).
\]
By~\cite[4.26]{BaMoSc16}, the latter descends to the subring $A[\frac{1}{\mu}]\subset A(C)$,
giving an isomorphism 
\[
\eta_{M}:T\otimes_{\mathcal{O}_{E}}A[{\textstyle \frac{1}{\mu}]}\stackrel{\simeq}{\longrightarrow}M[{\textstyle \frac{1}{\mu}]}.
\]
Note that since $\mu=[\epsilon]-1$ has residue $\epsilon-1\neq0$
in $C^{\flat}$, it is indeed invertible in $A(C)=W_{\mathcal{O}_{E}}(C^{\flat})$.
Tensoring with $A[\frac{1}{\mu}]\hookrightarrow B_{dR}$, we obtain
an isomorphism
\[
\eta_{M,dR}:T\otimes_{\mathcal{O}_{E}}B_{dR}\stackrel{\simeq}{\longrightarrow}M\otimes_{A}B_{dR}.
\]
This yields a Hodge-Tate module $(T,\Xi)$ over $\mathcal{O}_{E}$,
with 
\[
\Xi\eqd\eta_{M,dR}^{-1}(M\otimes_{A}B_{dR}^{+}).
\]
We have thus defined an $\mathcal{O}_{E}$-linear $\otimes$-functor
\[
\HTT^{\prime}:\Mod_{A,f}^{\varphi}\rightarrow\HT_{\mathcal{O}_{E}}^{B_{dR}},\qquad M\mapsto(T,\Xi).
\]
With $V=T\otimes_{\mathcal{O}_{E}}E$ as usual, we also denote by
\[
\HTT^{\prime}:\Mod_{A,f}^{\varphi}\otimes E\rightarrow\HT_{E}^{B_{dR}},\qquad M\otimes E\mapsto(V,\Xi)
\]
the induced $E$-linear $\otimes$-functor.

\subsubsection{Compatibility with Hodge filtrations}

Since $\xi'=\frac{\varphi(\mu)}{\mu}$ is already invertible in $A[\frac{1}{\varphi(\mu)}]$,
there is a commutative diagram whose first square is made of isomorphisms,

\[
\xymatrix{T\otimes_{\mathcal{O}_{E}}A[{\textstyle \frac{1}{\mu}}]\ar[dd]_{\mathrm{Id}\otimes\varphi}\ar[r]^{\eta_{M}} & M[{\textstyle \frac{1}{\mu}}]\ar[d]^{\varphi}\ar@{^{(}->}[r] & M\otimes A(C)\ar[dd]^{\phi_{N}}\\
 & \varphi^{\ast}M[{\textstyle \frac{1}{\varphi(\mu)}}]\ar[d]^{\varphi_{M}}\\
T\otimes_{\mathcal{O}_{E}}A[{\textstyle \frac{1}{\varphi(\mu)}]}\ar[r]^{\eta_{M}} & M[{\textstyle \frac{1}{\varphi(\mu)}}]\ar@{^{(}->}[r] & M\otimes A(C)
}
\]
This first square induces yet another commutative diagram of isomorphisms
\[
\xymatrix{T\otimes_{\mathcal{O}_{E}}B_{dR}\ar[dd]_{\mathrm{Id}\otimes\varphi}\ar[r]^{\eta_{M,dR}} & M\otimes_{A}B_{dR}\ar[d]^{\varphi\otimes\varphi}\\
 & \varphi^{\ast}M\otimes_{A}B_{dR}^{\prime}\ar[d]^{\varphi_{M}}\\
T\otimes_{\mathcal{O}_{E}}B_{dR}^{\prime}\ar[r]^{\eta_{M,dR}^{\prime}} & M\otimes_{A}B_{dR}^{\prime}
}
\]
with notations as in \ref{sub:RisA=00005B1/pi=00005D}. Restricting
to lattices, we obtain the following commutative diagrams of isomorphisms
(for the second diagram, note that $\mu\in(B_{dR}^{\prime+})^{\times}$):
\[
\xyC{1pc}\xymatrix{\Xi\ar[dd]_{\mathrm{Id}\otimes\varphi}\ar[r] & M\otimes_{A}B_{dR}^{+}\ar[d]^{\varphi\otimes\varphi} &  & V\otimes_{E}B_{dR}^{+}\ar[dd]_{\mathrm{Id}\otimes\varphi}\ar[r] & \eta_{M,dR}(V\otimes_{E}B_{dR}^{+})\ar[d]^{\varphi\otimes\varphi}\\
 & \varphi^{\ast}M\otimes_{A}B_{dR}^{\prime+}\ar[d]^{\varphi_{M}} & \mbox{and} &  & \varphi_{M}^{-1}(M\otimes_{A}B_{dR}^{\prime+})\ar[d]^{\varphi_{M}}\\
(\mathrm{Id}\otimes\varphi)(\Xi)\ar[r] & \varphi_{M}(\varphi^{\ast}M\otimes_{A}B_{dR}^{\prime+}) &  & V\otimes_{E}B_{dR}^{\prime+}\ar[r] & M\otimes_{A}B_{dR}^{\prime+}
}
\]
It follows that our various Hodge $\mathbb{Z}$-filtrations
\[
\begin{array}{c}
\mathcal{F}_{H}(M\otimes E)\quad\mbox{on}\quad M\otimes_{A}C^{\prime},\\
\mathcal{F}_{H}^{\prime}(M\otimes E)\quad\mbox{on}\quad M\otimes_{A}C,
\end{array}\,\mbox{and}\,\begin{array}{l}
\mathcal{F}_{H}(V,\Xi)=\mathcal{F}(V\otimes B_{dR}^{+},\Xi)\quad\mbox{on}\quad V\otimes_{E}C,\\
\mathcal{F}_{H}^{\prime}(V,\Xi)=\mathcal{F}(\Xi,V\otimes B_{dR}^{+})\quad\mbox{on}\quad\Xi\otimes_{B_{dR}^{+}}C.
\end{array}
\]
are related as follows:
\[
\begin{array}{rclll}
\mathcal{F}_{H}(M\otimes E) & = & \eta_{M,C}^{\prime}\left(\mathcal{F}_{H}(V,\Xi)\otimes_{C}C'\right) & \mbox{on} & M\otimes_{A}C^{\prime},\\
\mathcal{F}_{H}^{\prime}(M\otimes E) & = & \eta_{M,C}\left(\mathcal{F}_{H}^{\prime}(V,\Xi)\right) & \mbox{on} & M\otimes_{A}C
\end{array}
\]
where $\varphi:C\rightarrow C'$ is the residue of $\varphi:B_{dR}^{+}\rightarrow B_{dR}^{\prime+}$
and the isomorphisms 
\[
\eta_{M,C}:\Xi\otimes_{B_{dR}^{+}}C\stackrel{\simeq}{\longrightarrow}M\otimes_{A}C\quad\mbox{and}\quad\eta_{M,C'}:T\otimes_{\mathcal{O}_{E}}C'\stackrel{\simeq}{\longrightarrow}M\otimes_{A}C'
\]
are respectively induced by 
\[
\eta_{M,dR}:\Xi\stackrel{\simeq}{\longrightarrow}M\otimes_{A}B_{dR}^{+}\quad\mbox{and}\quad\eta_{M,dR}^{\prime}:T\otimes_{\mathcal{O}_{E}}B_{dR}^{\prime+}\stackrel{\simeq}{\longrightarrow}M\otimes_{A}B_{dR}^{\prime+}.
\]

\subsubsection{Compatibility with Tate objects}

The Tate object of $\Mod_{A,f}^{\varphi}$ is given by 
\[
A\{1\}=\left({\textstyle \frac{1}{\mu}}A\otimes\mathcal{O}_{E}(1),\varphi\otimes\mathrm{Id}\right).
\]
Thus since $\mu$ is invertible in $A(C)$, 
\[
A\{1\}(C^{\flat})=\left(A(C)\otimes\mathcal{O}_{E}(1),\varphi\otimes\mathrm{Id}\right).
\]
Since $\mathcal{O}_{E}=A(C)^{\varphi=\mathrm{Id}}$ and $t_{H}(A\{1\})=1$,
it follows that 
\[
\HTT^{\prime}\left(A\{1\}\right)=(\mathcal{O}_{E}(1),\xi^{-1}E(1)_{dR}^{+})
\]
is the Tate object of $\HT_{\mathcal{O}_{E}}^{B_{dR}}$.

\subsubsection{Fargue's theorem}

The following theorem was conjectured by Fargues in \cite{Fa15}. 
\begin{thm}
\label{thm:FarguesScholze}(Fargues, Scholze) The $\otimes$-functors
\[
\HTT^{\prime}:\Mod_{A,f}^{\varphi}\rightarrow\HT_{\mathcal{O}_{E}}^{B_{dR}}\quad\mbox{and}\quad\HTT^{\prime}:\Mod_{A,f}^{\varphi}\otimes E\rightarrow\HT_{E}^{B_{dR}}
\]
are equivalences of $\otimes$-categories.
\end{thm}
\noindent The full faithfulness is established in \cite[4.29]{BaMoSc16}.
A proof of the essential surjectivity is sketched in Scholze's Berkeley
lectures \cite{ScWe15}, where it is mostly attributed to Fargues.
An expanded and referenced version of this sketch is given in section~\ref{sub:AnalConstruction}
below.
\begin{cor}
The categories $\Mod_{A,f}^{\varphi}$ and $\Mod_{A,f}^{\varphi}\otimes E$
are quasi-abelian. 
\end{cor}
\noindent In particular, any morphism in these categories has a kernel
and a cokernel. But we have no explicit and manageable formulas for
them. Note also that we have two structures of exact category on $\Mod_{A,f}^{\varphi}$
and $\Mod_{A,f}^{\varphi}\otimes E$: the canonical structure which
any quasi-abelian category has, and the naive structure inherited
from the abelian category $\Mod_{A}^{\varphi}$. A three term complex
which is naively exact is also canonically exact, but the converse
is not true. We will investigate this in section~\ref{sub:Exactness}.

\subsection{The analytic construction\label{sub:AnalConstruction}}

\subsubsection{~}

In a category $\C$ with duals and effective object, let us say that
an object $X$ is anti-effective if its dual $X^{\vee}$ is effective.
We denote by $\C^{\geq}$ and $\C^{\leq}$ the full subcategories
of effective and anti-effective objects in $\C$.

\subsubsection{~}

We equip $A$ with its $(\pi,[\varpi])$-topology. Following \cite[12.2]{ScWe15},
we give names to four special points of $\Spa(A)=\Spa(A,A)$, labeled
by their residue fields: $y_{\mathbb{F}}$, $y_{C^{\flat}}$, $y_{L}$
and $y_{C}$, corresponding respectively to the trivial valuation
on the residue field $\mathbb{F}$ of $A$ and to the fixed valuations
on the $A$-algebras $C^{\flat}$, $L$ and $C$. Then $y_{\mathbb{F}}$
is the unique non-analytic point of $\Spa(A)$ and the complement
$\mathscr{Y}=\Spa(A)\setminus\{y_{\mathbb{F}}\}$ is equipped with
a continuous surjective map $\kappa:\mathscr{Y}\rightarrow[-\infty,+\infty]$
defined by 
\[
\kappa(y)\eqd\log_{q}\left(\frac{\log\left|[\varpi](\tilde{y})\right|}{\log\left|\pi(\tilde{y})\right|}\right)
\]
where $\tilde{y}$ is the maximal generalization of $y$, see~\cite[12.2]{ScWe15}.
We have 
\[
\kappa(y_{C^{\flat}})=-\infty,\qquad\kappa(y_{C})=0,\qquad\kappa(y_{L})=+\infty.
\]
The Frobenius $\varphi$ of $A$ induces an automorphism $\Spa(\varphi)$
of $\Spa(A)$ and $\mathscr{Y}$, which we still denote by $\varphi$.
It fixes $y_{\mathbb{F}}$, $y_{C^{\flat}}$ and $y_{L}$, but not
$y_{C}$. We set $y_{i}=\varphi^{i}(y_{C})$ for every $i\in\mathbb{Z}$,
so that $\kappa(y_{i})=i$ since more generally $\kappa(\varphi(y))=\kappa(y)+1$
for every $y\in\mathscr{Y}$. Thus $y_{0}=y_{C}$ while $y_{-1}$
corresponds to $A\twoheadrightarrow\mathcal{O}_{C^{\prime}}\hookrightarrow C'$.
For any interval $I\subset[-\infty,+\infty]$, we denote by $\mathscr{Y}_{I}$
the interior of the pre-image of $I$ under $\kappa$. We set 
\[
\mathscr{Y}^{+}\eqd\mathscr{Y}_{]-\infty,+\infty]},\quad\mathscr{Y}^{-}\eqd\mathscr{Y}_{[-\infty,+\infty[}\quad\mbox{and}\quad\mathscr{Y}^{\circ}\eqd\mathscr{Y}^{+}\cap\mathscr{Y}^{-}=\mathscr{Y}_{]-\infty,+\infty[}.
\]

\subsubsection{~}

By~\cite[13.1.1]{ScWe15}, $\mathscr{Y}$ is an honest -- or sheafy
-- adic space. This means that the presheaf $\mathscr{O}_{\mathscr{Y}}$
of analytic functions on $\mathscr{Y}$ is a sheaf on $\mathscr{Y}$.
Thus there is a well-defined $\otimes$-category $\Bun_{\mathscr{Y}_{I}}$
of vector bundles on $\mathscr{Y}_{I}$. A $\varphi$-equivariant
bundle on $\mathscr{Y}_{I}$ is a pair $(\mathscr{E},\varphi_{\mathscr{E}})$
where $\mathscr{E}$ is a vector bundle on $\mathscr{Y}_{I}$ and
$\varphi_{\mathscr{E}}:\varphi^{\ast}\mathscr{E}\vert\mathscr{Y}_{\varphi^{-1}(I)\cap I}\rightarrow\mathscr{E}\vert\mathscr{Y}_{\varphi^{-1}(I)\cap I}$
is an isomorphism. This defines a $\otimes$-category $\Bun_{\mathscr{Y}_{I}}^{\varphi}$.
By~\cite{Ke16}, the adic subspace $\mathscr{Y}^{\circ}$ of $\mathscr{Y}$
is strongly Noetherian. Thus for any interval $I\subset]-\infty,+\infty[$,
there is also a well-behaved abelian category $\Coh_{\mathscr{Y}_{I}}$
of coherent sheaves on $\mathscr{Y}_{I}$. A modification of vector
bundles on $\mathscr{Y}_{I}$ is a monomorphism $f:\mathscr{E}_{1}\hookrightarrow\mathscr{E}_{2}$
of vector bundles on $\mathscr{Y}_{I}$ whose cokernel is a coherent
sheaf supported at $\{y_{i}:i\in\mathbb{Z}\}\cap\mathscr{Y}_{I}$.
Similarly, there is a notion of $\varphi$-equivariant modification
of $\varphi$-equivariant vector bundles on $\mathscr{Y}_{I}$.

\subsubsection{~\label{sub:FromBKF2Modif*}}

By~\cite[3.6]{Ke16b}, the global section functor yields an equivalence
of $\otimes$-categories $\Gamma(\mathscr{Y},-):\Bun_{\mathscr{Y}}\rightarrow\Mod_{A,f}$
with inverse $M\mapsto M\otimes_{A}\mathscr{O}_{\mathscr{Y}}$. In
particular, every vector bundle $\mathscr{E}$ over $\mathscr{Y}$
is actually finite and free. Let $\Modif_{\mathscr{Y}}^{\star}$ be
the $\otimes$-category of pairs $(\mathscr{E},\psi_{\mathscr{E}})$
where $\mathscr{E}$ is a vector bundle on $\mathscr{Y}$ and $\psi_{\mathscr{E}}:\mathscr{E}\rightarrow\varphi^{\ast}\mathscr{E}$
is a modification supported at $\{y_{-1}\}$, i.e.~$\psi_{\mathscr{E}}$
is an isomorphism over $\mathscr{Y}\setminus\{y_{-1}\}$. Then plainly
\[
\xyC{2pc}\xyR{.2pc}\xymatrix{\Mod_{A,f}^{\varphi,\leq}\ar@{<->}[r] & \Modif_{\mathscr{Y}}^{\star}\\
(M,\varphi_{M})\ar@{|->}[r] & \left(\varphi_{M}^{-1}:M\rightarrow\varphi^{\ast}M\right)\otimes_{A}\mathscr{O}_{\mathscr{Y}}\\
\left(\Gamma(\mathscr{Y},\mathscr{E}),\Gamma(\mathscr{Y},\psi_{\mathscr{E}})^{-1}\right) & (\mathscr{E},\psi_{\mathscr{E}})\ar@{|->}[l]
}
\]
are mutually inverse equivalences of $\otimes$-categories.

\subsubsection{~\label{sub:FromModif*2ModifAn}}

Let $\Modif_{\mathscr{Y}^{-},\mathscr{Y}^{+}}^{\varphi,\geq}$ be
the $\otimes$-category of triples $(\mathscr{E}^{-},\mathscr{E}^{+},f_{\mathscr{E}})$
where $\mathscr{E}^{-}$ and $\mathscr{E}^{+}$ are $\varphi$-bundles
over respectively $\mathscr{Y}^{-}$ and $\mathscr{Y}^{+}$ while
$f_{\mathscr{E}}:\mathscr{E}_{-}\vert_{\mathscr{Y}^{\circ}}\rightarrow\mathscr{E}_{+}\vert_{\mathscr{Y}^{\circ}}$
is a $\varphi$-equivariant modification between their restriction
to $\mathscr{Y}^{\circ}=\mathscr{Y}^{+}\cap\mathscr{Y}^{-}$. We claim
that there are mutually inverse equivalences of $\otimes$-categories
\[
\xyR{.2pc}\xymatrix{\Modif_{\mathscr{Y}}^{\star}\ar@{<->}[r] & \Modif_{\mathscr{Y}^{-},\mathscr{Y}^{+}}^{\varphi,\geq}\\
(\mathscr{E},\psi_{\mathscr{E}})\ar@{<->}[r] & \left(\mathscr{E}^{-},\mathscr{E}^{+},f_{\mathscr{E}}\right)
}
\]
Starting on the left hand side, set $\mathscr{E}(i)=(\varphi^{i})^{\ast}\mathscr{E}$
and define $\theta_{i}:\mathscr{E}(i)\hookrightarrow\mathscr{E}(i+1)$
by $\theta_{i}=(\varphi^{i})^{\ast}(\theta_{0})$ for $i\in\mathbb{Z}$
with $\theta_{0}=\psi_{\mathscr{E}}:\mathscr{E}(0)\rightarrow\mathscr{E}(1)$.
Note that $\theta_{0}$ is a modification supported at $\{y_{-1}\}$,
thus $\theta_{i}$ is a modification supported at $\{y_{-i-1}\}$
for all $i\in\mathbb{Z}$. As in~\cite[\S 4.4]{Fa15}, the following
commutative diagram of vector bundles on $\mathscr{Y}$
\[
\xyR{2pc}\xyC{1.6pc}\xymatrix{\cdots\ar[r] & \varphi^{\ast}\mathscr{\mathscr{E}}(-2)\ar[r]^{\theta_{-1}} & \varphi^{\ast}\mathscr{\mathscr{E}}(-1)\ar[r]^{\theta_{0}} & \varphi^{\ast}\mathscr{E}(0)\ar[r]^{\theta_{1}} & \varphi^{\ast}\mathscr{\mathscr{E}}(1)\ar[r]^{\theta_{2}} & \varphi^{\ast}\mathscr{\mathscr{E}}(2)\ar[r]^{\theta_{3}} & \cdots\\
\cdots\ar[r] & \mathscr{\mathscr{E}}(-2)\ar[r]^{\theta_{-2}}\ar[u]^{\theta_{-2}} & \mathscr{\mathscr{E}}(-1)\ar[r]^{\theta_{-1}}\ar[u]^{\theta_{-1}} & \mathscr{E}(0)\ar[r]^{\theta_{0}}\ar[u]^{\theta_{0}} & \mathscr{\mathscr{E}}(1)\ar[r]^{\theta_{1}}\ar[u]^{\theta_{1}} & \mathscr{\mathscr{E}}(2)\ar[r]^{\theta_{2}}\ar[u]^{\theta_{2}} & \cdots
}
\]
defines two $\varphi$-equivariant sheaves on $\mathscr{Y}$, namely
\[
\begin{array}{rcccl}
\mathscr{\mathscr{E}}(-\infty) & \eqd & \underleftarrow{\lim}_{i\geq0}\mathscr{\mathscr{E}}(-i) & = & \cap_{i\geq0}\mathscr{\mathscr{E}}(-i)\\
\mathscr{\mathscr{E}}(+\infty) & \eqd & \underrightarrow{\lim}_{i\geq0}\mathscr{E}(+i) & = & \cup_{i\geq0}\mathscr{\mathscr{E}}(+i)
\end{array}
\]
whose inverse Frobenius mappings 
\[
\varphi_{\mathscr{\mathscr{E}}(-\infty)}^{-1}:\mathscr{\mathscr{E}}(-\infty)\rightarrow\varphi^{\ast}\mathscr{\mathscr{E}}(-\infty)\quad\mbox{and}\quad\varphi_{\mathscr{\mathscr{E}}(+\infty)}^{-1}:\mathscr{\mathscr{E}}(+\infty)\rightarrow\varphi^{\ast}\mathscr{\mathscr{E}}(+\infty)
\]
are induced by the vertical maps of the above diagram. Moreover, 
\[
\mathscr{\mathscr{E}}(-\infty)\hookrightarrow\mathscr{\mathscr{E}}(i)\quad\mbox{and}\quad\mathscr{\mathscr{E}}(i)\hookrightarrow\mathscr{\mathscr{E}}(+\infty)
\]
are respectively isomorphisms outside $\{y_{j}:j\geq-i\}$ and $\{y_{j}:j<-i\}$,
thus 
\[
\mathscr{E}^{-}\eqd\mathscr{E}(-\infty)\vert_{\mathscr{Y}^{-}}\quad\mbox{and}\quad\mathscr{E}^{+}\eqd\mathscr{E}(+\infty)\vert_{\mathscr{Y}^{+}}
\]
are $\varphi$-equivariant vector bundles over respectively $\mathscr{Y}^{-}$
and $\mathscr{Y}^{+}$, and 
\[
\left(f_{\mathscr{E}}:\mathscr{E}^{-}\vert_{\mathscr{Y}^{\circ}}\rightarrow\mathscr{E}^{+}\vert_{\mathscr{Y}^{\circ}}\right)\eqd\left(\mathscr{\mathscr{E}}(-\infty)\vert_{\mathscr{Y}^{\circ}}\rightarrow\mathscr{E}(0)\vert_{\mathscr{Y}^{\circ}}\rightarrow\mathscr{E}(+\infty)\vert_{\mathscr{Y}^{\circ}}\right)
\]
is a $\varphi$-equivariant modification as desired. 

Conversely, starting from $(\mathscr{E}^{-},\mathscr{E}^{+},f_{\mathscr{E}})$
on the right hand side, we define a vector bundle $\mathscr{E}$ on
$\mathscr{Y}$ by gluing $\mathscr{E}^{-}\vert_{\mathscr{Y}_{[-\infty,0[}}$
and $\mathscr{E}^{+}\vert_{\mathscr{Y}_{]-1,+\infty]}}$ along the
isomorphism induced by the restriction of $f_{\mathscr{E}}$ to $\mathscr{Y}_{]-1,0[}$.
Thus $\mathscr{E}\vert_{\mathscr{Y}^{\circ}}$ is the subsheaf of
$\mathscr{E}^{+}\vert_{\mathscr{Y}^{\circ}}$ made of those sections
whose restriction to $\mathscr{Y}_{]-\infty,0[}$ belong to the image
of $f_{\mathscr{E}}$. Since $f_{\mathscr{E}}$ is a $\varphi$-equivariant
modification, it follows that there is a commutative diagram
\[
\xymatrix{\mathscr{E}^{-}\vert_{\mathscr{Y}^{\circ}}\ar@{^{(}->}[r]\ar[d]^{\varphi_{\mathscr{E}^{-}}^{-1}} & \mathscr{E}\vert_{\mathscr{Y}^{\circ}}\ar@{^{(}->}[r]\ar[d]^{\psi_{\mathscr{E}}} & \mathscr{E}^{+}\vert_{\mathscr{Y}^{\circ}}\ar[d]^{\varphi_{\mathscr{E}^{+}}^{-1}}\\
\varphi^{\ast}\mathscr{E}^{-}\vert_{\mathscr{Y}^{\circ}}\ar@{^{(}->}[r] & \varphi^{\ast}\mathscr{E}\vert_{\mathscr{Y}^{\circ}}\ar@{^{(}->}[r] & \varphi^{\ast}\mathscr{E}^{+}\vert_{\mathscr{Y}^{\circ}}
}
\]
We extend $\psi_{\mathscr{E}}$ to $\mathscr{Y}$ by setting $\psi_{\mathscr{E}}:=\varphi_{\mathscr{E}^{-}}^{-1}$
on $\mathscr{E}_{[-\infty,-1[}$ and $\psi_{\mathscr{E}}:=\varphi_{\mathscr{E}^{+}}^{-1}$
on $\mathscr{Y}_{]-1,+\infty]}$. Therefore $\psi_{\mathscr{E}}:\mathscr{E}\rightarrow\varphi^{\ast}\mathscr{E}$
is an isomorphism away from $\kappa^{-1}(-1)\cap\{y_{i}\}=\{y_{-1}\}$,
i.e.~$\psi_{\mathscr{E}}$ is indeed a modification supported at
$y_{-1}$. 

One checks easily that these constructions yield mutually inverse
$\otimes$-functors.

\subsubsection{~\label{sub:VariantA=00005B1/pi=00005D}}

Starting with $\Mod_{A[\frac{1}{\pi}],f}^{\varphi}$, we may analogously
define $\otimes$-functors
\[
\xyR{.5pc}\xymatrix{\Mod_{A[\frac{1}{\pi}],f}^{\varphi,\leq}\ar[r] & \Modif_{\mathscr{Y}^{+}}^{\star}\ar@{<->}[r] & \Modif_{\mathscr{Y}^{\circ},\mathscr{Y}^{+}}^{\varphi,\geq}\\
(N,\varphi_{N})\ar@{|->}[r] & (\mathscr{E},\psi_{\mathscr{E}})\ar@{|->}[r] & (\mathscr{E}^{-},\mathscr{E}^{+},f_{\mathscr{E}})
}
\]
with the obvious definitions for the $\otimes$-categories $\Modif_{\mathscr{Y}^{+}}^{\star}$
and $\Modif_{\mathscr{Y}^{\circ},\mathscr{Y}^{+}}^{\varphi,\geq}$,
where 
\[
(\mathscr{E},\psi_{\mathscr{E}})\eqd(N,\varphi_{N}^{-1})\otimes_{A[\frac{1}{\pi}]}\mathscr{O}_{\mathscr{Y}^{+}}
\]
and $\mathscr{E}^{-}\in\Bun_{\mathscr{Y}^{\circ}}^{\varphi}$, $\mathscr{E}^{+}\in\Bun_{\mathscr{Y}^{+}}^{\varphi}$.
While $\Modif_{\mathscr{Y}^{+}}^{\star}\leftrightarrow\Modif_{\mathscr{Y}^{\circ},\mathscr{Y}^{+}}^{\varphi,\geq}$
are still mutually inverse equivalences of $\otimes$-categories,
it is not clear that the first functor is an equivalence. Indeed,
the functor $\Mod_{A[\frac{1}{\pi}],f}\rightarrow\Bun_{\mathscr{Y}^{+}}$
is already not full.

\subsubsection{~\label{sub:PhiBunOverY-}}

As in \cite[12.3.4]{ScWe15} and \cite[8.5.3]{KeLi15}, there are
equivalences of $\otimes$-categories
\[
\xymatrix{\Bun_{\mathscr{Y}^{-}}^{\varphi}\ar[r] & \Bun_{\mathscr{R}_{-}^{\mathrm{int}}}^{\varphi}\ar[r] & \Bun_{A(C)}^{\varphi}\ar[r] & \Bun_{\mathcal{O}_{E}}\\
\mathscr{E}\ar@{|->}[r] & \mathscr{E}_{y_{C^{\flat}}}\ar@{|->}[r] & \mathscr{E}_{y_{C^{\flat}}}^{\wedge}\ar@{|->}[r] & (\mathscr{E}_{y_{C^{\flat}}}^{\wedge})^{\varphi_{\mathscr{E}}=1}
}
\]
where $\mathscr{R}_{-}^{\mathrm{int}}:=\lim_{r\mapsto-\infty}\Gamma(\mathscr{Y}_{[-\infty,r]},\mathscr{O}_{\mathscr{Y}})$
is the local ring of $\mathscr{Y}$ at $y_{C^{\flat}}$; this is the
integral Robba ring, a Henselian discrete valuation ring with uniformizer
$\pi$, residue field $C^{\flat}$ and completion $A(C)=W_{\mathcal{O}_{E}}(C^{\flat})$~\cite[1.8.2]{FaFo15}.
The objects of the middle two categories are the finite free étale
$\varphi$-modules $(N,\varphi_{N})$ over the indicated local rings,
and the functor between them is the base change map (or $\pi$-adic
completion) with respect to $\mathscr{R}_{-}^{\mathrm{int}}\hookrightarrow A(C)$.
We have already encountered the last functor in~\ref{sub:RisW(Kflat)}:
it maps $(N,\varphi_{N})$ to $T=N^{\varphi_{N}=1}$. The inverse
$\otimes$-functor maps the finite free $\mathcal{O}_{E}$-module
$T$ to the ``constant'' $\varphi$-bundle $(\mathscr{E}^{-},\varphi_{\mathscr{E}^{-}})=(T\otimes_{\mathcal{O}_{E}}\mathscr{O}_{\mathscr{Y}^{-}},\mathrm{Id}\otimes\varphi)$
over $\mathscr{Y}^{-}$. In particular, every $\varphi$-bundle over
$\mathscr{Y}^{-}$ is actually finite free.

\subsubsection{~\label{sub:FFDiagram}}

There is also a commutative diagram of $\otimes$-categories \cite[\S 11.4]{FaFo15}
\[
\xyR{2.5pc}\xyC{5pc}\xymatrix{\Bun_{\mathscr{R}_{+}^{\mathrm{int}}}^{\varphi}\ar@{.>}[d]_{-\otimes\overline{B}} & \Bun_{\mathscr{Y}^{+}}^{\prime,\varphi}\ar@{.>}[l]_{(-)_{y_{L}}}\ar[r]^{-\vert\mathscr{Y}^{\circ}}\ar@<1ex>[d]^{\Gamma(\mathscr{Y}^{+},-)} & \Bun_{\mathscr{Y}^{\circ}}^{\varphi}\ar@<.7ex>[r]^{-/\varphi}\ar@<.7ex>[d]^{\Gamma(\mathscr{Y}^{\circ},-)} & \Bun_{\mathscr{X}}\ar@<1ex>[l]^{\pi^{\ast}}\ar@<1ex>[d]^{(-)^{alg}}\\
\Bun_{\overline{B}}^{\varphi} & \Bun_{B^{+}}^{\prime,\varphi}\ar[l]_{-\otimes\overline{B}}\ar[r]^{-\otimes B}\ar@<1ex>[u]^{-\otimes\mathscr{O}_{\mathscr{Y}^{+}}} & \Bun_{B}^{\varphi}\ar[r]^{\mathcal{E}}\ar@<1ex>[u]^{-\otimes\mathscr{O}_{\mathscr{Y}^{\circ}}} & \Bun_{X}\ar@<1ex>[u]^{(-)^{an}}
}
\]
in which all solid arrows are equivalences of $\otimes$-categories.
In the first line, 
\[
\mathscr{R}_{+}^{\mathrm{int}}\eqd\underrightarrow{\lim}\,\Gamma(\mathscr{Y}_{[r,+\infty]},\mathscr{O}_{\mathscr{Y}})
\]
is the analog of the integral Robba ring $\mathscr{R}_{-}^{\mathrm{int}}$
with $y_{C^{\flat}}$ replaced by $y_{L}$, and 
\[
\mathscr{X}\eqd\mathscr{Y}^{\circ}/\varphi^{\mathbb{Z}}
\]
is the adic version of the Fargues-Fontaine curve $X$, a strongly
noetherian analytic space. There is a morphism of locally ringed space
$\mathscr{X}\rightarrow X$ which induces pull-back $\otimes$-functors
$(-)^{an}:\Coh_{X}\rightarrow\Coh_{\mathscr{X}}$ and $(-)^{an}:\Bun_{X}\rightarrow\Bun_{\mathscr{X}}$.
The equivalence of $\otimes$-categories $\Bun_{\mathscr{Y}^{\circ}}^{\varphi}\leftrightarrow\Bun_{\mathscr{X}}$
maps a vector bundle on $\mathscr{X}$ to its pull-back through the
$\varphi$-invariant morphism $\pi:\mathscr{Y}^{\circ}\rightarrow\mathscr{X}$,
and maps a $\varphi$-bundle $\mathscr{E}$ on $\mathscr{Y}^{\circ}$
to the sheaf $\mathscr{E}/\varphi_{\mathscr{E}}$ of $\varphi_{\mathscr{E}}$-invariant
sections of $\pi_{\ast}\mathscr{E}$. We denote by $\mathscr{E}\mapsto\mathscr{E}(d)$
the Tate twists on $\Bun_{\mathscr{X}}$ and $\Bun_{\mathscr{Y}^{\circ}}^{\varphi}$
corresponding to the Tate objects $\mathscr{O}_{\mathscr{X}}(1)=\mathcal{O}_{X}(1)^{an}$
and $\mathscr{O}_{\mathscr{Y}^{\circ}}(1)=\pi^{\ast}\mathscr{O}_{\mathscr{X}}(1)$.
In the second line, the $A${[}$\frac{1}{\pi}]$-algebras 
\[
B\hookleftarrow B^{+}\twoheadrightarrow\overline{B}
\]
are defined in \cite[1.10]{FaFo15}. They are related to the adic
space $\mathscr{Y}$ by 
\[
B^{+}=\Gamma(\mathscr{Y}^{+},\mathscr{O}_{\mathscr{Y}})\quad\mbox{and}\quad B=\Gamma(\mathscr{Y}^{\circ},\mathscr{O}_{\mathscr{Y}}).
\]
Moreover, $\overline{B}$ is a local domain with residue field $L$
which is also a quotient of $\mathscr{R}_{+}^{\mathrm{int}}$. The
Fargues-Fontaine curve $X$ equals $\Proj(P)$ where $P:=\oplus_{d\geq0}P_{d}$
with 
\[
P_{d}\eqd\Gamma(X,\mathcal{O}_{X}(d))=\Gamma(\mathscr{X},\mathscr{O}_{\mathscr{X}}(d))=B^{\varphi=\pi^{d}}=(B^{+})^{\varphi=\pi^{d}}.
\]
The $\otimes$-functor $\mathcal{E}:\Bun_{B}^{\varphi}\rightarrow\Bun_{X}$
maps a finite projective étale $\varphi$-module $(N,\varphi_{N})$
to the quasi-coherent sheaf on $X$ associated with the graded $P$-module
$\oplus_{d\geq0}N^{\varphi_{N}=\pi^{d}}$. The $\otimes$-functor
$(-)^{alg}:\Bun_{\mathscr{X}}\rightarrow\Bun_{X}$ maps an adic vector
bundle $\mathscr{E}$ on $\mathscr{X}$ to the quasi-coherent sheaf
on $X$ associated with the graded $P$-module $\oplus_{d\geq0}\Gamma(\mathscr{X},\mathscr{E}(d))$.
In the second column of our diagram, the primes refer to the full
$\otimes$-subcategories of finite free objects in the relevant $\otimes$-categories
of $\varphi$-bundles. Thus plainly, 
\[
\Gamma(\mathscr{Y}^{+},-):\Bun_{\mathscr{Y}^{+}}^{\prime,\varphi}\longleftrightarrow\Bun_{B^{+}}^{\prime,\varphi}:(-\otimes\mathscr{O}_{\mathscr{Y}^{+}})
\]
are mutually inverse equivalences of $\otimes$-categories. The $\otimes$-functors
\[
\Bun_{\overline{B}}^{\varphi}\stackrel{1.17}{\longleftarrow}\Bun_{B^{+}}^{\prime,\varphi}\stackrel{1.19}{\longrightarrow}\Bun_{X}\stackrel{2.2}{\longrightarrow}\Bun_{\mathscr{X}}\quad\mbox{and}\quad\Bun_{\mathscr{Y}^{\circ}}^{\varphi}\stackrel{3.1}{\longrightarrow}\Bun_{B}^{\varphi}
\]
are equivalence of $\otimes$-categories by the indicated references
in \cite[\S 11]{FaFo15}, and so are therefore also all of the above
solid arrow functors. In particular, every $\varphi$-bundle on $\mathscr{Y}^{\circ}$
is finite free and extends uniquely to a finite free $\varphi$-bundle
on $\mathscr{Y}^{+}$.

\subsubsection{~\label{sub:PhiBunY=0000B0}}

This is in sharp contrast with what happens at $y_{C^{\flat}}$: not
every $\varphi$-bundle on $\mathscr{Y}^{\circ}$ extends to $\mathscr{Y}^{-}$,
and those who do have many extensions. This is related to semi-stability
as follows. Let $(-)^{+}:\Bun_{\mathscr{Y}^{\circ}}^{\varphi}\rightarrow\Bun_{\mathscr{Y}^{+}}^{\prime,\varphi}$
be a chosen $\otimes$-inverse of the restriction $\otimes$-functor
$\Bun_{\mathscr{Y}^{+}}^{\prime,\varphi}\rightarrow\Bun_{\mathscr{Y}^{\circ}}^{\varphi}$.
We then have three $\otimes$-functors
\[
\xyC{1pc}\xyR{1pc}\xymatrix{ & \Bun_{\mathscr{R}_{-}}^{\varphi} & \mbox{given by} & \underset{r\mapsto-\infty}{\lim}\Gamma\left(\mathscr{Y}_{]-\infty,r]},-\right)\simeq\Gamma\left(\mathscr{Y}^{\circ},-\right)\otimes_{B}\mathscr{R}_{-}\\
\Bun_{\mathscr{Y}^{\circ}}^{\varphi}\ar[ru]\ar[r]\ar[rd] & \Bun_{X} & \mbox{given by} & (-/\varphi)^{alg}\simeq\mathcal{E}\circ\Gamma(\mathscr{Y}^{\circ},-)\\
\Bun_{\mathscr{Y}^{+}}^{\varphi}\ar[r]\ar@{<->}[u] & \Bun_{L}^{\varphi} & \mbox{given by} & (-)_{y_{L}}^{+}\otimes_{\mathscr{R}_{+}^{\mathrm{int}}}L\simeq\Gamma(\mathscr{Y}^{+},(-)^{+})\otimes_{B^{+}}L
}
\]
where $\mathscr{R}_{-}:=\lim_{r\mapsto-\infty}\Gamma(\mathscr{Y}_{]-\infty,r]},\mathscr{O}_{\mathscr{Y}})$
is the Robba ring; this is a Bezout ring by~\cite[3.5.8]{FaFo15}
or \cite[2.9.6]{Ke05}. The first $\otimes$-functor is an equivalence
of categories by \cite[11.2.22]{FaFo15}, and we have already seen
that so is the second. The third one is not: $\Bun_{L}^{\varphi}$
is abelian semi-simple while $\Bun_{X}$ (along with $\Bun_{\mathscr{Y}^{\circ}}^{\varphi}$
and $\Bun_{\mathscr{R}_{-}}^{\varphi}$) is only quasi-abelian, and
not at all semi-simple. The three target categories are quasi-abelian
$\otimes$-categories, with a Harder-Narasimhan formalism compatible
with $\otimes$-products: this is due respectively to Kedlaya \cite{Ke05},
Fargues and Fontaine \cite{FaFo15} (see \ref{sub:Newton4BundlesOnCurve}),
and to the Dieudonné-Manin classification of isocrystals, which actually
gives rise to a pair of opposed Newton slope filtrations $\mathcal{F}_{N}$
and $\mathcal{F}_{N}^{\iota}$ (see~\ref{sub:RisL}). These formalisms
are compatible, provided that we pick the opposed Newton filtration
$\mathcal{F}_{N}^{\iota}$ on $\Bun_{L}^{\varphi}$. 

The compatibility of the slope filtrations along $\Bun_{X}\simeq\Bun_{\mathscr{Y}^{\circ}}^{\varphi}\simeq\Bun_{\mathscr{R}_{-}}^{\varphi}$
is build up in the proof of \cite[11.2.22]{FaFo15}. Their compatibility
along $\Bun_{X}\simeq\Bun_{\mathscr{Y}^{\circ}}^{\varphi}\rightarrow\Bun_{L}^{\varphi}$
can be seen as follows. Starting with a $\varphi$-bundle $(\mathscr{E},\varphi_{\mathscr{E}})$
on $\mathscr{Y}^{\circ}$, set 
\[
(M,\varphi_{M})=\Gamma\left(\mathscr{Y}^{+},(\mathscr{E},\varphi_{\mathscr{E}})^{+}\right).
\]
This is a finite free étale $\varphi$-module over $B^{+}$ and $(N,\varphi_{N})=(M,\varphi_{M})\otimes_{B^{+}}L$
is the image of $(\mathscr{E},\varphi_{\mathscr{E}})$ in $\Bun_{L}^{\varphi}$.
Fix a section $s:\mathbb{F}\hookrightarrow\mathcal{O}_{C^{\flat}}$
of the quotient map $\mathcal{O}_{C^{\flat}}\twoheadrightarrow\mathbb{F}$.
This gives rise to sections of $A\twoheadrightarrow\mathcal{O}_{L}$
and $A[\frac{1}{\pi}]\hookrightarrow B^{+}\twoheadrightarrow\mathscr{R}_{+}^{\mathrm{int}}\twoheadrightarrow\overline{B}\twoheadrightarrow L$,
which we still denote by $s$. Then $(M,\varphi_{M})$ is non-canonically
isomorphic to $(N,\varphi_{N})\otimes_{L,s}B^{+}$ by \cite[\S 11.1]{FaFo15},
thus $(\mathscr{E}/\varphi_{\mathscr{E}})^{alg}\simeq\mathcal{E}((M,\varphi_{M})\otimes_{B^{+}}B)$
is non-canonically isomorphic to $\mathcal{E}^{s}(N,\varphi_{N}):=\mathcal{E}((N,\varphi_{N})\otimes_{L,s}B)$.
Our claim now follows from~\cite[\S 8.2.4]{FaFo15}, where this $\otimes$-functor
$\mathcal{E}^{s}:\Bun_{L}^{\varphi}\rightarrow\Bun_{X}$ is denoted
by $\mathcal{E}$. 

Now by Kedlaya's theory, we have equivalences of $\otimes$-categories
\[
\xyC{4pc}\xymatrix{\Vect_{E} & \Bun_{\mathscr{R}_{-}^{\mathrm{int}}}^{\varphi}\otimes E\ar@{<->}[l]\ar[r]^{-\otimes\mathscr{R}_{-}} & \Bun_{\mathscr{R}_{-}}^{\varphi,0}}
\]
where $\Bun_{\mathscr{R}_{-}}^{\varphi,0}$ is the full $\otimes$-subcategory
of slope $0$ semi-stable objects in $\Bun_{\mathscr{R}_{-}}^{\varphi}$.
The composite $\otimes$-functor $\Vect_{E}\rightarrow\Bun_{\mathscr{R}_{-}}^{\varphi,0}$
maps $V$ to $(V\otimes_{E}\mathscr{R}_{-},\mathrm{Id}\otimes\varphi)$
with inverse $(N,\varphi_{N})\mapsto N^{\varphi_{N}=1}$. It follows
that we have equivalences of $\otimes$-categories 
\[
\xyC{4pc}\xymatrix{\Vect_{E} & \Bun_{\mathscr{Y}^{-}}^{\varphi}\otimes E\ar@{<->}[l]\ar[r]^{-\vert\mathscr{Y}^{\circ}} & \Bun_{\mathscr{Y}^{\circ}}^{\varphi,0}}
\]
where $\Bun_{\mathscr{Y}^{\circ}}^{\varphi,0}$ is the full $\otimes$-subcategory
of slope $0$ semi-stable objects in $\Bun_{\mathscr{Y}^{\circ}}^{\varphi}$.
The composite functor $\Vect_{E}\rightarrow\Bun_{\mathscr{Y}^{\circ}}^{\varphi,0}$
maps $V$ to $(V\otimes_{E}\mathscr{O}_{\mathscr{Y}^{\circ}},\mathrm{Id}\otimes\varphi)$
with inverse $(\mathscr{E},\varphi_{\mathscr{E}})\mapsto\Gamma(\mathscr{Y}^{\circ},\mathscr{E})^{\varphi_{\mathscr{E}}=1}$.
In other words, a $\varphi$-bundle $(\mathscr{E},\varphi_{\mathscr{E}})$
over $\mathscr{Y}^{\circ}$ extends to a $\varphi$-bundle over $\mathscr{Y}^{-}$
if and only if it is semi-stable of slope $0$ and then, there is
a functorial bijective correspondance between the set of all such
extensions and the set of all $\mathcal{O}_{E}$-lattices $T$ in
$V=\Gamma(\mathscr{Y}^{\circ},\mathscr{E})^{\varphi_{\mathscr{E}}=1}$,
given by $T\mapsto\left(T\otimes\mathscr{O}_{\mathscr{Y}^{-}},\mathrm{Id}\otimes\varphi\right)$.

\subsubsection{~\label{sub:ComputeAnalFunct}}

We shall now compute the equivalence of $\otimes$-categories 
\[
\xyR{.2pc}\xymatrix{\Mod_{A,f}^{\varphi,\leq}\ar[r] & \Modif_{\mathscr{Y}}^{\star}\ar[r] & \Modif_{\mathscr{Y}^{-},\mathscr{Y}^{+}}^{\varphi,\geq}\\
(M,\varphi_{M})\ar@{|->}[r] & \left(\mathscr{E},\psi_{\mathscr{E}}\right)\ar@{|->}[r] & \left(\mathscr{E}^{-},\mathscr{E}^{+},f_{\mathscr{E}}\right)
}
\]
Starting with the anti-effective finite free BKF-module $(M,\varphi_{M})$
over $A$, set 
\[
T=(M\otimes_{A}A(C))^{\varphi_{M}\otimes\varphi=1}\quad\mbox{and}\quad(\overline{M},\varphi_{\overline{M}})=(M,\varphi_{M})\otimes_{A}\overline{B}.
\]
Thus $T$ is a finite free $\mathcal{O}_{E}$-module and $(\overline{M},\varphi_{\overline{M}})$
is a finite free étale $\varphi$-module over $\overline{B}$ (since
$\xi'$ is invertible in $\overline{B}$). By~\cite[4.26]{BaMoSc16}
and its proof, the canonical isomorphism
\[
\left(T\otimes_{\mathcal{O}_{E}}A(C),\mathrm{Id}\otimes\varphi\right)\simeq\left(M\otimes_{A}A(C),\varphi_{M}\otimes\varphi\right)
\]
descends to an isomorphism over the subring $A[\frac{1}{\mu}]\subset A(C)$,
\[
\eta_{M}^{-}[{\textstyle \frac{1}{\mu}}]:T\otimes_{\mathcal{O}_{E}}A[{\textstyle \frac{1}{\mu}}]\stackrel{\simeq}{\longrightarrow}M\otimes_{A}A[{\textstyle \frac{1}{\mu}}]
\]
which is induced by a $\varphi^{-1}$-invariant $\mathcal{O}_{E}$-linear
morphism 
\[
\eta_{M}^{-}:T\hookrightarrow M.
\]
The latter gives a morphism of modifications of vector bundles on
$\mathscr{Y}$, 
\[
\eta_{M}^{-}\otimes\mathscr{O}_{\mathscr{Y}}:\left(T\otimes\mathscr{O}_{\mathscr{Y}},\mathrm{Id}\otimes\varphi^{-1}\right)\hookrightarrow\left(\mathscr{E},\psi_{\mathscr{E}}\right)
\]
whose restriction to $\mathscr{Y}^{-}$ factors through a morphism
of $\varphi$-bundles over $\mathscr{Y}^{-}$, 
\[
f_{M}^{-}:\left(T\otimes\mathscr{O}_{\mathscr{Y}^{-}},\mathrm{Id}\otimes\varphi\right)\hookrightarrow\left(\mathscr{E}^{-},\varphi_{\mathscr{E}^{-}}\right).
\]
Since $\mu$ is invertible on $\mathscr{Y}_{[-\infty,0[}$, both $\eta_{M}^{-}\otimes\mathscr{O}_{\mathscr{Y}}$
and $f_{M}^{-}$ are isomorphisms over $\mathscr{Y}_{[-\infty,0[}$.
In particular, the localization of $f_{M}^{-}$ at $y_{C^{\flat}}$
is an isomorphism, and so is therefore $f_{M}^{-}$ itself by \ref{sub:PhiBunOverY-}. 

On the other hand, pick any finite free étale $\varphi$-module $(D^{+},\varphi_{D^{+}})$
over $B^{+}$ reducing to $(\overline{M},\varphi_{\overline{M}})$
over $\overline{B}$. By \cite[4.26]{Fa15} applied to the effective
dual BKF-module $(M,\varphi_{M})^{\vee}=(M^{\vee},\varphi_{M}^{\vee-1})$,
there is a unique $\varphi^{-1}$-equivariant morphism 
\[
\eta_{M}^{+}:\left(M\otimes_{A}B^{+},\varphi_{M}^{-1}\otimes\varphi^{-1}\right)\rightarrow\left(D^{+},\varphi_{D^{+}}^{-1}\right)
\]
reducing to the given isomorphism $(\overline{M},\varphi_{\overline{M}}^{-1})\simeq\left(D^{+},\varphi_{D^{+}}^{-1}\right)\otimes_{B^{+}}\overline{B}$.
As above, the latter yields a morphism of modifications of vector
bundles over $\mathscr{Y}^{+}$, 
\[
\eta_{M}^{+}\otimes\mathscr{O}_{\mathscr{Y}}:\left(\mathscr{E},\psi_{\mathscr{E}}\right)\vert_{\mathscr{Y}^{+}}\rightarrow\left(D^{+},\varphi_{D^{+}}^{-1}\right)\otimes_{B^{+}}\mathscr{O}_{\mathscr{Y}^{+}}
\]
which induces a morphism of $\varphi$-bundles over $\mathscr{Y}^{+}$,
\[
f_{M}^{+}:\left(\mathscr{E}^{+},\varphi_{\mathscr{E}^{+}}\right)\rightarrow\left(D^{+},\varphi_{D^{+}}\right)\otimes_{B^{+}}\mathscr{O}_{\mathscr{Y}^{+}}.
\]
By \cite[4.31]{Fa15}, $\eta_{M}^{+}\otimes\mathscr{O}_{\mathscr{Y}}$
restricts to an isomorphism over $\mathscr{Y}_{[r,+\infty]}$ for
$r\gg0$. Since also $\mathscr{E}\vert_{\mathscr{Y}^{+}}\hookrightarrow\mathscr{E}^{+}$
is an isomorphism over $\mathscr{Y}_{]-1,\infty]}$, it follows that
$f_{M}^{+}$ is an isomorphism over $\mathscr{Y}_{[r,+\infty]}$.
But then $(\varphi^{i})^{\ast}(f_{M}^{+})\simeq f_{M}^{+}$ is an
isomorphism over $\mathscr{Y}_{[r-i,+\infty]}$ for all $i\geq0$,
thus $f_{M}^{+}$ is an isomorphism over the whole of $\mathscr{Y}^{+}$. 

Finally, let $f_{M}:T\otimes_{\mathcal{O}_{E}}\mathscr{O}_{\mathscr{Y}^{\circ}}\rightarrow D^{+}\otimes_{B^{+}}\mathscr{O}_{\mathscr{Y}^{\circ}}$
be the $\varphi$-equivariant morphism
\[
\xymatrix{T\otimes_{\mathcal{O}_{E}}\mathscr{O}_{\mathscr{Y}^{\circ}}\ar[r]^{\eta_{M}^{-}\otimes\mathscr{O}_{\mathscr{Y}^{\circ}}} & M\otimes_{A}\mathscr{O}_{\mathscr{Y}^{\circ}}\ar[r]^{\eta_{M}^{+}\otimes\mathscr{O}_{\mathscr{Y}^{\circ}}} & D^{+}\otimes_{B^{+}}\mathscr{O}_{\mathscr{Y}^{\circ}}}
\]
We thus have shown that $f_{M}^{-}$ and $f_{M}^{+}$ induce an isomorphism
\[
\left(\mathscr{E}^{-},\mathscr{E}^{+},f_{\mathscr{E}}\right)\simeq\left(T\otimes_{\mathcal{O}_{E}}\mathscr{O}_{\mathscr{Y}^{-}},D^{+}\otimes_{B^{+}}\mathscr{O}_{\mathscr{Y}^{+}},f_{M}\right).
\]
In particular, $\mathscr{E}^{+}\simeq D^{+}\otimes_{B^{+}}\mathscr{O}_{\mathscr{Y}^{+}}$
is finite free (we did not knew this so far) and 
\[
\eta_{M}^{+}\otimes\mathscr{O}_{\mathscr{Y}^{+}}:M\otimes_{A}\mathscr{O}_{\mathscr{Y}^{+}}\rightarrow D^{+}\otimes_{B^{+}}\mathscr{O}_{\mathscr{Y}^{+}}
\]
is an isomorphism over $\mathscr{Y}_{]-1,+\infty]}$. The freeness
of $\mathscr{E}^{+}$ also yields a canonical choice for the finite
free lift $D^{+}$ of $\overline{M}=M\otimes_{A}\overline{B}$: we
may take $D^{+}=\Gamma(\mathscr{Y}^{+},\mathscr{E}^{+})$ with the
isomorphism $\overline{M}\simeq D^{+}\otimes_{B^{+}}\overline{B}$
induced by $\mathscr{E}\vert_{\mathscr{Y}^{+}}\hookrightarrow\mathscr{E}^{+}$.

\subsubsection{~\label{sub:ConstruFarguesFunctEff}}

The discussion above shows that the $\otimes$-functor
\[
\xyC{2pc}\xymatrix{\Mod_{A[\frac{1}{\pi}],f}^{\varphi,\leq}\ar[r] & \Modif_{\mathscr{Y}^{+}}^{\star}\ar[r] & \Modif_{\mathscr{Y}^{\circ},\mathscr{Y}^{+}}^{\varphi,\geq}}
\]
induces an equivalence of $\otimes$-categories
\[
\xymatrix{\Mod_{A,f}^{\varphi,\leq}\otimes E\ar[r] & \Modif_{\mathscr{Y}^{\circ},\mathscr{Y}^{+}}^{\varphi,ad,\geq}}
\]
where $\Modif_{\mathscr{Y}^{\circ},\mathscr{Y}^{+}}^{\varphi,ad,\geq}$
is the full $\otimes$-subcategory of objects $(\mathscr{E}^{-},\mathscr{E}^{+},f_{\mathscr{E}})$
in $\Modif_{\mathscr{Y}^{\circ},\mathscr{Y}^{+}}^{\varphi,\geq}$
such that $\mathscr{E}^{-}\in\Bun_{\mathscr{Y}^{\circ}}^{\varphi}$
belongs to $\Bun_{\mathscr{Y}^{\circ}}^{\varphi,0}$. Moreover for
any such object, $\mathscr{E}^{+}$ is actually finite free, and $-\vert_{\mathscr{Y}^{\circ}}:\Bun_{\mathscr{Y}^{+}}^{\prime,\varphi}\stackrel{\simeq}{\rightarrow}\Bun_{\mathscr{Y}^{\circ}}^{\varphi}$
thus induces an equivalence 
\[
\xymatrix{\Modif_{\mathscr{Y}^{\circ},\mathscr{Y}^{+}}^{\varphi,ad,\geq}\ar[r] & \Modif_{\mathscr{Y}^{\circ},\mathscr{Y}^{\circ}}^{\varphi,ad,\geq}\\
(\mathscr{E}^{-},\mathscr{E}^{+},f_{\mathscr{E}})\ar@{|->}[r] & (\mathscr{E}^{-},\mathscr{E}^{+}\vert_{\mathscr{Y}^{\circ}},f_{\mathscr{E}})
}
\]
of $\otimes$-categories. Finally, the equivalence of $\otimes$-categories
\[
\left(\xymatrix{\Bun_{\mathscr{Y}^{\circ}}^{\varphi}\ar[r]^{-/\varphi} & \Bun_{\mathscr{X}}\ar[r]^{(-)^{\mathrm{alg}}} & \Bun_{X}}
\right)=\left(\xymatrix{\Bun_{\mathscr{Y}^{\circ}}^{\varphi}\ar[r]^{\Gamma(\mathscr{Y}^{\circ},-)} & \Bun_{B}^{\varphi}\ar[r]^{\mathcal{E}} & \Bun_{X}}
\right)
\]
induces equivalences of $\otimes$-categories 
\[
\xymatrix{\Modif_{\mathscr{Y}^{\circ},\mathscr{Y}^{\circ}}^{\varphi,\geq}\ar[r] & \Modif_{X}^{\geq}}
\quad\mbox{and}\quad\xymatrix{\Modif_{\mathscr{Y}^{\circ},\mathscr{Y}^{\circ}}^{\varphi,ad,\geq}\ar[r] & \Modif_{X}^{ad,\geq}}
.
\]
Putting everything together, we obtain an equivalence of $\otimes$-categories
\[
\xymatrix{\underline{\mathcal{E}}:\Mod_{A,f}^{\varphi,\leq}\otimes E\ar[r] & \Modif_{X}^{ad,\geq}.}
\]
This is of course the restriction of the $\otimes$-functor
\[
\xymatrix{\underline{\mathcal{E}}:\Mod_{A[\frac{1}{\pi}]}^{\varphi,\leq}\ar[r] & \Modif_{\mathscr{Y}^{\circ},\mathscr{Y}^{+}}^{\varphi,\geq}\ar[r] & \Modif_{\mathscr{Y}^{\circ},\mathscr{Y}^{\circ}}^{\varphi,\geq}\ar[r] & \Modif_{X}^{\geq}}
\]
but the first two components of the latter may not be equivalences.

\subsubsection{Compatibility with Hodge filtrations\label{sub:CompHodge}}

The morphisms of locally ringed space 
\[
\mathscr{Y}^{\circ}\rightarrow\mathscr{X}\rightarrow X\quad\mbox{and}\quad\mathscr{Y}^{\circ}\rightarrow\Spa(A)\rightarrow\Spec(A)
\]
map $y_{i}\in\left|\mathscr{Y}^{\circ}\right|$ to respectively $\infty\in\left|X\right|$
and $\mathfrak{m}_{i}=A\varphi^{-i}(\xi)\in\left|\Spec(A)\right|$.
Moreover, they induce isomorphism between the corresponding completed
local rings $\mathscr{O}_{\mathscr{Y},y_{i}}^{\wedge}$, $\mathcal{O}_{X,\infty}^{\wedge}$
and $A_{\mathfrak{m}_{i}}^{\wedge}=B_{dR,\mathfrak{m}_{i}}^{+}$.
For $i=0$, the latter is just $B_{dR}^{+}$. For $(N,\varphi_{N})$
in $\Mod_{A[\frac{1}{\pi}]}^{\varphi,\leq}$ mapping to $(\mathscr{E},\psi_{\mathscr{E}})$
in $\Modif_{\mathscr{Y}^{+}}^{\ast}$ and $\underline{\mathcal{E}}=(\mathcal{E}_{1}\hookrightarrow\mathcal{E}_{2})$
in $\Modif_{X}^{\geq}$, we thus find 
\begin{eqnarray*}
\left(\mathcal{E}_{1,\infty}^{\wedge}\hookrightarrow\mathcal{E}_{2,\infty}^{\wedge}\right) & \simeq & \left(\mathscr{\mathscr{E}}(-\infty)_{y_{0}}^{\wedge}\hookrightarrow\mathscr{\mathscr{E}}(+\infty)_{y_{0}}^{\wedge}\right)\\
 & = & \left(\mathscr{\mathscr{E}}(-1)_{y_{0}}^{\wedge}\hookrightarrow\mathscr{\mathscr{E}}(0)_{y_{0}}^{\wedge}\right)\\
 & \simeq & \left((\varphi^{-1})^{\ast}(\varphi_{N}^{-1}):(\varphi^{-1})^{\ast}N\otimes B_{dR}^{+}\hookrightarrow N\otimes B_{dR}^{+}\right).
\end{eqnarray*}
It follows that 
\begin{eqnarray*}
\mathcal{F}_{H,2}(\underline{\mathcal{E}}) & = & \mathcal{F}_{H}^{\iota}(N,\varphi_{N})\quad\mbox{on}\quad\mathcal{E}_{2}(\infty)=N\otimes_{A}C\\
\mathcal{F}_{H,1}(\underline{\mathcal{E}})\otimes_{C,\varphi}C' & = & \mathcal{F}_{H}(N,\varphi_{N})\quad\mbox{on}\quad\mathcal{E}_{1}(\infty)\otimes_{C}C'=N\otimes_{A}C'
\end{eqnarray*}

\subsubsection{Compatibility with Tate objects\label{sub:CompTate}}

The Tate object over $A$ is anti-effective, 
\[
A\{1\}=\left({\textstyle \frac{1}{\mu}}A\otimes\mathcal{O}_{E}(1),\varphi\otimes\mathrm{Id}\right).
\]
The corresponding sequence $\cdots\rightarrow\mathscr{E}(i)\rightarrow\mathscr{E}(i+1)\rightarrow\cdots$
is obtained from
\[
\xyC{2pc}\xymatrix{\cdots\ar[r] & \frac{1}{\varphi^{-2}(\mu)}A\ar@{^{(}->}[r] & \frac{1}{\varphi^{-1}(\mu)}A\ar@{^{(}->}[r] & \frac{1}{\mu}A\ar@{^{(}->}[r] & \frac{1}{\varphi(\mu)}A\ar@{^{(}->}[r] & \frac{1}{\varphi^{2}(\mu)}A\ar@{^{(}->}[r] & \cdots}
\]
by tensoring with $-\otimes_{A}(\mathscr{O}_{\mathscr{Y}}\otimes\mathcal{O}_{E}(1))$.
Thus by~\cite[3.23]{BaMoSc16},
\[
\left(\mathscr{E}^{-}\vert_{\mathscr{Y}^{\circ}}\hookrightarrow\mathscr{E}^{+}\vert_{\mathscr{Y}^{\circ}}\right)=\left(\mathscr{O}_{\mathscr{Y}^{\circ}}\otimes E(1)\hookrightarrow\mathscr{O}_{\mathscr{Y}^{\circ}}(1)\otimes E(1)\right)
\]
where $\mathscr{O}_{\mathscr{Y}^{\circ}}\hookrightarrow\mathscr{O}_{\mathscr{Y}^{\circ}}(1)$
maps to $\mathcal{O}_{X}\hookrightarrow\mathcal{O}_{X}(1)$, therefore
\[
\underline{\mathcal{E}}\left(A[{\textstyle \frac{1}{\pi}}]\{1\}\right)=\left(\mathcal{O}_{X}\otimes E(1)\hookrightarrow\mathcal{O}_{X}(1)\otimes E(1)\right)=\underline{\mathcal{O}}_{X}\{1\}.
\]
The $\otimes$-functor constructed in $\ref{sub:ConstruFarguesFunctEff}$
thus extends to a $\otimes$-functor 
\[
\underline{\mathcal{E}}:\Mod_{A[\frac{1}{\pi}],f}^{\varphi}\rightarrow\Modif_{X}
\]
mapping $N$ to $\underline{\mathcal{E}}(N\{i\})\{-i\}$ for $i\gg0$.
The latter is still compatible with Hodge filtrations by~\ref{sub:Hodge-and-NewtonForModif}
and proposition~\ref{prop:t_H(twist)4A=00005B1/pi=00005D}, and it
induces an equivalence of $\otimes$-categories 
\[
\underline{\mathcal{E}}:\Mod_{A,f}^{\varphi}\otimes E\rightarrow\Modif_{X}^{ad}.
\]

\subsubsection{Compatibility with Newton types\label{sub:CompNewton}}

For $(N,\varphi_{N})$ in $\Mod_{A,f}^{\varphi}\otimes E$ of rank
$r\in\mathbb{N}$ mapping to $\underline{\mathcal{E}}=(\mathcal{E}_{1},\mathcal{E}_{2},f)$
in $\Modif_{X}^{ad}$ and to $(D,\varphi_{D})=(M,\varphi_{M})\otimes_{A}L$
in $\Bun_{L}^{\varphi}$,
\[
t_{N}(\underline{\mathcal{E}})=t_{N}(\mathcal{E}_{2})\quad\mbox{equals}\quad t_{N}^{\iota}(D,\varphi_{D})\quad\mbox{in}\quad\mathbb{Q}_{\geq}^{r}.
\]
Indeed, we may assume that $(N,\varphi_{N})=(M,\varphi_{M})\otimes E$
for an anti-effective finite free BKF-module $(M,\varphi_{M})$ over
$A$ by \ref{sub:RisL} and \ref{sub:Hodge-and-NewtonForModif}. If
$(M,\varphi_{M})$ maps to $(\mathscr{E},\psi_{\mathscr{E}})$ and
$(\mathscr{E}^{+},\mathscr{E}^{-},f_{\mathscr{E}})$ as above, then
$\mathcal{E}_{2}$ is the image of $(\mathscr{E}^{+},\varphi_{\mathscr{E}^{+}})$
under 
\[
\mathcal{E}\circ\Gamma(\mathscr{Y}^{\circ},-):\Bun_{\mathscr{Y}^{+}}^{\varphi}\rightarrow\Bun_{X}
\]
thus $t_{N}(\mathcal{E}_{2})=t_{N}^{\iota}(D',\varphi_{D'})$ by \ref{sub:PhiBunY=0000B0}
where $(D^{\prime},\varphi_{D^{\prime}})$ is the image of $(\mathscr{E}^{+},\varphi_{\mathscr{E}^{+}})$
under 
\[
(-)_{y_{L}}\otimes_{\mathscr{R}_{+}^{\mathrm{int}}}L:\Bun_{\mathscr{Y}^{+}}^{\varphi}\rightarrow\Bun_{L}^{\varphi}.
\]
Since $(\mathscr{E},\psi_{\mathscr{E}})\vert_{\mathscr{Y}^{+}}\hookrightarrow(\mathscr{E}^{+},\varphi_{\mathscr{E}^{+}}^{-1})$
is an isomorphism over $\mathscr{Y}_{]-1,\infty]}$, it induces 
\[
(M,\varphi_{M}^{-1})\otimes_{A}\mathscr{R}_{+}^{\mathrm{int}}=(\mathscr{E},\psi_{\mathscr{E}})_{y_{L}}\stackrel{\simeq}{\longrightarrow}(\mathscr{E}^{+},\varphi_{\mathscr{E}^{+}}^{-1})_{y_{L}}
\]
therefore also $(D,\varphi_{D})=(M,\varphi_{M})\otimes_{A}L\simeq(D',\varphi_{D'})$,
which proves our claim.

\subsubsection{Compatibility with Bhatt-Morrow-Scholze\label{sub:CompScholzeFargues}}

We now claim that the $\otimes$-functor 
\[
\HTT\circ\underline{\mathcal{E}}:\Mod_{A,f}^{\varphi}\otimes E\rightarrow\Modif_{X}^{ad}\rightarrow\HT_{E}^{B_{dR}}
\]
is canonically isomorphic to the Bhatt-Morrow-Scholze $\otimes$-functor
\[
\HTT':\Mod_{A,f}^{\varphi}\otimes E\rightarrow\HT_{E}^{B_{dR}}
\]
of section~\ref{sub:The-Bhatt-Morrow-Scholze-functor}. Since both
functors are compatible with Tate twists, it is sufficient to establish
that they have canonically isomorphic restrictions to the full $\otimes$-subcategory
of anti-effective objects in $\Mod_{A,f}^{\varphi}\otimes E$, and
this immediately follows from the computations in section~\ref{sub:ComputeAnalFunct}
and \ref{sub:CompHodge}.

\subsubsection{Proof of Theorem~\ref{thm:FarguesScholze}}

It remains to establish that the $\otimes$-functor 
\[
\HTT':\Mod_{A,f}^{\varphi}\rightarrow\HT_{\mathcal{O}_{E}}^{B_{dR}}
\]
is an equivalence of $\otimes$-categories. Consider the ($2-)$commutative
diagram 
\[
\xyC{3pc}\xyR{2pc}\xymatrix{\Mod_{A,f}^{\varphi}\ar[r]^{\HTT'}\ar[d]_{-\otimes E} & \HT_{\mathcal{O}_{E}}^{B_{dR}}\ar[r]^{T}\ar[d]_{-\otimes E} & \Bun_{\mathcal{O}_{E}}\ar[d]^{-\otimes E}\\
\Mod_{A,f}^{\varphi}\otimes E\ar[r]^{\HTT'} & \HT_{E}^{B_{dR}}\ar[r]^{V} & \Vect_{E}
}
\]
Since the second square is cartesian, it is sufficient to establish
that the outer rectangle is cartesian, for then so will be the first
square, and its top row will thus be an equivalence of categories
since so is the second row. We may again restrict our attention to
anti-effective objects. The outer rectangle then factors as 
\[
\xyC{3pc}\xyR{2pc}\xymatrix{\Mod_{A,f}^{\varphi,\leq}\ar[r]^{\underline{\mathscr{E}}}\ar[d]_{-\otimes E} & \Modif_{\mathscr{Y}^{-},\mathscr{Y}^{+}}^{\varphi,\geq}\ar[r]^{\mathscr{E}^{-}}\ar[d]_{-\vert_{\mathscr{Y}^{\circ}}} & \Bun_{\mathscr{Y}^{-}}^{\varphi}\ar[d]_{-\vert_{\mathscr{Y}^{\circ}}}\ar[r] & \Bun_{\mathcal{O}_{E}}\ar[d]^{-\otimes E}\\
\Mod_{A,f}^{\varphi,\leq}\otimes E\ar[r]^{\underline{\mathscr{E}}} & \Modif_{\mathscr{Y}^{\circ},\mathscr{Y}^{+}}^{\varphi,ad,\geq}\ar[r]^{\mathscr{E}^{-}} & \Bun_{\mathscr{Y}^{\circ}}^{\varphi,0}\ar[r] & \Vect_{E}
}
\]
In this commutative diagram, the first square is cartesian since the
two $\underline{\mathscr{E}}$'s are equivalences of $\otimes$-categories,
the second square is obviously cartesian, and the third square is
cartesian by Kedlaya's theory as explained in~\ref{sub:PhiBunY=0000B0}.
So the outer rectangle is indeed cartesian. This finishes the proof
of theorem~\ref{thm:FarguesScholze}.

\subsubsection{Final questions}

Is it true that every $\varphi$-bundle over $\mathscr{Y}^{+}$ is
finite and free? Is there an integral version of the Fargues-Fontaine
curve $X$ corresponding to $\mathscr{Y}^{-}/\varphi^{\mathbb{Z}}$?
And is it true that $\underline{\mathcal{E}}:\Mod_{A[\frac{1}{\pi}],f}^{\varphi}\rightarrow\Modif_{X}$
is an equivalence of $\otimes$-categories?

\subsection{Exactness\label{sub:Exactness}}

\subsubsection{~}

We now want to investigate the difference between naive and canonical
short exact sequences in $\Mod_{A,f}^{\varphi}$. We start with a
canonical short exact sequence, 
\[
0\rightarrow(M_{1},\varphi_{1})\rightarrow(M_{2},\varphi_{2})\rightarrow(M_{3},\varphi_{3})\rightarrow0.
\]
The corresponding complex of Hodge-Tate module is a short exact sequence
\[
0\rightarrow(T_{1},\Xi_{1})\rightarrow(T_{2},\Xi_{2})\rightarrow(T_{3},\Xi_{3})\rightarrow0
\]
and we now know what it means: the underlying complexes of $\mathcal{O}_{E}$
and $B_{dR}^{+}$-modules are both exact. Since $M_{i}[\frac{1}{\mu}]\simeq T_{i}\otimes A[\frac{1}{\mu}]$,
it follows that 
\[
0\rightarrow M_{1}[{\textstyle \frac{1}{\mu}}]\rightarrow M_{2}[{\textstyle \frac{1}{\mu}}]\rightarrow M_{3}[{\textstyle \frac{1}{\mu}}]\rightarrow0
\]
is exact. In particular, $M_{1}\rightarrow M_{2}$ is injective.

\subsubsection{~}

Let $(Q,\varphi_{Q})$ be the cokernel of $(M_{1},\varphi_{1})\hookrightarrow(M_{2},\varphi_{2})$
in the abelian category $\Mod_{A}^{\varphi}$. Then $Q$ is the cokernel
of $M_{1}\hookrightarrow M_{2}$ in $\Mod_{A}$, therefore $Q$ is
a perfect $A$-module of projective dimension $\leq1$ with $Q[\frac{1}{\mu}]\simeq M_{3}[\frac{1}{\mu}]$
finite free over $A[\frac{1}{\mu}]$. 
\begin{lem}
If $Q[\pi^{\infty}]$ is finitely presented over $A$, then $Q[\pi^{\infty}]=0$.\end{lem}
\begin{proof}
Suppose that $Q[\pi^{\infty}]$ is finitely presented over $A$. Its
inverse image $M_{1}^{\prime}$ in $M_{2}$ is then a finitely presented
$A$-module with $M_{1}^{\prime}[\frac{1}{\pi}]\simeq M_{1}[\frac{1}{\pi}]$
free over $A$, so $M_{1}^{\prime}$ is a torsion-free BKF-module.
Then $M_{1}\subset M_{1}^{\prime}\subset M_{1,f}^{\prime}\subset M_{2}$
with $\overline{M}_{1}^{\prime}=M_{1,f}^{\prime}/M_{1}$ killed by
$\pi^{n}$ for $n\gg0$, so $M_{1,f}^{\prime}$ is contained in the
kernel of $(M_{2},\varphi_{2})\rightarrow(M_{3},\varphi_{3})$ in
$\Mod_{A,f}^{\varphi}$, i.e.~$M_{1,f}^{\prime}\hookrightarrow M_{2}$
factors through $M_{1}\hookrightarrow M_{2}$, which means that actually
$M_{1}=M_{1}^{\prime}=M_{1,f}^{\prime}$, and indeed $Q[\pi^{\infty}]=M_{1}^{\prime}/M_{1}=0$.
\end{proof}

\subsubsection{~}

Recall that $B_{crys}^{+}=A_{crys}[\frac{1}{\pi}]$ where $A_{crys}$
is the $\pi$-adic completion of the $A$-subalgebra of $A[\frac{1}{\pi}]$
generated by $\frac{\xi^{m}}{m!}$ for all $m\geq0$. 
\begin{prop}
\label{prop:ConditionsExactAlgebraic}The following conditions are
equivalent:
\begin{enumerate}
\item Our complex induces an exact sequence of $B_{crys}^{+}$-modules 
\[
0\rightarrow M_{1}\otimes_{A}B_{crys}^{+}\rightarrow M_{2}\otimes_{A}B_{crys}^{+}\rightarrow M_{3}\otimes_{A}B_{crys}^{+}\rightarrow0.
\]

\item Our complex induces an exact sequence of $B_{crys}^{+}$-modules 
\[
M_{1}\otimes_{A}B_{crys}^{+}\rightarrow M_{2}\otimes_{A}B_{crys}^{+}\rightarrow M_{3}\otimes_{A}B_{crys}^{+}\rightarrow0.
\]

\item $Q[\frac{1}{\pi}]$ if free over $A[\frac{1}{\pi}]$. 
\item $Q[\frac{1}{\pi}]$ is projective over $A[\frac{1}{\pi}].$
\item Our complex induces an exact sequence of $A$-modules
\[
0\rightarrow M_{1}\rightarrow M_{2}\rightarrow M_{3}\rightarrow\overline{Q}\rightarrow0
\]
with $\overline{Q}$ supported at $\{\mathfrak{m}\}$, i.e.~$\overline{Q}\in\Mod_{A,\mathfrak{m}^{\infty}}$.
\item Our complex induces an exact sequence
\[
0\rightarrow\tilde{M}_{1}\rightarrow\tilde{M}_{2}\rightarrow\tilde{M}_{3}\rightarrow0
\]
of quasi-coherent sheaves on $U=\Spec(A)\setminus\{\mathfrak{m}\}$.
\item Our complex induces an exact sequence of $A[\frac{1}{\pi}]$-modules
\[
0\rightarrow M_{1}[{\textstyle \frac{1}{\pi}}]\rightarrow M_{2}[{\textstyle \frac{1}{\pi}}]\rightarrow M_{3}[{\textstyle \frac{1}{\pi}}]\rightarrow0
\]

\item Our complex is isogeneous to a complex which is naively exact.
\end{enumerate}
\end{prop}
\begin{proof}
$(1)\Rightarrow(2)$ is obvious. $(2)\Rightarrow(3)$ follows from
\cite[4.1.9]{BaMoSc16}. $(3)\Leftrightarrow(4)$ is \cite[4.12]{BaMoSc16}. 

$(3)\Rightarrow(5)$ The assumption says that $Q$ is a BKF-module.
Then $Q[\pi^{\infty}]$ is finitely presented, hence trivial by the
previous lemma. It is then obvious that 
\[
M_{2}\twoheadrightarrow Q\hookrightarrow Q_{f}
\]
is a cokernel of $M_{1}\hookrightarrow M_{2}$ in $\Mod_{A,f}^{\varphi}$,
which proves $(5)$ with $M_{3}=Q_{f}$. 

$(5)\Rightarrow(6)\Rightarrow(7)\Rightarrow(1)$ and $(8)\Rightarrow(7)$
are obvious.

$(5)\Rightarrow(8)$: if $\pi^{n}\overline{Q}=0$, the pull-back through
multiplication by $\pi^{n}$ on $M_{3}$ yields an exact sequence
\[
\begin{array}{ccccccccccc}
0 & \rightarrow & M_{1} & \rightarrow & M_{2}^{\prime} & \rightarrow & M_{3} & \rightarrow & 0\\
 &  & \parallel &  & \downarrow &  & \downarrow\pi^{n}\\
0 & \rightarrow & M_{1} & \rightarrow & M_{2} & \rightarrow & M_{3} & \rightarrow & \overline{Q} & \rightarrow & 0
\end{array}
\]
of the desired form, i.e.~isogeneous to the initial sequence and
naively exact.
\end{proof}

\subsubsection{~}

Suppose that our BKF-modules are anti-effective and let 
\[
0\rightarrow(\mathscr{E}_{1},\psi_{1})\rightarrow(\mathscr{E}_{2},\psi_{2})\rightarrow(\mathscr{E}_{3},\psi_{3})\rightarrow0
\]
\[
0\rightarrow(\mathscr{E}_{1}^{-}\hookrightarrow\mathscr{E}_{1}^{+})\rightarrow(\mathscr{E}_{2}^{-}\hookrightarrow\mathscr{E}_{2}^{+})\rightarrow(\mathscr{E}_{3}^{-}\hookrightarrow\mathscr{E}_{3}^{+})\rightarrow0
\]
be the corresponding complexes in $\Modif_{\mathscr{Y}}^{\star}$
and $\Modif_{\mathscr{Y}^{-},\mathscr{Y}^{+}}^{\geq}$. Note that
\[
0\rightarrow\mathscr{E}_{1}^{-}\rightarrow\mathscr{E}_{2}^{-}\rightarrow\mathscr{E}_{3}^{-}\rightarrow0
\]
is a (split) short exact sequence of sheaves on $\mathscr{Y}^{-}$
since $\mathscr{E}_{i}^{-}=T_{i}\otimes_{\mathcal{O}_{E}}\mathscr{O}_{\mathscr{Y}^{-}}$. 
\begin{prop}
\label{prop:ConditionsExactnessAdic}The conditions of proposition~\ref{prop:ConditionsExactAlgebraic}
are equivalent to:
\begin{enumerate}
\item Anyone of the following complexes is exact:

\begin{enumerate}
\item \textup{$0\rightarrow\mathscr{E}_{1}\rightarrow\mathscr{E}_{2}\rightarrow\mathscr{E}_{3}\rightarrow0$
in $\Bun_{\mathscr{Y}}$.}
\item $0\rightarrow\mathscr{E}_{1}\vert_{\mathscr{Y}^{+}}\rightarrow\mathscr{E}_{2}\vert_{\mathscr{Y}^{+}}\rightarrow\mathscr{E}_{3}\vert_{\mathscr{Y}^{+}}\rightarrow0$
in $\Bun_{\mathscr{Y}^{+}}$.
\item $0\rightarrow\mathscr{E}_{1}^{+}\rightarrow\mathscr{E}_{2}^{+}\rightarrow\mathscr{E}_{3}^{+}\rightarrow0$
in $\Bun_{\mathscr{Y}^{+}}$.
\item $0\rightarrow M_{1}\otimes\mathscr{R}_{+}^{\mathrm{int}}\rightarrow M_{2}\otimes\mathscr{R}_{+}^{\mathrm{int}}\rightarrow M_{3}\otimes\mathscr{R}_{+}^{\mathrm{int}}\rightarrow0$
in $\Mod_{\mathscr{R}_{+}^{\mathrm{int}}}$.
\item $0\rightarrow M_{1}\otimes\overline{B}\rightarrow M_{2}\otimes\overline{B}\rightarrow M_{3}\otimes\overline{B}\rightarrow0$
in $\Mod_{\overline{B}}$.
\end{enumerate}
\item Anyone of the following complexes is split exact.

\begin{enumerate}
\item $0\rightarrow\mathscr{E}_{1}^{+}\rightarrow\mathscr{E}_{2}^{+}\rightarrow\mathscr{E}_{3}^{+}\rightarrow0$
in $\Bun_{\mathscr{Y}^{+}}^{\varphi}$.
\item $0\rightarrow\mathscr{E}_{1}^{+}\vert_{\mathscr{Y}^{\circ}}\rightarrow\mathscr{E}_{2}^{+}\vert_{\mathscr{Y}^{\circ}}\rightarrow\mathscr{E}_{3}^{+}\vert_{\mathscr{Y}^{\circ}}\rightarrow0$
in $\Bun_{\mathscr{Y}^{\circ}}^{\varphi}$.
\item $0\rightarrow M_{1}\otimes\overline{B}\rightarrow M_{2}\otimes\overline{B}\rightarrow M_{3}\otimes\overline{B}\rightarrow0$
in $\Bun_{\overline{B}}^{\varphi}$.
\end{enumerate}
\end{enumerate}
\end{prop}
\begin{proof}
In $(1)$, plainly $(a)\Rightarrow(b)$, moreover $(b)\Rightarrow(c)$
by construction of $\mathscr{E}\mapsto\mathscr{E}^{+}$, $(c)\Rightarrow(d)$
by localization at $y_{L}$ and $(d)\Rightarrow(e)$ by base change
along $\mathscr{R}_{+}^{\mathrm{int}}\twoheadrightarrow\overline{B}$
(using that $M_{3}$ is free over $A$). Moreover, $(c)\Rightarrow(a)$
since $\mathscr{E}_{i}=\mathscr{E}_{i}^{-}$ on $\mathscr{Y}_{[-\infty,0[}$
and $\mathscr{E}_{i}=\mathscr{E}_{i}^{+}$ on $\mathscr{Y}_{]-1,+\infty]}$.
Since $\Bun_{\overline{B}}^{\varphi}\simeq\Bun_{\mathscr{Y}^{+}}^{\prime,\varphi}\simeq\Bun_{\mathscr{Y}^{\circ}}^{\varphi}$,
the three conditions of $(2)$ are equivalent. Obviously $(2a)\Rightarrow(1c)$,
and $(1e)\Rightarrow(2c)$ by \cite[\S 11.1]{FaFo15}. Condition~$(7)$
of proposition~\ref{prop:ConditionsExactAlgebraic} implies $(2c)$.
Finally $(1a)$ implies condition~$(5)$ of proposition~\ref{prop:ConditionsExactAlgebraic}
by the next proposition -- since indeed $M_{i}=\Gamma(\mathscr{Y},\mathscr{E}_{i})$
and $\mathscr{E}_{1}=M_{1}\otimes_{A}\mathscr{O}_{\mathscr{Y}}\simeq\mathscr{O}_{\mathscr{Y}}^{r_{1}}$
with $r_{1}=\rank_{A}M_{1}$. \end{proof}
\begin{prop}
We have $H^{1}(\mathscr{Y}^{+},\mathscr{O}_{\mathscr{Y}})=0=H^{1}(\mathscr{Y}^{-},\mathscr{O}_{\mathscr{Y}})$
and 
\[
H^{1}(\mathscr{Y},\mathscr{O}_{\mathscr{Y}})[{\textstyle \frac{1}{\pi}}]=H^{1}(\mathscr{Y},\mathscr{O}_{\mathscr{Y}})[{\textstyle \frac{1}{[\varpi]}}]=0.
\]
\end{prop}
\begin{proof}
The following proof was indicated to us by Fargues. First since 
\[
\mathscr{Y}^{-}=\cup_{s}\mathscr{Y}_{[-\infty,s]}\quad\mbox{and}\quad\mathscr{Y}^{+}=\cup_{r}\mathscr{Y}_{[r,+\infty]},
\]
we have exact sequences of $A$-modules 
\[
0\rightarrow R^{1}\underleftarrow{\lim}H^{0}\left(\mathscr{Y}_{[-\infty,s]},\mathscr{O}_{\mathscr{Y}}\right)\rightarrow H^{1}\left(\mathscr{Y}^{-},\mathscr{O}_{\mathscr{Y}}\right)\rightarrow\underleftarrow{\lim}H^{1}\left(\mathscr{Y}_{[-\infty,s]},\mathscr{O}_{\mathscr{Y}}\right)\rightarrow0
\]
\[
0\rightarrow R^{1}\underleftarrow{\lim}H^{0}\left(\mathscr{Y}_{[r,+\infty]},\mathscr{O}_{\mathscr{Y}}\right)\rightarrow H^{1}\left(\mathscr{Y}^{+},\mathscr{O}_{\mathscr{Y}}\right)\rightarrow\underleftarrow{\lim}H^{1}\left(\mathscr{Y}_{[r,+\infty]},\mathscr{O}_{\mathscr{Y}}\right)\rightarrow0
\]
By~\cite[2.7.7]{KeLi15}, $H^{1}\left(\mathscr{Y}_{[-\infty,s]},\mathscr{O}_{\mathscr{Y}}\right)=0=H^{1}\left(\mathscr{Y}_{[r,+\infty]},\mathscr{O}_{\mathscr{Y}}\right)$,
thus 
\[
\underleftarrow{\lim}H^{1}\left(\mathscr{Y}_{[-\infty,s]},\mathscr{O}_{\mathscr{Y}}\right)=0=\underleftarrow{\lim}H^{1}\left(\mathscr{Y}_{[r,+\infty]},\mathscr{O}_{\mathscr{Y}}\right).
\]
On the other hand the images of the restriction maps 
\[
H^{0}\left(\mathscr{Y}^{-},\mathscr{O}_{\mathscr{Y}}\right)\rightarrow H^{0}\left(\mathscr{Y}_{[-\infty,s]},\mathscr{O}_{\mathscr{Y}}\right)\quad\mbox{and}\quad H^{0}\left(\mathscr{Y}^{+},\mathscr{O}_{\mathscr{Y}}\right)\rightarrow H^{0}\left(\mathscr{Y}_{[r,+\infty]},\mathscr{O}_{\mathscr{Y}}\right)
\]
are dense in their complete codomain, thus by the Mittag-Leffler lemma,
\[
R^{1}\underleftarrow{\lim}H^{0}\left(\mathscr{Y}_{[-\infty,s]},\mathscr{O}_{\mathscr{Y}}\right)=0=R^{1}\underleftarrow{\lim}H^{0}\left(\mathscr{Y}_{[r,+\infty]},\mathscr{O}_{\mathscr{Y}}\right).
\]
Thus indeed $H^{1}(\mathscr{Y}^{+},\mathscr{O}_{\mathscr{Y}})=0=H^{1}(\mathscr{Y}^{-},\mathscr{O}_{\mathscr{Y}})$.
The Mayer-Vietoris sequence gives
\[
H^{1}\left(\mathscr{Y},\mathscr{O}_{\mathscr{Y}}\right)=\coker\left(H^{0}\left(\mathscr{Y}^{-},\mathscr{O}_{\mathscr{Y}}\right)\oplus H^{0}\left(\mathscr{Y}^{+},\mathscr{O}_{\mathscr{Y}}\right)\rightarrow H^{0}\left(\mathscr{Y}^{\circ},\mathscr{O}_{\mathscr{Y}}\right)\right).
\]
One checks that this cokernel is indeed annihilated by $-[\frac{1}{\pi}]$
and $-[\frac{1}{[\varpi]}]$.
\end{proof}

\subsubsection{~}

Returning to the general case, let 
\[
0\rightarrow(\mathcal{E}_{1,s},\mathcal{E}_{1,t},f_{1})\rightarrow(\mathcal{E}_{2,s},\mathcal{E}_{2,t},f_{2})\rightarrow(\mathcal{E}_{3,s},\mathcal{E}_{3,t},f_{3})\rightarrow0
\]
be the image of our canonical short exact sequence in $\Modif_{X}^{ad}$.
Then
\[
0\rightarrow\mathcal{E}_{1,s}\rightarrow\mathcal{E}_{2,s}\rightarrow\mathcal{E}_{3,s}\rightarrow0\quad\mbox{and}\quad0\rightarrow\mathcal{E}_{1,t}\rightarrow\mathcal{E}_{2,t}\rightarrow\mathcal{E}_{3,t}\rightarrow0
\]
are short exact sequences in $\Bun_{X}$, and the first one is even
split. 
\begin{prop}
\label{prop:CritExactOnModif}The conditions of proposition~\ref{prop:ConditionsExactAlgebraic}
are equivalent to: 

The exact sequence $0\rightarrow\mathcal{E}_{1,t}\rightarrow\mathcal{E}_{2,t}\rightarrow\mathcal{E}_{3,t}\rightarrow0$
is split.\end{prop}
\begin{proof}
Using the compatibility of $\underline{\mathcal{E}}:\Mod_{A,f}^{\varphi}\otimes E\rightarrow\Modif_{X}^{ad}$
with Tate twists, we may assume that our BKF-modules are anti-effective.
Our claim then follows from the criterion $(2.b)$ of proposition~\ref{prop:ConditionsExactnessAdic}
since $\mathcal{E}_{i,t}=(-/\varphi)^{alg}(\mathscr{E}_{i}^{+}\vert_{\mathscr{Y}^{\circ}})$
with the local notations, and $(-/\varphi)^{alg}:\Bun_{\mathscr{Y}^{\circ}}^{\varphi}\rightarrow\Bun_{X}$
is an exact equivalence of categories. 
\end{proof}

\subsubsection{Application}

Let $\Modif_{X}^{ad,\ast}$ be the strictly full subcategory of $\Modif_{X}^{ad}$
whose objects are the admissible modifications $(\mathcal{E}_{1},\mathcal{E}_{2},f_{\mathcal{E}})$
such that the $\mathbb{Q}$-filtration on $\mathcal{E}_{2}$ induced
by the Fargues $\mathbb{Q}$-filtration of $(\mathcal{E}_{1},\mathcal{E}_{2},f_{\mathcal{E}})$
is split. 
\begin{prop}
\label{prop:SpecialSubCats}Fix $(M,\varphi_{M})\in\Mod_{A,f}^{\varphi}\otimes E$
with image $(\mathcal{E}_{1},\mathcal{E}_{2},f_{\mathcal{E}})\in\Modif_{X}^{ad}$
and rank $r\in\mathbb{N}$. Then $t_{F,\infty}(M,\varphi_{M})(r)=t_{F}(\mathcal{E}_{1},\mathcal{E}_{2},f_{\mathcal{E}})(r)$
and for every $s\in[0,r]$, 
\begin{eqnarray*}
(M,\varphi_{M})\in\Mod_{A,f}^{\varphi,\ast}\otimes E & \Longrightarrow & t_{F,\infty}(M,\varphi_{M})(s)\leq t_{F}(\mathcal{E}_{1},\mathcal{E}_{2},f_{\mathcal{E}})(s),\\
(\mathcal{E}_{1},\mathcal{E}_{2},f_{\mathcal{E}})\in\Modif_{X}^{ad,\ast} & \Longrightarrow & t_{F}(\mathcal{E}_{1},\mathcal{E}_{2},f_{\mathcal{E}})(s)\leq t_{F,\infty}(M,\varphi_{M})(s).
\end{eqnarray*}
If both condition holds, then $\underline{\mathcal{E}}$ maps the
Fargues filtration $\mathcal{F}_{F}$ on $(M,\varphi_{M})$ (from
proposition~\ref{prop:TypeHN}) to the Fargues filtration $\mathcal{F}_{F}$
on $(\mathcal{E}_{1},\mathcal{E}_{2},f_{\mathcal{E}})$ (defined in
section~\ref{sub:DefFarguesFiltrOnModif}). \end{prop}
\begin{proof}
The first claim follows from~\ref{sub:CompNewton}. 

$(1)$ Suppose that $(M,\varphi_{M})$ belongs to $\Mod_{A,f}^{\varphi,\ast}$,
so that $t_{F,\infty}(M,\varphi_{M})=t_{F}(M,\varphi_{M})$ by proposition~\ref{prop:TypeHN}.
The graph of $t_{F}(M,\varphi_{M})$ (resp. $t_{F}(\mathcal{E}_{1},\mathcal{E}_{2},f_{\mathcal{E}})$)
is the concave upper bound of $\mathcal{A}$ (resp.~$\mathcal{B}$)
where 
\begin{eqnarray*}
\mathcal{A} & = & \left\{ (\rank,\deg)\left(\mathcal{F}_{F}^{\gamma}(M,\varphi_{M})\right):\gamma\in\mathbb{R}\right\} ,\\
\mathcal{B} & = & \left\{ (\rank,\deg)\left(\underline{\mathcal{E}}'\right):\underline{\mathcal{E}}'\mbox{ strict subobject of }(\mathcal{E}_{1},\mathcal{E}_{2},f_{\mathcal{E}})\mbox{ in }\Modif_{X}^{ad}\right\} .
\end{eqnarray*}
Now for every $\gamma\in\mathbb{R}$, the naively exact sequence 
\[
0\rightarrow\mathcal{F}_{F}^{\gamma}M\rightarrow M\rightarrow M/\mathcal{F}_{F}^{\gamma}M\rightarrow0
\]
in $\Mod_{A,f}^{\varphi}$ induces a canonically exact sequence 
\[
0\rightarrow\underline{\mathcal{E}}\left(\mathcal{F}_{F}^{\gamma}M\right)\rightarrow(\mathcal{E}_{1},\mathcal{E}_{2},f_{\mathcal{E}})\rightarrow\underline{\mathcal{E}}\left(M/\mathcal{F}_{F}^{\gamma}M\right)\rightarrow0
\]
in $\Modif_{X}^{ad}$. Thus $\mathcal{A}\subset\mathcal{B}$ and our
claim easily follows. 

$(2)$ Suppose that $(\mathcal{E}_{1},\mathcal{E}_{2},f_{\mathcal{E}})$
belongs to $\Modif_{X}^{ad,\ast}$. We need to show that for all $\gamma\in\mathbb{R}$,
$d\leq f(s)$ where $f=t_{F,\infty}(M,\varphi_{M})$ and $(s,d)=(\rank,\deg)\left(\mathcal{F}_{F}^{\gamma}(\mathcal{E}_{1},\mathcal{E}_{2},f_{\mathcal{E}})\right)$.
By assumption, propositions~\ref{prop:CritExactOnModif} and \ref{prop:DeftFinfty},
we may assume that the exact sequence
\[
0\rightarrow\mathcal{F}_{F}^{\gamma}(\mathcal{E}_{1},\mathcal{E}_{2},f_{\mathcal{E}})\rightarrow(\mathcal{E}_{1},\mathcal{E}_{2},f_{\mathcal{E}})\rightarrow\frac{(\mathcal{E}_{1},\mathcal{E}_{2},f_{\mathcal{E}})}{\mathcal{F}_{F}^{\gamma}(\mathcal{E}_{1},\mathcal{E}_{2},f_{\mathcal{E}})}\rightarrow0
\]
in $\Modif_{X}^{ad}$ arises from a naively exact sequence 
\[
0\rightarrow M'\rightarrow M\rightarrow M''\rightarrow0
\]
in $\Mod_{A,f}^{\varphi}$, which gives rise to exact sequences 
\[
0\rightarrow M'_{n}\rightarrow M_{n}\rightarrow M''_{n}\rightarrow0
\]
in $\Mod_{A,t}^{\varphi}$ for all $n\geq0$. Then by definition of
$f'=t_{F,\infty}(M')$ and $f=t_{F,\infty}(M)$, 
\[
d=\deg(M')=f'(s)=\lim_{n\rightarrow\infty}{\textstyle \frac{1}{n}}t_{F}(M'_{n})(ns)\leq\lim_{n\rightarrow\infty}{\textstyle \frac{1}{n}}t_{F}(M{}_{n})(ns)=f(s)
\]
using proposition~\ref{prop:ExactSeqInModt} for the middle inequality. 

$(3)$ Suppose now that both conditions hold. For $\gamma\in\mathbb{R}$,
consider the image of the (naively) exact sequence
\[
0\rightarrow\mathcal{F}_{F}^{\gamma}M\rightarrow M\rightarrow M/\mathcal{F}_{F}^{\gamma}M\rightarrow0
\]
from~proposition~\ref{prop:TypeHN}, which is an exact sequence
in $\Modif_{X}^{ad}$, 
\[
0\rightarrow\underline{\mathcal{E}}(\mathcal{F}_{F}^{\gamma}M)\rightarrow(\mathcal{E}_{1},\mathcal{E}_{2},f_{\mathcal{E}})\rightarrow\underline{\mathcal{E}}(M/\mathcal{F}_{F}^{\gamma}M)\rightarrow0.
\]
Set $(r_{\gamma},d_{\gamma})=(\rank,\deg)(\mathcal{F}_{F}^{\gamma}M)$,
so that $f(r_{\gamma})=d_{\gamma}$ where $f=t_{F}(M)$. By $(1)$
and $(2)$, we know that $f=t_{F}(\mathcal{E}_{1},\mathcal{E}_{2},f_{\mathcal{E}})$,
thus also $(r_{\gamma},d_{\gamma})=(\rank,\deg)(\mathcal{F}_{F}^{\gamma}(\mathcal{E}_{1},\mathcal{E}_{2},f_{\mathcal{E}}))$.
It then follows from proposition~\ref{prop:OnAdmModif_FarguesExactness}
that $\underline{\mathcal{E}}(\mathcal{F}_{F}^{\gamma}M)=\mathcal{F}_{F}^{\gamma}\underline{\mathcal{E}}(\mathcal{F}_{F}^{\gamma}M)$.
By functoriality of $\mathcal{F}_{F}$ on $\Modif_{X}^{ad}$, we find
that $\underline{\mathcal{E}}(\mathcal{F}_{F}^{\gamma}M)\hookrightarrow(\mathcal{E}_{1},\mathcal{E}_{2},f_{\mathcal{E}})$
induces a monomorphism $\underline{\mathcal{E}}(\mathcal{F}_{F}^{\gamma}M)\hookrightarrow\mathcal{F}_{F}^{\gamma}(\mathcal{E}_{1},\mathcal{E}_{2},f_{\mathcal{E}})$.
Since its domain and codomain have the same rank and degree, this
monomorphism is indeed an isomorphism. 
\end{proof}
\bibliographystyle{plain}
\bibliography{MyBib}

\begin{thebibliography}{10}

\bibitem{An09}
Yves Andr{\'e}.
\newblock {Slope filtrations}.
\newblock {\em Confluentes Math.}, 1(1):1--85, 2009.

\bibitem{BaMoSc16}
B.~{Bhatt}, M.~{Morrow}, and P.~{Scholze}.
\newblock {Integral p-adic Hodge theory}.
\newblock {\em ArXiv e-prints}, February 2016.

\bibitem{BoVdB03}
A.~Bondal and M.~van~den Bergh.
\newblock {Generators and representability of functors in commutative and
  noncommutative geometry}.
\newblock {\em Mosc. Math. J.}, 3(1):1--36, 258, 2003.

\bibitem{BoAC56}
N.~Bourbaki.
\newblock {\em {{\'E}l{\'e}ments de math{\'e}matique. {F}asc. {XXX}.
  {A}lg{\`e}bre commutative. {C}hapitre 5: {E}ntiers. {C}hapitre 6:
  {V}aluations}}.
\newblock {Actualit{\'e}s Scientifiques et Industrielles, No. 1308}. Hermann,
  Paris, 1964.

\bibitem{Co14}
C.~Cornut.
\newblock {Filtrations and Buildings. To appear in Memoirs of the AMS.}

\bibitem{Co16}
Christophe Cornut.
\newblock {On Harder-Narasimhan filtrations and their compatibility with tensor
  products. To appear in Confluentes Mathematici}.
\newblock 2016.

\bibitem{Fa12}
L.~Fargues.
\newblock {Th{\'e}orie de la r{\'e}duction pour les groupes p-divisibles,
  Preprint}.

\bibitem{Fa15}
Laurent Fargues.
\newblock {Quelques r{\'e}sultats et conjectures concernant la courbe}.
\newblock {\em Ast{\'e}risque}, (369):325--374, 2015.

\bibitem{Fa16}
Laurent Fargues.
\newblock {Geometrization of the local Langlands correspondence: an overview.
  Preprint}.
\newblock 2016.

\bibitem{FaFo15}
Laurent Fargues and Jean-Marc Fontaine.
\newblock {Courbes et fibr{\'e}s vectoriels en th{\'e}orie de Hodge $p$-adique,
  Preprint}.

\bibitem{Fo90}
Jean-Marc Fontaine.
\newblock {Repr{\'e}sentations {$p$}-adiques des corps locaux. {I}}.
\newblock In {\em {The {G}rothendieck {F}estschrift, {V}ol.\ {II}}}, volume~87
  of {\em {Progr. Math.}}, pages 249--309. Birkh{\"a}user Boston, Boston, MA,
  1990.

\bibitem{GaRa03}
Ofer Gabber and Lorenzo Ramero.
\newblock {\em {Almost ring theory}}, volume 1800 of {\em {Lecture Notes in
  Mathematics}}.
\newblock Springer-Verlag, Berlin, 2003.

\bibitem{SGA2r}
Alexander Grothendieck.
\newblock {\em {Cohomologie locale des faisceaux coh{\'e}rents et
  th{\'e}or{\`e}mes de {L}efschetz locaux et globaux ({SGA} 2)}}.
\newblock {Documents Math{\'e}matiques (Paris) [Mathematical Documents
  (Paris)], 4}. Soci{\'e}t{\'e} Math{\'e}matique de France, Paris, 2005.
\newblock S{\'e}minaire de G{\'e}om{\'e}trie Alg{\'e}brique du Bois Marie,
  1962, Augment{\'e} d'un expos{\'e} de Mich{\`e}le Raynaud. [With an
  expos{\'e} by Mich{\`e}le Raynaud], With a preface and edited by Yves Laszlo,
  Revised reprint of the 1968 French original.

\bibitem{HaMa07}
Tracy~Dawn Hamilton and Thomas Marley.
\newblock {Non-{N}oetherian {C}ohen-{M}acaulay rings}.
\newblock {\em J. Algebra}, 307(1):343--360, 2007.

\bibitem{Imp16}
Macarena~Peche Irissarry.
\newblock {The reduction of $G$-ordinary crystalline representations with
  $G$-structures. Preprint.}
\newblock 2016.

\bibitem{Ka66}
Irving Kaplansky.
\newblock {Elementary divisors and modules}.
\newblock {\em Trans. Amer. Math. Soc.}, 66:464--491, 1949.

\bibitem{Ke05}
Kiran~S. Kedlaya.
\newblock {Slope filtrations revisited}.
\newblock {\em Doc. Math.}, 10:447--525, 2005.

\bibitem{Ke16}
Kiran~S. Kedlaya.
\newblock {Noetherian properties of {F}argues-{F}ontaine curves}.
\newblock {\em Int. Math. Res. Not. IMRN}, (8):2544--2567, 2016.

\bibitem{Ke16b}
Kiran~S. Kedlaya.
\newblock {Some ring-theoretic properties of $A_{inf}$, preprint}.
\newblock 2016.

\bibitem{KeLi15}
Kiran~S. Kedlaya and Ruochuan Liu.
\newblock {Relative {$p$}-adic {H}odge theory: foundations}.
\newblock {\em Ast{\'e}risque}, (371):239, 2015.

\bibitem{LeWE16}
Brandon Levin and Carl Wang-Erickson.
\newblock {A Harder-Narasimhan theory for Kisin modules}.
\newblock 2016.

\bibitem{No76}
D.~G. Northcott.
\newblock {\em {Finite free resolutions}}.
\newblock Cambridge University Press, Cambridge-New York-Melbourne, 1976.
\newblock Cambridge Tracts in Mathematics, No. 71.

\bibitem{Sc17}
Peter Scholze.
\newblock {$p$-adic Geometry, preprint}.
\newblock 2017.

\bibitem{ScWe15}
Peter Scholze and Jared Weinstein.
\newblock {Peter Scholze{\rq}s lectures on p-adic geometry, Notes by J.
  Weinstein}.
\newblock 2017.

\end{thebibliography}

\end{document}